\documentclass[11pt,reqno]{amsart}
\usepackage[a4paper,margin=1.1in]{geometry}
\usepackage[T1]{fontenc}
\usepackage{lmodern,microtype,amsmath,amssymb,amsthm,mathtools,mathrsfs}
\microtypesetup{expansion=false}
\usepackage{enumitem,needspace,booktabs,longtable,array}
\usepackage{titlesec}
\usepackage{indentfirst}
\usepackage{xcolor}
\usepackage[all,cmtip]{xy}
\usepackage[colorlinks=true,linkcolor={red!50!black},citecolor={blue!50!black},urlcolor={blue!80!black}]{hyperref}
\hypersetup{pdftitle={Ramified quivers with potentials, monodromy characters, and cluster bases},pdfauthor={Jiarui Fei}}

\setlist[enumerate]{label=\textup{(\roman*)},leftmargin=25pt,itemsep=3pt,topsep=5pt}
\setlist[itemize]{leftmargin=18pt,itemsep=2pt,topsep=4pt}
\numberwithin{equation}{section}

\newtheorem{theorem}{Theorem}[section]
\newtheorem{proposition}[theorem]{Proposition}
\newtheorem{lemma}[theorem]{Lemma}
\newtheorem{corollary}[theorem]{Corollary}
\theoremstyle{definition}
\newtheorem{definition}[theorem]{Definition}
\newtheorem{hypothesis}[theorem]{Hypothesis}
\newtheorem{example}[theorem]{Example}
\theoremstyle{remark}
\newtheorem{remark}[theorem]{Remark}
\newtheorem*{maintheoremA}{Theorem A}
\newtheorem*{maintheoremB}{Theorem B}
\newtheorem*{maintheoremC}{Theorem C}
\newtheorem*{maintheoremD}{Theorem D}
\newtheorem*{maintheoremE}{Theorem E}

\newcommand{\op}[1]{\operatorname{#1}}
\newcommand{\kk}{\Bbbk}
\newcommand{\OO}{\mathcal O}
\newcommand{\RR}{\mathscr R}
\renewcommand{\AA}{\mathscr A}
\newcommand{\TT}{\mathscr T}
\newcommand{\PP}{\mathscr P}
\newcommand{\m}{\mathfrak m}

\newcommand{\rad}{\operatorname{rad}}
\newcommand{\Tr}{\operatorname{Tr}}
\newcommand{\Cyc}{\operatorname{Cyc}}
\newcommand{\Def}{\operatorname{Def}}
\newcommand{\Hom}{\operatorname{Hom}}

\newcommand{\PHom}{\operatorname{PHom}}
\newcommand{\PHOM}{\mathscr{P}\!\Hom}
\newcommand{\End}{\operatorname{End}}
\newcommand{\Aut}{\operatorname{Aut}}
\newcommand{\im}{\operatorname{im}}
\newcommand{\coker}{\operatorname{coker}}
\newcommand{\rank}{\operatorname{rank}}
\newcommand{\diag}{\operatorname{diag}}
\newcommand{\id}{\operatorname{id}}
\newcommand{\Spec}{\operatorname{Spec}}
\newcommand{\Frac}{\operatorname{Frac}}
\newcommand{\GL}{\operatorname{GL}}
\newcommand{\Gr}{\operatorname{Gr}}
\newcommand{\cyc}{\mathrm{cyc}}
\newcommand{\red}{\mathrm{red}}
\newcommand{\cl}[1]{\overline{#1}}
\newcommand{\C}{\mathbb C}
\newcommand{\Res}{\operatorname{Res}}
\newcommand{\chis}{\chi_\sigma}
\newcommand{\Gal}{\operatorname{Gal}}
\newcommand{\add}{\operatorname{add}}
\newcommand{\Newt}{\operatorname{Newt}}
\newcommand{\Z}{\mathbb Z}
\newcommand{\Q}{\mathbb Q}
\newcommand{\Good}{\mathscr G}
\newcommand{\Cand}{\mathscr C}
\newcommand{\Pot}{\operatorname{Pot}}
\newcommand{\lcm}{\operatorname{lcm}}
\newcommand{\Mid}{\operatorname{Mid}}
\newcommand{\Kram}{\mathbb K}
\newcommand{\Oram}{\mathbb O}
\newcommand{\Palg}{\mathcal P}
\newcommand{\CT}[2]{#1\langle\!\langle#2\rangle\!\rangle}

\title[Ramified QPs, monodromy, and cluster bases]{Ramified quivers with potentials,\\monodromy characters, and cluster bases I}
\author{Jiarui Fei}
\address{School of Mathematical Sciences, Shanghai Jiao Tong University, China}
\email{jiarui@sjtu.edu.cn}
\thanks{The author was supported in part by the National Natural Science Foundation of China (Nos.~12131015 and 12571038).}
\date{}

\titleformat{\section}{\normalfont\large\bfseries}{\thesection}{0.7em}{}
\titleformat{\subsection}{\normalfont\normalsize\bfseries}{\thesubsection}{0.65em}{}
\titlespacing{\section}{0pt}{3.0ex plus 0.8ex minus 0.2ex}{1.2ex}
\titlespacing{\subsection}{0pt}{2.1ex plus 0.6ex minus 0.2ex}{0.65ex}
\makeatletter
\def\@dotsep{4.5}
\def\@pnumwidth{2.2em}
\def\@tocrmarg{3.0em}
\renewcommand{\l@section}[2]{{\fontsize{10}{10.5}\selectfont\bfseries\@dottedtocline{1}{0em}{2.2em}{#1}{#2}}}
\renewcommand{\l@subsection}[2]{{\fontsize{10}{10.5}\selectfont\@dottedtocline{2}{2.2em}{3.2em}{#1}{#2}}}
\makeatother
\begin{document}
\begin{abstract}
We construct categorical models and cluster characters for
skew-symmetrizable cluster algebras admitting good ramified realizations.
Goodness is equivalent to nondegenerate realizability. When a reddening
sequence is given, it is enough to check sign coherence of the arrow
types along that sequence.
Triangular extensions with arbitrary connecting bimodule types preserve goodness 
when both component matrices admit reddening sequences. Consequently,
every matrix obtained from one vertex by triangular extensions and
mutations admits a good Jacobi-finite nondegenerate realization for every
compatible positive integral symmetrizer.

For a good Jacobi-finite nondegenerate realization, the monodromy
$F$-polynomials satisfy mutation for all finite decorated
representations, without a rigidity hypothesis. A full initial order
determines valued universal presentations and their cluster characters.
Under the seed reachability condition, these characters form a basis
of the middle cluster algebra for arbitrary geometric coefficients,
with frozen variables inverted, and contain all cluster monomials.
The same family is an upper-cluster-algebra basis when the middle and
upper algebras coincide. For a fixed Jacobian algebra and
$\delta$-vector, the Newton polytope of the generic monodromy
$F$-polynomial is independent of the order.
\end{abstract}
\maketitle
\begingroup\small\tableofcontents\endgroup

\section*{Introduction} 
\addcontentsline{toc}{section}{Introduction}
Quivers with potentials provide a representation-theoretic framework
for mutation of skew-symmetric cluster algebras, while cluster
characters express Laurent functions through the geometry of quiver
Grassmannians. Extending these constructions to skew-symmetrizable
matrices requires more than replacing arrow spaces by bimodules.

There are two distinct difficulties. The exchange matrix records only
the dimensions of the arrow bimodules, whereas their isomorphism types
govern quadratic cancellation under mutation. Moreover, after splitting
the vertex fields, geometric Euler characteristics of representation
Grassmannians need not give the coefficients required by cluster
mutation. We retain the actual bimodule structures throughout mutation
and use traces of tame monodromy for the coefficients.

Our principal results proceed from realizations to characters and
then to bases. Theorem A characterizes nondegenerate realizability;
Theorem C gives mutation-compatible monodromy characters and the
generality theorem for presentations; and Theorem D constructs bases
for arbitrary geometric coefficients under seed reachability.
Theorem B supplies the integral-order foundation, while Theorem E
identifies a broad class of exchange matrices to which the theory applies.

\subsection{Representation-theoretic background}
For ordinary quivers, Caldero--Chapoton characters express cluster
variables through Euler characteristics of quiver Grassmannians,
the varieties of subrepresentations of prescribed dimension
\cite{CC}. Palu extended this construction to triangulated
$2$-Calabi--Yau categories \cite{Palu}. Derksen, Weyman and Zelevinsky
introduced mutation of quivers with potentials and decorated
representations and proved the corresponding transformations of
$g$-vectors and $F$-polynomials \cite{DWZ1,DWZ2}. General presentations
and generic cluster characters extend the theory from rigid objects
to families and bases \cite{DF,Plamondon,GLFS}.

For skew-symmetrizable matrices, Rupel's quantum character treats
acyclic valued quivers over finite fields \cite{Rupel}. Demonet uses
a finite-group action on a Frobenius category with $2$-Calabi--Yau
stable category and identifies the character variables in each orbit
\cite[Definition 3.48]{Demonet}; Wu gives the analogous construction
for relative cluster characters on Higgs categories \cite[Section 6]{Wu}.
Their characters are defined on all objects, while the cluster models
use group-stable rigid collections and admissible orbit mutations.
Huang--Li construct unfoldings of acyclic sign-skew-symmetric matrices
by locally finite quivers which may have infinitely many vertices
\cite{HuangLi}.

Geiss, Leclerc and Schr\"oer introduced the algebras
$H(C,D,\Omega)$ and their locally free representation theory
\cite{GLS1}. They subsequently constructed a power-series order
$\widehat H$ whose special fibre is $H(C,D,\Omega)$ and whose
generic fibre is a hereditary species $\widetilde H$ \cite[Section 4]{GLS2}.
Reduction and localization match the isomorphism classes of rigid
locally free modules over $H$ and $\widehat H$ with those of rigid
$\widetilde H$-modules, and identify their support-tilting exchange
graphs \cite[Theorem 1.1]{GLS2}. Indecomposable rigid locally free
modules are parametrized by real Schur roots
\cite[Theorem 1.2]{GLS2}. The acyclic, zero-potential case of our
integral construction recovers this order for the standard GLS
integral arrow bimodules. Their locally free Caldero--Chapoton formulas appear in
\cite{GLS5}. Mou proves reflection formulas for these functions
\cite{MouReflections}, and Mou--Su construct generic bases in every
affine type \cite{MouSu}. In the orbifold construction of
Labardini-Fragoso--Mou, the $F$-polynomial mutation formula requires
both the source module and its mutation to be locally free. Their
Lemma~4.13 expresses this condition through the local image and kernel,
and Section~4.9 gives a locally free module whose mutation is not
locally free \cite{LabardiniMou}. In our construction the vertex algebras are fields,
so all the local kernels, images and quotients are vector spaces. The
character formula therefore applies to every finite decorated
representation of a ramified Jacobian algebra satisfying the stated
local mutation hypotheses.

Mutation of bimodule potentials is available in Nguefack's tensor-algebra
theory, in the strongly primitive species of Labardini-Fragoso--Zelevinsky,
and in the orbifold species of Geuenich--Labardini-Fragoso
\cite{Nguefack,LZ,GL}. Compared with these constructions, our
contributions address three further issues: recognizing nondegenerate
realizability from simple bimodule types, obtaining $F$-polynomial
mutation for all finite decorated representations under the local QP
hypotheses, and constructing bases under the seed reachability condition.
The one-vertex calculation in Section~\ref{intro:subsec:monodromy}
explains why tame monodromy is needed for the character formula.
For bases, we use Qin's principal-coefficient criterion \cite{Qin},
the theta algebra of Gross--Hacking--Keel--Kontsevich \cite{GHKK},
and the coefficient-specialization theorem of
Cheung--Magee--Mandel--Muller \cite{CMMM}.

\subsection{Vertex fields and the mutation problem}
Fix $n$ vertices, positive integers $d_i$, and compatible roots of $t$. Put
\[
\begin{gathered}
 \OO=\C[[t]],\qquad K=\C((t)),\qquad \varpi_i=t^{1/d_i},\\
 \OO_i=\C[[\varpi_i]],\qquad K_i=\C((\varpi_i)),\qquad
 \RR=\prod_i\OO_i,\qquad R=\prod_iK_i.
\end{gathered}
\]
Here $\varpi_i$ is a uniformizer of $\OO_i$, that is, a generator
of its maximal ideal. We use script letters for integral lattices and
orders. Unless a different scalar field is indicated, $\hom$ means
$\dim_K\Hom$. Section~\ref{alg:sec:setup} gives the integral
algebra conventions.

A \emph{modulation} is a finite-dimensional $R$-bimodule $A$. If $e_i$ is the identity
of the $i$th factor of $R$, then $A_{ij}=e_iAe_j$ is the arrow space
from $i$ to $j$, with its left $K_i$-action and right $K_j$-action.
We use right modules and multiply paths from left to right.
Isomorphisms of modulations fix $R$ pointwise. With $D=\diag(d_i)$,
the exchange matrix is
\[ b_{ij}=\frac{\dim_K A_{ij}-\dim_K A_{ji}}{d_j},\qquad BD=-(BD)^{\mathsf T}. \]
The modulation is \emph{two-acyclic} if there are no loops and no
arrows in both directions between a pair of vertices.

For an ordinary two-acyclic quiver, \emph{mutation at $k$} creates one
composite arrow for each path $i\to k\to j$, reverses the arrows
incident with $k$, and removes a maximal collection of opposite arrow
pairs. A \emph{source} has no incoming arrows and a \emph{sink} has no
outgoing arrows; mutation at either only reverses the incident arrows.
The matrix rule is
\begin{equation}\label{alg:eq:FZmutation}
 b'_{ij}=\begin{cases}
 -b_{ij},&i=k\text{ or }j=k,\\
 b_{ij}+[b_{ik}]_+[b_{kj}]_+-[-b_{ik}]_+[-b_{kj}]_+,&i,j\ne k.
 \end{cases}
\end{equation}
where $[a]_+=\max(a,0)$, applied componentwise to vectors.
For a \emph{mutation sequence} $w=(k_1,\ldots,k_r)$, we apply the
mutations in the listed order and write
$\mu_w=\mu_{k_r}\circ\cdots\circ\mu_{k_1}$.

For species, forming composite arrows means taking tensor products,
and cancellation depends on the simple bimodule types. Let
$L=K(t^{1/2})$, with $\tau(t^{1/2})=-t^{1/2}$, and consider
\[ \vcenter{\hbox{$\displaystyle \xymatrix@C=38pt{1\ar[r]^a&2\ar[r]^b&3}$}} \qquad (K_1,K_2,K_3)=(L,K,L). \]
Take $A_{12}={}_LL_K$ and $A_{23}={}_KL_L$. Mutation at vertex $2$
creates the composite
\[ L\otimes_KL\simeq L_{\id}\oplus L_\tau. \]
Here $L_{\id}$ has the usual scalar actions, whereas the right action
on $L_\tau$ is twisted by $\tau$. Replacing this composite by
$L_{\id}^{\oplus2}$ preserves the exchange matrix but changes the next
mutation. Thus we must follow the arrow bimodules themselves, rather
than choose new ones from each mutated matrix.

\subsection{Good modulations and integral realizations}
A \emph{potential} $S$ is a possibly infinite sum of cyclic paths,
taken modulo the arrow-adic closure of the span of differences of
cyclic rotations. The pair $(A,S)$ is a \emph{quiver with potential}
(QP). Its \emph{completed Jacobian algebra} is
\[ P=\mathcal P(A,S)=R\langle\!\langle A\rangle\!\rangle/J(S), \]
where $J(S)$ is the closed two-sided ideal generated by the cyclic
derivatives of $S$. A \emph{right-equivalence} is a continuous
isomorphism of completed tensor algebras that fixes $R$, preserves the
projection to tensor degree zero, and carries one potential to a
cyclically equivalent potential. The QP is \emph{Jacobi-finite} if
$\dim_KP<\infty$.

A \emph{ramified realization} of $(B,d)$ is a QP over these vertex
fields whose modulation gives $B$. Its \emph{premutation} at $k$
replaces passages $ab$ through $k$ by composites $[ab]$, reverses the
incident bimodules by duality, and adds the canonical cubic term
$\sum_{a,b}[ab]b^*a^*$, written after splitting the vertex fields.
Splitting off the quadratic arrow pairs gives the \emph{reduced part}
$\mu_k(A,S)$, whose potential has no quadratic terms.
Section~\ref{alg:sec:premutation} gives the tensor construction and
proves independence of the reduced right-equivalence class.

The simple $K_i$--$K_j$ bimodules are indexed by
$\Z/\gcd(d_i,d_j)$. We record their multiplicities with a positive sign
for $i\to j$ and a negative sign for the dual types in the reverse
direction. Tensor products give an explicit mutation rule for these
signed multiplicities (Section~\ref{sec:good-modulations}). A vector
is \emph{sign-coherent} if its entries are all nonnegative or all
nonpositive. A modulation is \emph{good} if these vectors remain
sign-coherent after every finite mutation sequence. A potential is
\emph{nondegenerate} if every sequence of reduced QP mutations remains
two-acyclic. After splitting the vertex fields, these mutations act on
entire vertex orbits of the split quiver.

\begin{maintheoremA}
For a two-acyclic ramified modulation, goodness is equivalent to
existence of a nondegenerate potential. Equivalently,
for every finite mutation sequence, nondegeneracy along that sequence
holds on a nonempty Zariski-open subset determined by finitely many
potential coefficients. The intersection of these subsets over all
finite sequences is nonempty.
A Jacobi-finite member exists in this locus exactly when the same good
modulation admits a Jacobi-finite potential after splitting.
\end{maintheoremA}

Theorem~\ref{gm:thm:realization} establishes this equivalence on the
full potential space. A \emph{reddening sequence} is one for which
mutation of $(B\mid-I_n)$ ends with nonnegative coefficient rows.
Theorem~\ref{gm:thm:reddening} makes goodness a one-sequence test:
the signed simple-bimodule multiplicities must remain sign-coherent
along that sequence, without assuming a finite mutation class.
Passing the test also gives a Jacobi-finite nondegenerate potential
(Corollary~\ref{gm:cor:reddening-finite}); no separate Jacobi-finite
witness is required.

Triangular extensions give recursive constructions and exact counts
(Theorem~\ref{gm:thm:triangular-extension}); in particular, every
two-acyclic modulation of an acyclic matrix is good. For finite
mutation classes without a reddening sequence,
Proposition~\ref{gm:thm:finitealgorithm} provides an algorithm.
Theorem~\ref{gm:thm:integral} gives coefficientwise $t$-adic density
of nondegenerate potentials on any prescribed arrow lattice free over
both vertex rings, including Jacobi-finite choices in the reddening case.

Valued universal presentations require a full initial order,
additional data beyond the ramified realization. A \emph{full order}
in $P$ is a unital $\OO$-subalgebra, finite over $\OO$ and spanning
$P$ over $K$. Fix an \emph{integral arrow lattice} $\AA\subset A$,
finite over $\OO$, spanning $A$, stable under the vertex orders, and
free over both vertex rings in each component. Put
\[ \TT=\RR\langle\!\langle\AA\rangle\!\rangle,\qquad \PP(S)=\TT/(\TT\cap J(S))\subseteq P. \]
The intersection is taken in the arrow-completed tensor algebra over
$K$. Since quadratic splitting need not split the integral lattice,
we transport the image order into the reduced Jacobian algebra.

\begin{maintheoremB}
Fix an integral arrow lattice $\AA\subset A$ as above, and suppose
$(A,S)$ is Jacobi-finite. The contracted algebra $\PP(S)$ is finite
and free over $\OO$, and the natural map
$\PP(S)\otimes_{\OO}K\to P$ is an isomorphism. Thus $\PP(S)$ is a full
$\OO$-order in $P$. For every $i,j$, the component $e_i\PP(S)e_j$ is
free as a left $\OO_i$-module and as a right $\OO_j$-module.

If two quadratic splittings over $K$ induce isomorphisms
$r_i:P\to P^{(i)}$, there is a right-equivalence between the reduced
QPs that induces $r_2r_1^{-1}$ and carries $r_1(\PP(S))$ onto
$r_2(\PP(S))$.
Jacobi-finiteness is preserved by premutation and reduced mutation.
\end{maintheoremB}

The three assertions are proved in
Theorem~\ref{alg:thm:contraction},
Proposition~\ref{alg:thm:nonsplit}, and
Corollary~\ref{alg:cor:mutation-finiteness}, respectively.
The contraction theorem also gives covariance under integral
right-equivalence. The later character construction allows any full
initial order containing $\RR$, not only a contracted order.

The \emph{special fibre} $\PP^{\mathrm{sp}}(S)=\PP(S)/t\PP(S)$
has vertex rings
$H_i=\OO_i/t\OO_i\simeq\C[\varepsilon_i]/(\varepsilon_i^{d_i})$,
the truncated polynomial rings of the GLS construction \cite{GLS1}.
Theorem B gives the component dimension identity
$\dim_\C e_i\PP^{\mathrm{sp}}(S)e_j =\dim_K e_i\mathcal P(A,S)e_j$.
For the standard GLS arrow lattices in the acyclic zero-potential
case, this recovers $\widehat H$, its hereditary generic fibre
$\widetilde H$, and the special fibre $H(C,D,\Omega)$
\cite{GLS2}. In general the special fibre need not be a GLS algebra,
nor the Jacobian algebra obtained by reducing cyclic derivatives
modulo $t$; see Section~\ref{alg:subsec:special-fibre}.

\subsection{Monodromy characters and mutation}
\label{intro:subsec:monodromy}
Even at a single ramified vertex, geometric Euler characteristic does
not give the required $F$-polynomial. Let $E=K(t^{1/2})$. The variety
of one-dimensional $E$-quotients of $E^2$, regarded over $K$, is
$\Res_{E/K}\mathbb P^1$. Here $\Res$ denotes \emph{restriction of
scalars}; over an algebraic closure $\overline K$, the variety becomes
$\mathbb P^1\times\mathbb P^1$. Its geometric Euler characteristic
is four, whereas a tame generator interchanges the two factors and
has alternating cohomological trace two. Thus Euler characteristic
would give $1+4y+y^2$ for $E^2$, while monodromy gives the polynomial
$(1+y)^2$ required by mutation.

Fix the \emph{tame generator} by
$\sigma(t^{1/m})=\exp(2\pi\sqrt{-1}/m)t^{1/m}$. For a variety $X$
over $K$, write
\[ \chi_\sigma(X)=\sum_j(-1)^j\Tr\bigl(\sigma\mid H_c^j(X_{\overline K},\Q_\ell)\bigr), \]
where $\ell$ is a prime and $H_c^j$ denotes compactly supported
$\ell$-adic \'etale cohomology. The \emph{sheets} above vertex $i$ are
the factors of $K_i\otimes_K\overline K\simeq\overline K^{d_i}$.
Writing $\Gr_e(V)$ for the Grassmannian of $e$-dimensional subspaces,
we obtain, for $0\le e\le r$,
\begin{equation}\label{intro:eq:normalization}
 \chi_\sigma\!\left(\Res_{K_i/K}\Gr_e(K_i^r)\right)=\binom re,
 \qquad
 \chi\!\left((\Gr_e(\overline K^r))^{d_i}\right)=\binom re^{d_i}.
\end{equation}
Geometric Euler characteristic counts independent choices on the
sheets. The cyclic trace counts the choices fixed by their permutation.

A \emph{decorated representation} $\mathcal M=(M,V)$ consists of a
finite right $P$-module $M$ and a finite right $R$-module $V$ with
no arrow actions; $(0,K_i)$ is the \emph{negative simple} at $i$.
At a mutation vertex $k$, the incoming action map
$\alpha:M_{\mathrm{in}}\to M_k$ enters Theorem C. The outgoing and
second-derivative maps $\beta,\gamma$, their relations, and the
kernel--image mutation formulas are given in Section~\ref{sec:local}.

Put $M_i=Me_i$ and $\mathbf r_M=(\dim_{K_i}M_i)_i$, and let
$\Gr^R_{\mathbf a}(M)$ be the \emph{representation Grassmannian} of
submodules of rank vector $\mathbf a$. We use the \emph{quotient-rank
convention}: the exponent records the ranks of the quotient, so
\[ F_M(\mathbf y)=\sum_{\mathbf e}\chi_\sigma\!\left(\Gr^R_{\mathbf r_M-\mathbf e}(M)\right)\mathbf y^{\mathbf e}. \]
The condition of \emph{balanced rank} requires quotient dimension $e_i$
on each sheet above $i$, not merely total dimension $d_ie_i$.
Formula~\eqref{intro:eq:normalization} is the basic normalization.
The polynomial uses the representation over $K$ and its monodromy,
not the geometric Euler characteristic of its special fibre.

Rank and weight vectors are rows; $\mathbf e_i$ denotes the standard
coordinate vector, distinct from the idempotent $e_i$. Put
$P(a)=\bigoplus_i(e_iP)^{a_i}$ for $a\geq0$. The \emph{presentation
weight} (or \emph{$\delta$-vector}) of $P_-\to P_+$ is
$\delta=[P_+]-[P_-]$. For presentations without a common projective
summand, the two multiplicity vectors are $[-\delta]_+$ and
$[\delta]_+$. A direct summand
$P(v)\xrightarrow{0}0$ contributes decoration of rank $v$.
Decoration and minimal presentations are treated in
Section~\ref{sec:rep-invariants}.

The two parts of the next theorem have different scopes: the first
concerns a single mutation under its local hypotheses; the second
uses a good global realization to treat general presentations.

\begin{maintheoremC}
\textup{(i) Character mutation.} For mutation of finite decorated representations
$\mathcal M\mapsto\mathcal M'$ at $k$ between ramified QPs that are
two-acyclic over $R$, set
$h_k=-\dim_{K_k}\coker\alpha$ and
$h'_k=-\dim_{K_k}\coker\alpha'$. Then
\[
F_{M'}(\mathbf y')
   =\frac{(1+(y_k')^{-1})^{h_k}}{(1+y_k')^{h'_k}}F_M(\mathbf Y),
 \qquad
 Y_i=
 \begin{cases}
 (y_k')^{-1},& i=k,\\
 y_i'(y_k')^{[b_{ik}]_+}(1+y_k')^{-b_{ik}},& i\ne k.
 \end{cases}
\]
\smallskip\noindent
\textup{(ii) General presentations.} For a chosen Jacobi-finite
nondegenerate realization of a good modulation, mutation takes
the decorated cokernels of general projective presentations to objects
with general presentations of tropical weight $\mu_k\delta$. On these families
$h_k=-[\delta_k]_+$ and $h'_k=-[-\delta_k]_+$, so the prefactor becomes
$(y_k')^{[\delta_k]_+}(1+y_k')^{-\delta_k}$.
\end{maintheoremC}

To prove character mutation, we fix the quotient away from the
mutation vertex and stratify its possible kernels after splitting.
Nonconstant sheet-rank tuples contribute zero monodromy trace;
relative Grassmannians on the constant-rank strata supply the binomial
factors. For general presentations, descent of invariant open sets
allows iteration and identifies the exponents with the parts of
$\delta$.

The coefficients are integral by the trace theorem of \cite{Nicaise}.
The \emph{Newton polytope} is the convex hull of exponents with nonzero
coefficient. Unique extremal kernels make its vertices integral and
their coefficients one. Theorem~\ref{newton:thm:order-independence}
then shows that, for fixed $(A,S,\delta)$, the monodromy $F$-polynomials selected by full initial orders
have the same Newton polytope: arithmetic variation changes only
nonvertex coefficients, possibly including boundary coefficients.
Opposite duality gives the injective formulation.

\subsection{The basis associated with an initial order}
Fix a good Jacobi-finite nondegenerate realization $(A,S)$ and a
full initial order $\PP_0\supseteq\RR$. For each presentation weight
$\delta$, projective Hom modules determine an integral lattice
$\PHOM(\delta)$ in the affine presentation space
\[ X_\delta=\PHom(\delta)=\Hom_P\!\left(P([-\delta]_+),P([\delta]_+)\right). \]
An integral basis $f_1,\ldots,f_q$ gives the universal presentation
$f_{\mathrm{univ}}=\sum_j z_jf_j$ and the function field
$K(X_\delta)=K(z_1,\ldots,z_q)$. The lattice determines the
\emph{Gauss valuation}: on polynomials, take the least $t$-adic
valuation of the coefficients, then extend to fractions. This
valuation does not depend on the chosen integral basis.

The decorated cokernel is formed over the fixed valued field
$L_\delta=\overline{\C(z_1,\ldots,z_q)}((t))$,
with its vertex extensions, throughout mutations from $\delta$.
Write $F_{\delta;\PP_0}$ for its monodromy $F$-polynomial.
Section~\ref{subsec:valued-families} proves that the represented
objects remain general under mutation. The choice of integral basis
is immaterial, but the initial order can affect the resulting character.

For an $n\times m$ extended exchange matrix $\widetilde B$ whose
principal submatrix is $B$, and an extended weight
$\widetilde\delta=(\delta,\delta_{\rm fr})$, define
\[ \mathsf G^{\PP_0}_{\widetilde\delta} =x^{-\widetilde\delta} F_{\delta;\PP_0}(x^{-\widetilde b_i})_{i=1}^n. \]
Here $\widetilde b_i$ is the $i$th row of $\widetilde B$ and
$x^u=\prod_i x_i^{u_i}$. The first $n$ variables are mutable, while
the others are frozen and inverted. The cluster algebra
$\mathcal A(\widetilde B)$ is generated by cluster variables over
the Laurent coefficient ring. The upper cluster algebra
$\mathcal U(\widetilde B)$ is the intersection of all seed Laurent
rings in their common field. The middle cluster algebra
$\Mid(\widetilde B)$ is spanned by the specialized polynomial
theta functions of \cite{GHKK}; these depend on the exchange
pattern. Section~\ref{sec:cluster-application} defines the
coefficient specialization and compares this span with the
monodromy characters.

The \emph{seed reachability condition} means that a mutation sequence in
the principal-coefficient pattern gives terminal mutable cluster
variables with initial $g$-vectors $-\mathbf e_i$, up to permutation.
It does not restrict the coefficient extension subsequently chosen.
Our convention and Qin's are related by
$\widehat B=-\widetilde B^{\mathsf T}$ and $g=-\widetilde\delta$.

\begin{maintheoremD}
Fix a good modulation, a Jacobi-finite nondegenerate potential, and a
full initial order $\PP_0\supseteq\RR$. Assume the seed reachability condition. For every geometric coefficient extension $\widetilde B$,
the family $\{\mathsf G^{\PP_0}_{\widetilde\delta}:\widetilde\delta\in\Z^m\}$
is a $\Z$-basis of $\Mid(\widetilde B)$ and contains all cluster
monomials. The members with zero initial frozen weight form a basis
over the Laurent coefficient ring. There are inclusions
\[ \mathcal A(\widetilde B)\subseteq\Mid(\widetilde B)\subseteq\mathcal U(\widetilde B). \]
The same family is a basis of the upper cluster algebra whenever the
middle and upper algebras coincide. This holds, in particular, if
$\rank\widetilde B=n$ or if
$\mathcal A(\widetilde B)=\mathcal U(\widetilde B)$.
\end{maintheoremD}

For principal coefficients, fixed-field transport gives pointed
Laurent characters, and Qin's criterion \cite[Theorem 4.3.1]{Qin}
gives an upper-algebra basis. The theta basis and its specialization
independence \cite[Theorem 3.10]{CMMM} then make the map from the
base-changed principal algebra into the target Laurent ring injective.
Its image is the middle algebra. Thus the order selects a basis of
an algebra determined by the exchange pattern.

The reachable $E$-rigid representations realize cluster monomials;
here $E$-rigidity means vanishing of the presentation invariant
$\op{E}(\mathcal M)$ (Section~\ref{sec:local}).
Corollary~\ref{cor:class-P-basis} applies the theorem to realizations
obtained by extensions and mutations. In particular, principal
coefficients give bases of their upper cluster algebras.

\subsection{Applications and geometric orders}
The construction extends beyond globally unfoldable patterns.
Example~\ref{ex:nonunfoldable-basis} treats the four-vertex matrix of
\cite[Example 14.4]{LZ} with parameters $(2,3)$. It excludes both
finite global unfoldings and locally finite unfoldings with infinite
orbits, hence also the finite-group orbit constructions of
\cite{Demonet,Wu}; nevertheless, a reddening sequence supplies a
good Jacobi-finite nondegenerate realization and its bases.
The checks use the same exchange convention. This unfolding
obstruction differs from the failure of a good modulation in
Example~\ref{ex:four-vertex-type-obstruction};
Remark~\ref{ex:four-vertex-duality} explains their duality.

The realizations are closed under a larger construction. A
\emph{triangular extension} joins two disjoint exchange matrices by
arrows all pointing from one part to the other. When one part is a
single vertex, this is a \emph{sink--source extension}. Write
$\mathfrak P$ for the class generated from the one-vertex zero matrix
by triangular extensions and mutations.

\begin{maintheoremE}
Every matrix in $\mathfrak P$ admits a good Jacobi-finite nondegenerate
ramified realization for each compatible positive integral symmetrizer.
For every such realization, every full initial order, and every
geometric coefficient extension, the monodromy characters form a
basis of the middle cluster algebra containing all cluster monomials.
\end{maintheoremE}

Theorem~\ref{gm:thm:triangular-extension} strengthens this: when
both components admit reddening sequences, the extension is good
exactly when both restrictions are good; the connecting types are
independent and counted by \eqref{gm:eq:triangular-count}.
Corollaries~\ref{gm:cor:class-P},~\ref{gm:cor:reddening-finite},
and~\ref{cor:class-P-basis} yield Theorem E, including the Cartan
string-diagram cases \cite{Cao}.

Multiplicativity for individual representations does not by itself
imply factorization for universal presentations initialized at
different weights. Theorem~\ref{fac:thm:criterion} characterizes
when the direct-sum morphism preserves the order-selected Gauss
valuations. For the integral tensor algebra of the degree-$(1,4)$
one-arrow species this gives the rank-two power basis of \cite{SZ}:
\[ F_{r(2,-1)}=((1+y_1)^2+y_1^2y_2)^r, \qquad \op{CM}(B)\sqcup\{Z^r:r\ge1\},\quad Z=\frac{(x_2+1)^2+x_1^4}{x_1^2x_2}. \]
Here $\op{CM}(B)$ is the set of coefficient-free cluster monomials,
including $1$. The remaining basis elements multiply along the
exceptional ray. Section~\ref{sec:nonzero} gives complementary
calculations with nonzero potentials and nonzero $\gamma$-maps,
including a nonrigid surface potential.

The orders in these examples also have a geometric interpretation.
For a ramified divisor $j:D\hookrightarrow\mathbb P^1_{\OO}$ given by
$z^d=t$, the algebra
$\End(\OO_{\mathbb P^1_{\OO}}\oplus j_*\OO_D)^{\mathrm{op}}$ is the
integral tensor algebra of the one-arrow species. Disjoint ramified divisors give star-shaped
species. Theorem~\ref{thm:divisor-order} proves that all full orders in
these $K$-algebras give the same cluster characters. For $d=4$, elementary
transformations along $D$ determine the element $Z$ above, and their
independent direct sums give its powers
(Example~\ref{cor:divisor-basis}).

\subsection{Organization}
Part~\ref{part:ramified} develops the realizations and presentations.
Sections~2--3 prove the integral-order and QP-mutation results,
including Theorem B. Section~4 gives the reddening and triangular-extension
tests; Section~5 constructs Jacobi-finite nondegenerate potentials
and integral approximations. Sections~6--7 establish decorated
mutation and fixed-field generality for presentations
(Corollary~\ref{thm:fixed-field-generality}).

Part~\ref{part:characters} develops monodromy characters and bases.
Section~\ref{sec:monodromy} proves Theorem C; Section~\ref{sec:cluster-application}
uses Newton polytopes and fixed-field transport for Theorem D.
Sections~\ref{sec:factorization}--\ref{sec:order-sensitivity}
cover arithmetic factorization, geometric orders, nonzero-potential
examples, and order-dependent transport. The last construction is
not required for the basis theorem. The auxiliary-choice comparison
is at the end of Section~\ref{sec:weights}; the appendices treat
decorated Auslander--Reiten translation and tropical Hom formulas.

\part{Ramified Realizations and General Presentations}
\label{part:ramified}\label{part:representations}

\section{Completed Jacobian algebras and integral orders}\label{alg:sec:setup}
We construct integral orders by mapping the integral tensor algebra
into the completed Jacobian algebra over $K$. When the latter is
finite-dimensional, only finitely many arrow degrees survive.
Theorem~\ref{alg:thm:contraction} uses this observation to prove that
the image is a full order, free over both vertex rings in each
component. Proposition~\ref{alg:prop:finite-algorithm} makes the
construction computable in a finite quotient. We begin with the
vertex rings and their arrow bimodules.

\subsection{Vertex rings and arrow bimodules}

Throughout this section, $\kk$ is a field of characteristic zero. Put
$\OO=\kk[[t]]$ and $K=\kk((t))$. The realization and monodromy
arguments will use $\kk=\C$.
For the finite vertex set $I$ and positive integers $d_i$, set
\[ \OO_i=\OO[\varpi_i]/(\varpi_i^{d_i}-t),\quad \RR=\prod_{i\in I}\OO_i,\qquad R=\RR\otimes_{\OO}K,\qquad K_i=\Frac(\OO_i). \]
The map $\OO_i\to\kk[[\varpi_i]]$ defined by $t\mapsto\varpi_i^{d_i}$
is an isomorphism. Thus $\OO_i$ is a \emph{discrete valuation ring},
with \emph{uniformizer} $\varpi_i$, a generator of its maximal ideal.
Its fraction field is $K_i=\kk((\varpi_i))$, a finite separable
extension of $K$.

Write $e_i$ for the identity of the $i$th factor of $\RR$, also viewed
as an idempotent of $R$. A bimodule is \emph{$\OO$-central} if its two
$\OO$-actions agree. A \emph{ramified arrow datum} is an $\OO$-central
$\RR$-bimodule $\AA$, finite over $\OO$, such that $e_i\AA e_j$ is
free as a left $\OO_i$-module and as a right $\OO_j$-module for every
$i,j$. Put $A=\AA\otimes_\OO K$. An arrow from $i$ to $j$ belongs
to $e_i\AA e_j$, and paths are multiplied from left to right. We assume
$e_i\AA e_i=0$ unless stated otherwise. Vertex parameters have arrow
degree zero. The two freeness conditions do not require freeness over
$\OO_i\otimes_\OO\OO_j$.

For a bimodule $M$ over an algebra $C$, write
\[ T_C(M)=\bigoplus_{r\geq0}M^{\otimes_Cr},\qquad C\langle\!\langle M\rangle\!\rangle=\prod_{r\geq0}M^{\otimes_Cr},\qquad M^{\otimes_C0}=C. \]
The first is its \emph{tensor algebra}; the second is its completion
by tensor degree, with the continuous multiplication induced by
concatenation. Section~\ref{alg:subsec:automorphisms} develops this
completion for integral arrow data.

For every positive integer $p$, put $\Oram_p=\kk[[t^{1/p}]]$ and $\Kram_p=\Frac(\Oram_p)$.
For positive integers $p,q$, write
\[ g=\gcd(p,q),\qquad a=p/g,\qquad b=q/g. \]
Write $x=t^{1/p}$, $y=t^{1/q}$, and $u=t^{1/g}$ for the compatible
roots. The standard Geiss--Leclerc--Schr\"oer (GLS) \emph{integral arrow
bimodule} is
\begin{equation}\label{alg:eq:GLSorder}
 M_{p,q}=\Oram_p\otimes_{\Oram_g}\Oram_q,
 \qquad u\longmapsto x^a,\quad u\longmapsto y^b.
\end{equation}
It is an $\Oram_p$--$\Oram_q$-bimodule and an $\OO$-lattice in
its generic fibre $\Kram_p\otimes_{\Kram_g}\Kram_q$.
This is the standard untwisted arrow bimodule of
\cite[Section 4.2]{GLS2} in integral form. Other simple types and
their sums will be needed for general modulations.

\begin{proposition}\label{alg:prop:GLSbasis}
There are isomorphisms
$M_{p,q}\simeq \OO[x,y]/(x^p-t,\ y^b-x^a) \simeq\kk[[s^a,s^b]]\subseteq\kk[[s]]$ with $s^{\op{lcm}(p,q)}=t$.
A left $\Oram_p$-basis is $1,y,\ldots,y^{b-1}$, and a right $\Oram_q$-basis is $1,x,\ldots,x^{a-1}$. The $\OO$-rank is $pq/g$.
\end{proposition}
\begin{proof}
The polynomial $y^b-x^a$ is monic in $y$, so division gives the left basis. The symmetric presentation gives the right basis. Since $\gcd(a,b)=1$, the parametrization $x=s^b,y=s^a$ identifies the complete plane-curve ring with $\kk[[s^a,s^b]]$. The latter need not be normal.
\end{proof}

The \emph{normalization}, or integral closure, of this curve ring is
$\kk[[s]]$. The integral arrow bimodule in \eqref{alg:eq:GLSorder} uses
$\kk[[s^a,s^b]]$, which need not be integrally closed.

Tensor products are taken over the indicated intermediate vertex ring.
Choosing bases successively over those rings shows that the product is
free over the vertex rings at its two ends.

Define the $\OO$-linear form
$\lambda_p\!\left(\sum_{r=0}^{p-1}c_r(t)x^r\right)=c_{p-1}(t)$.
Thus $\lambda_p$ extracts the coefficient of $x^{p-1}$. On the generic fibre it can equivalently be written as
$\lambda_p(z)=\op{Tr}_{\Kram_p/K}\!\left(\frac{z}{p x^{p-1}}\right)$.
For $p>1$, this is not the ordinary field trace. The pairing $(v,w)\mapsto\lambda_p(vw)$ is perfect over $\OO$, with dual bases
$1,x,\ldots,x^{p-1}$ and
$x^{p-1},\ldots,x,1$.

For $\Oram_g\subseteq\Oram_p$, define the relative form by extracting the coefficient of $x^{a-1}$ in the $\Oram_g$-basis. On $M_{p,q}$, the left evaluation extracts $y^{b-1}$ and the right evaluation extracts $x^{a-1}$. Composing with $\lambda_p$ or $\lambda_q$ gives the same $\OO$-linear form: both extract the coefficient of
$u^{g-1}x^{a-1}y^{b-1}$.
Such a linear form, whose multiplication pairing is perfect, is called a \emph{Frobenius form}. The relative forms above identify
the reversed standard module with both one-sided duals. Evaluation is
$\phi\otimes u\mapsto\phi(u)$; coevaluation sends $1$ to
$\sum_a u_a\otimes u_a^\vee$ for a basis and its dual. These
basis-independent maps commute with tensor products and coefficient
extension; they are the GLS adjunctions in integral coordinates
\cite[Section 5.1]{GLS1}.

These adjunctions use one-sided projectivity; they do not require the
integral vertex algebra $\RR$ to be separable. The \emph{enveloping algebra} $R\otimes_KR^{\mathrm{op}}$ acts on $K$-central $R$-bimodules: $(a\otimes b)$ acts by $m\mapsto amb$.

\begin{example}\label{run:ex:path}
Let $L=K(s)$, $s^2=t$, with $\tau(s)=-s$, and put
$\OO_L=\kk[[s]]$. Take the quiver
\[ \vcenter{\hbox{$\displaystyle \xymatrix@C=38pt{1\ar[r]^a&2\ar[r]^b&3}$}} \qquad (K_1,K_2,K_3)=(L,K,L), \]
with
$A_{12}={}_LL_K$ and $A_{23}={}_KL_L$.
Both arrows are standard GLS modules. The matrix is
\[
B_0=\begin{pmatrix}0&2&0\\-1&0&1\\0&-2&0\end{pmatrix},
 \qquad d=(2,1,2).
\]
Its finite path algebra contains the bimodule of length-two paths
\begin{equation}\label{run:eq:composite}
 A_{12}\otimes_K A_{23}=L\otimes_KL
       \simeq L_{\mathrm{id}}\oplus L_\tau,
 \qquad x\otimes y\longmapsto(xy,x\tau(y)).
\end{equation}
Here $L_\tau$ has its right action twisted by $\tau$. The two bimodules on the right are nonisomorphic. The length-two
bimodule becomes the composite arrow bimodule in
Example~\ref{run:ex:path-mutation}.

The corresponding integral composite is
\[ \mathscr D=\OO_L\otimes_\OO\OO_L =\OO[x,y]/(x^2-t,y^2-t). \]
The two idempotents after inverting $t$ involve
$\frac12(1\pm xy/t)$; they do not belong to $\mathscr D$.
Indeed $\mathscr D/t\mathscr D=\kk[x,y]/(x^2,y^2)$ is local, so
it cannot be the direct sum of the two component orders. Thus the
decomposition \eqref{run:eq:composite} over $K$ does not split this
integral tensor product.
\end{example}

\subsection{Completed tensor algebras and cyclic calculus}
\label{alg:subsec:automorphisms}\label{alg:subsec:potentials}
\label{alg:subsec:cyclic-calculus}\label{alg:subsec:separable-calculus}
\label{alg:subsec:cyclic-data}
As above, $R$ and $A$ denote the vertex algebra and arrow bimodule over
$K$; script letters denote their integral forms. We use the same symbol
$S$ for a cyclic class and a representative when the calculation is
independent of that choice. Paths are written in the reverse order to
those of Derksen--Weyman--Zelevinsky (DWZ).

Write $\AA^{(r)}=\AA^{\otimes_{\RR}r}$, with $\AA^{(0)}=\RR$. Set
\[ \TT=\prod_{r\ge0}\AA^{(r)},\qquad T_K=\prod_{r\ge0}\bigl(\AA^{(r)}\otimes_{\OO}K\bigr). \]
The positive-degree parts are the \emph{arrow ideals}, denoted by
$\m=\prod_{r\ge1}\AA^{(r)}$ and $\m_K$. In the
\emph{arrow-adic topology}, the powers $\m^N$ are neighbourhoods of zero; all closures below use this topology, rather than the $t$-adic one. Each tensor degree is finite free over $\OO$. This follows by successively using freeness over the intermediate vertex rings.

There is an injection $\TT[1/t]\hookrightarrow T_K$, but it need not be onto. An element of the first algebra has a common bound on its coefficient
denominators; an element of the second need not. The two algebras have
the same truncation at every fixed arrow length.

We will use the following facts over both the integral vertex algebra and
$R$. Let $T=\CT{C}{M}$, with arrow ideal $\m$, where $C$
is a fixed coefficient algebra. An \emph{augmented homomorphism} fixes $C$ and
preserves the augmentation $T\to C$, the projection to tensor degree zero.
For a finite-dimensional $R$-bimodule $M$ over $K$, write
$M^\vee=\Hom_K(M,K)$ with $(a\phi b)(m)=\phi(bma)$; this is its
dual with source and target exchanged. A choice of Frobenius forms on the vertex algebras identifies it with the one-sided duals in the tensor--Hom adjunctions.

\begin{proposition}\label{alg:prop:tensorautomorphism}
A continuous augmented homomorphism
$\varphi:\CT{C}{M}\to\CT{C}{M'}$ is determined by its
restriction
$\varphi|_M=\varphi^{(1)}+\varphi^{(\ge2)},\ \varphi^{(1)}:M\to M',\ \varphi^{(\ge2)}:M\to(\m')^2$,
where both maps are $C$-bimodule maps. Conversely, any such pair extends
uniquely. The homomorphism is an isomorphism if and only if
$\varphi^{(1)}$ is a bimodule isomorphism.
\end{proposition}
\begin{proof}
Substitution converges in every finite arrow truncation, and bimodule
linearity preserves the balancing relations. Density of ordinary tensors
gives uniqueness. An isomorphism induces an isomorphism on $\m/\m^2$.
Conversely, factor out the invertible linear part. If $\varphi(m)=m+h(m)$
with $h(M)\subseteq\m^2$, subtract the homogeneous degree-$n$ error
of an inverse modulo $\m^n$ to obtain one modulo $\m^{n+1}$.
These are bimodule corrections. Their limits give two-sided inverse
continuous maps.
\end{proof}

\begin{definition}\label{alg:def:unitriangular}
An augmented automorphism is \emph{unitriangular} if its linear part is
the identity. It has depth at least $d\ge1$ if
$\varphi(m)-m\in\m^{d+1}$ for every $m\in M$.
\end{definition}
For such an automorphism,
\begin{equation}\label{alg:eq:unitriangular-estimate}
 (\varphi-\id)(\m^r)\subseteq\m^{r+d}.
\end{equation}
Indeed, a changed monomial of length $r$ has at least one factor replaced
by a term of length at least $d+1$. It follows that a sequence of
compositions of unitriangular automorphisms converges whenever the depths
of the successive factors tend to infinity. The inverse compositions
converge as well. The filtration is by arrow degree. For example, $a\mapsto a+tb$,
with $b$ an arrow, changes the linear part and is not unitriangular.

We write the \emph{cyclic quotient} of a complete algebra as
$\Cyc(T)=T/\overline{[T,T]}$.
Here $[T,T]$ is the coefficient-linear span of the commutators
$xy-yx$, not the ideal they generate. Two elements are \emph{cyclically equivalent},
written $f\sim_{\cyc}g$, when they define the same cyclic class.
For an ordinary quiver, every nonclosed path is a commutator with an
idempotent, and closed paths represent the same class when one is obtained
from the other by cyclic rotation. Thus the positive-degree part of the cyclic quotient is the
product, over the arrow degrees, of the spaces spanned by cyclic words
modulo rotations.

A quiver with potential, or QP, over $R$ is a pair $(A,S)$,
where $A$ is a finite loop-free $R$-bimodule and $S$ is a cyclic
class with no terms of arrow degree zero or one. It is \emph{reduced} if
its quadratic part vanishes. Statements involving a representative of a
potential are understood modulo cyclic equivalence.

For a split quiver with vertex set $Q_0$, arrow set $Q_1$, and
$R=K^{Q_0}$, define the cyclic derivative in the arrow $a$ by
\begin{equation}\label{alg:eq:cyclic-derivative}
 \partial_a(a_1\cdots a_n)
 =\sum_{r:a_r=a}a_{r+1}\cdots a_n a_1\cdots a_{r-1}.
\end{equation}
Extend continuously and linearly. The formula vanishes on commutators,
so it depends only on the cyclic class. The closed two-sided ideal and
the Jacobian algebra are
\[ J(S)=\overline{\sum_{a\in Q_1}T(\partial_aS)T},\qquad \mathcal P(A,S)=R\langle\!\langle A\rangle\!\rangle/J(S). \]
A QP is Jacobi-finite if this algebra is finite-dimensional over
$K$. The ideal is generated by the complete derivatives: different
homogeneous terms in one derivative are not imposed as separate relations.

For a path $p$, put
$\Delta_a p=\sum_{p=uav}u\otimes v$, and set
$(u\otimes v)\mathbin{\#}g=vgu$. Finite truncations give the identities
\[
\begin{aligned}
 \partial_a(fg)&=(\Delta_a f)\mathbin{\#}g
                  +(\Delta_a g)\mathbin{\#}f,\\
 \partial_a\varphi(S)&=
 \sum_{b\in Q_1}(\Delta_a\varphi(b))\mathbin{\#}
                        \varphi(\partial_bS).
\end{aligned}
\]
For the first identity, group the occurrences of $a$ according to the
factor containing them. For the second, each occurrence lies in one
$\varphi(b)$; cutting there gives the derivative multiplied by the
remaining path factors. This proves both formulas for polynomials,
and continuity extends them to series. These are the DWZ cyclic
Leibniz and chain rules in our convention for multiplying arrows \cite{DWZ1}.

If $h:A\to\m$ is a bimodule map and $D_h$ the derivation extending
it and vanishing on $R$, then
\begin{equation}\label{alg:eq:firstvariation-basic}
 D_hS\sim_{\cyc}\sum_a h(a)\partial_aS.
\end{equation}
For a homogeneous potential of degree $r$, setting $h(a)=a$ gives
$\sum_a a\partial_aS\sim_{\cyc}rS$. Counting only occurrences from a
specified direct arrow summand gives the corresponding partial Euler
identity. All of these formulas concern derivatives in the arrow
directions; vertex parameters are coefficients.

For the vertex algebra $R=\prod_iK_i$, choose a finite Galois extension
$E/K$ splitting every $K_i$. Then
$R\otimes_KE$ is a product of copies of $E$, and
$E\otimes_KT_K$ is the completed path algebra on the corresponding
primitive idempotents and any bases of its arrow spaces. Tensor degree by
tensor degree, finite scalar extension commutes with this completion.
The Galois action is semilinear, meaning $g(\lambda v)=g(\lambda)g(v)$.
The extended potential is invariant under this action. The ideal
defined by \eqref{alg:eq:cyclic-derivative} is therefore Galois-stable;
we define
$J(S)$ over $K$ by descending that ideal. Equivalently, an element
belongs to $J(S)$ if its extension to $E$ belongs to the split ideal.

The chain rule for a change of arrow basis and its inverse identifies
the two derivative ideals. A common splitting extension proves
independence of $E$. Exactness of finite Galois invariants in characteristic
zero descends the cyclic identities as well. This gives the separable
species calculus of \cite{Nguefack} in the conventions used here.

A \emph{right-equivalence} $(A,S)\to(A',S')$ is an augmented continuous
$R$-algebra isomorphism carrying $S$ to $S'$ cyclically. The chain
rule and its inverse give
\[ \varphi(J(S))=J(\varphi(S)),\qquad \mathcal P(A,S)\simeq\mathcal P(A',S'). \]
An integral right-equivalence is a continuous augmented $\RR$-algebra isomorphism between the integral completed tensor algebras whose extension to the completed algebras over $K$ is a right-equivalence. The order
construction uses a potential over $K$ and an integral arrow lattice;
no integral cyclic-derivative calculus is needed. Any later use of
integral derivatives includes its compatibility hypothesis.

We define the {\em potential deformation space} by
\[ \Def(A,S)=\Cyc(\mathcal P(A,S))/\Cyc(R). \]
When the arrow bimodule is understood we write $\Def(S)$. A QP is
\emph{rigid} if this space is zero. For a loop-free quiver, the cyclic
image of $J(S)$ consists precisely of cyclic first variations
\eqref{alg:eq:firstvariation-basic}: rotate the path factors surrounding each derivative to the position of the varied arrow. The required correction has positive
length, since its two vertex idempotents are distinct. The closure issue
for formal series is proved in Lemma~\ref{alg:lem:closed-context}. This
identifies $\Def(S)$ with the quotient of potential directions by
infinitesimal changes of arrows.

\subsection{Finite contraction}\label{alg:sec:contraction}\label{alg:subsec:flat-definition}

Let $S$ be a specified potential on $T_K$, let $J_K=J(S)$, and put $\Palg_K=T_K/J_K$. The integral tensor algebra maps to $\Palg_K$. We define its \emph{contracted order} by
\[ \PP(S)=\im(\TT\longrightarrow \Palg_K) =\TT/(\TT\cap J_K). \]
A \emph{full lattice} in a finite $K$-space is a finite $\OO$-submodule
spanning it over $K$. A full order is a unital subalgebra which is a full
lattice. Over the discrete valuation ring (DVR) $\OO$, flatness is
equivalent to torsion-freeness; the contraction theorem below
proves that this image is indeed a full order.

The \emph{Jacobson radical} of an algebra $P$ is the intersection of
its maximal right ideals. We write it as $\rad P$. For a
finite-dimensional algebra it is nilpotent. For a quotient of a completed
tensor algebra over the semisimple vertex algebra $R$ by an ideal in
the arrow ideal, it is the image of that arrow ideal.

\begin{lemma}\label{alg:lem:toplayer}
Let $T=\CT{R}{M}$, with $R$ finite-dimensional semisimple, and let $J\subseteq\m$ be a closed ideal. If
$\m^n\subseteq J+\m^{n+1}$,
then $\m^n\subseteq J$. If $T/J$ is finite-dimensional, such an $n$ exists.
\end{lemma}
\begin{proof}
Multiplying this inclusion by $\m$ gives $\m^{n+1}\subseteq J+\m^{n+2}$. Substitution yields $\m^n\subseteq J+\m^{n+2}$, and iteration gives the same inclusion with any higher power. Closedness of $J$ proves the first assertion. In the completed tensor algebra, $1-xy$ is invertible whenever $y\in\m$, by geometric-series inversion. Thus the image of $\m$ lies in the Jacobson radical of every quotient. In a finite-dimensional quotient this radical is nilpotent.
\end{proof}

\begin{theorem}\label{alg:thm:contraction}
Suppose $\Palg_K$ is finite-dimensional. Then $\PP(S)$ is finite and free over $\OO$, and
$\PP(S)\otimes_{\OO}K\simeq \Palg_K$.
For every $i,j$, the bimodule $e_i\PP(S)e_j$ is free as a left $\OO_i$-module and as a right $\OO_j$-module. In particular,
\begin{equation}\label{alg:eq:component-ranks}
 d_i\rank_{\OO_i}(e_i\PP(S)e_j)
 =\dim_K e_i\Palg_Ke_j
 =d_j\rank_{\OO_j}(e_i\PP(S)e_j).
\end{equation}
Integral right-equivalences carry contracted ideals to contracted ideals
and induce the prescribed isomorphisms of orders.
\end{theorem}
\begin{proof}
By Lemma~\ref{alg:lem:toplayer}, $\m_K^N\subseteq J_K$ for some $N$. Hence $\m^N\subseteq\TT\cap J_K$, and $\PP(S)$ is a quotient of the finite free module $\TT/\m^N$. It embeds into $\Palg_K$, so it has no $t$-torsion. A finite torsion-free module over a DVR is free \cite[Tag 0AUW]{Stacks}. Its $K$-span is $\Palg_K$: every class is represented in the finitely many degrees below $N$, where a common coefficient denominator can be cleared.

Let $v\in e_i\PP(S)e_j$. If $\varpi_i^rv=0$, multiplying by a further power of $\varpi_i$ gives $t^qv=0$. Thus $e_i\PP(S)e_j$ is torsion-free over $\OO_i$; the same argument applies over $\OO_j$ on the right. Finiteness then gives freeness over both vertex rings. Equation~\eqref{alg:eq:component-ranks} follows by computing the $\OO$-rank. Finally, an integral right-equivalence $\phi$ extends to the completed algebras over $K$. The cyclic chain rule gives $\phi(J_K)=J'_K$, and $\phi(\TT)=\TT'$ then gives $\phi(\TT\cap J_K)=\TT'\cap J'_K$.
\end{proof}

\subsection{The special fibre of the contracted order}\label{alg:subsec:special-fibre}
Let $\varepsilon_i$ be the residue of $\varpi_i$. Put
$H_i=\OO_i/t\OO_i\simeq\kk[\varepsilon_i]/(\varepsilon_i^{d_i})$,
$H=\RR/t\RR$, $A_0=\AA/t\AA$, and $T_0=H\langle\!\langle A_0\rangle\!\rangle
\simeq\TT/t\TT$. For $\mathscr I=\TT\cap J(S)$ define
\begin{equation}\label{alg:eq:special-flat-definition}
 \PP^{\mathrm{sp}}(S)=\PP(S)/t\PP(S)=T_0/I_0(S),\qquad
 I_0(S)=(\mathscr I+t\TT)/t\TT.
\end{equation}
Thus
\[ \mathcal P(A,S)\ \longleftarrow\ \PP(S) \ \longrightarrow\ \PP^{\mathrm{sp}}(S), \]
with inclusion on the left and reduction on the right. The augmentation embeds the distinguished vertex algebra $H$ in the special fibre.

\begin{corollary}\label{alg:cor:special-fibre}
Under Theorem~\ref{alg:thm:contraction}, $\PP^{\mathrm{sp}}(S)$ is finite-dimensional.
Each $e_i\PP^{\mathrm{sp}}(S)e_j$ is free as a left $H_i$-module and as a right $H_j$-module, with
\[ \dim_\kk e_i\PP^{\mathrm{sp}}(S)e_j=\dim_K e_i\mathcal P(A,S)e_j. \]
The isomorphisms of orders induce isomorphisms of the special fibres
fixing $H$.
\end{corollary}
\begin{proof}
Reduce the bases of $e_i\PP(S)e_j$ from Theorem~\ref{alg:thm:contraction}
modulo $t$. Their dimensions over $\kk$ are the original $\OO$-ranks;
an integral isomorphism commutes with this reduction.
\end{proof}

For an acyclic arrow datum built from the standard GLS integral arrow
bimodules \eqref{alg:eq:GLSorder} (with the prescribed multiplicities)
and $S=0$, the contracted order is the power-series order
$\widehat H$ of \cite[Section 4.2 and Proposition 4.1]{GLS2}, up to
our arrow convention: its generic fibre is the hereditary species
$\widetilde H$ and its special fibre is the GLS algebra
$H(C,D,\Omega)$. Indeed, the acyclic tensor algebra is finite over
$\OO$, so contraction does not change it. With nonzero potential or
other integral arrow bimodules, no such identification is implicit.

Suppose $S$ is integral and an integral cyclic calculus supplies a
closed derivative ideal $J_{\OO}(S)$ whose extension to $K$ generates
$J(S)$ after arrow completion. If its image modulo $t$ is $J_H(S_0)$,
where $S_0=S\bmod t$, then $J_{\OO}(S)\subseteq\mathscr I$ gives
\begin{equation}\label{alg:eq:ordinary-to-flat}
 q_0:T_0/J_H(S_0)\twoheadrightarrow\PP^{\mathrm{sp}}(S),\qquad
 \ker q_0=I_0(S)/J_H(S_0).
\end{equation}
Example~\ref{alg:ex:completion} shows that this map need not be an
isomorphism. The definition \eqref{alg:eq:special-flat-definition}
does not require an integral cyclic calculus.

Under the finiteness hypothesis of Theorem~\ref{alg:thm:contraction},
the contracted order also has the following \emph{universal property}.
Every augmented quotient $q:\TT\twoheadrightarrow C$ which kills
$J_{\OO}(S)$, is $t$-torsion-free, and has nilpotent arrow ideal
factors uniquely through $\PP(S)$. Indeed, nilpotence extends $q$ to
$T_K\to C[1/t]$ using only finitely many arrow degrees. This extension
kills $J(S)$, so torsion-freeness makes $q$ kill $\TT\cap J(S)$.
The contracted order itself has these properties by
Theorem~\ref{alg:thm:contraction}.

\subsection{Contraction in a finite quotient}
Put $T_{\le n}=\TT/\m^{n+1}$. Choose integral multiples of cyclic derivatives spanning the derivative space over $K$. In this finite algebra, take all products of these
derivatives with paths on either side. Let $N_n\subseteq T_{\le n}$ be their $\OO$-span. Multiplying a derivative by a power of $t$ leaves its $K$-span unchanged. The \emph{$t$-saturation} of a submodule $N$ is
$N:t^\infty=\{x:t^a x\in N\text{ for some }a\ge0\}$. Thus
\begin{equation}\label{alg:eq:finite-saturation}
 L_n=T_{\le n}\cap(N_n\otimes_{\OO}K)=N_n:t^\infty
\end{equation}
is well defined independently of the chosen integral derivative generators.

\begin{proposition}\label{alg:prop:finite-algorithm}
If the degree-$n$ layer vanishes in $T_{\le n,K}/N_{n,K}$, then
$\TT\cap J_K=\pi_n^{-1}(L_n)$,
where $\pi_n:\TT\to T_{\le n}$. Such an integer $n$ exists whenever $\Palg_K$ is finite. As a sufficient numerical stopping condition, one may use
$\dim_K(T_{\le n,K}/N_{n,K})\le n$.
\end{proposition}
\begin{proof}
The first condition says $\m_K^n\subseteq J_K+\m_K^{n+1}$. Lemma~\ref{alg:lem:toplayer} gives $\m_K^n\subseteq J_K$, so
membership in $J_K$ is determined by this finite truncation. Intersecting with $T_{\le n}$ gives \eqref{alg:eq:finite-saturation}. If the truncated quotient has dimension $q\le n$, its nilpotent arrow ideal has $q$-th power zero. Hence $\m_K^q\subseteq J_K+\m_K^{n+1}\subseteq J_K+\m_K^{q+1}$, and the same lemma applies.
\end{proof}

We first form the submodule generated by all these products, and then take its saturation.
Normalizing individual derivative vectors need not give the same ideal.

\begin{example}[Completion and contraction]\label{alg:ex:completion}
Let $Q$ be the oriented triangle over $\OO$,
\[ \vcenter{\hbox{$\displaystyle \xymatrix@C=32pt@R=18pt{&2\ar[dr]^b&\\1\ar[ur]^a&&3\ar[ll]_c}$}},\qquad U=abc, \]
and consider $S_0=\tfrac12U^2$ and $S_1=\tfrac12U^2+tU$.
Their completed Jacobian ideals over $K$ are
\[ J_K(S_0)=\m_K^5,\qquad J_K(S_1)=\m_K^2. \]
Indeed, a derivative of $S_1$ is $bc(U+t)$, and $U+t$ is invertible
in the completed algebra over $K$.

The integral derivative quotient for $S_1$ is nevertheless free of
rank $15$ over $\OO$. If $p_{i,r}$ is the unique path of length $r$
starting at $i$, the relations $p_{i,5}=-t p_{i,2}$ give the basis
$p_{i,r}$, $0\le r\le4$. Reducing longer paths is determined by their
length, and formal path sums converge coefficientwise in this basis.
In particular, the class of $bc$ is nonzero and has no $t$-torsion, whereas
\[ bc\in\TT\cap J_K(S_1). \]
Consequently $\TT\cap J_K(S_1)\ne J_{\OO}(S_1):t^\infty$:
removing $t$-torsion does not replace completion over $K$.

Although $S_0\equiv S_1\pmod t$, their contracted special fibres are
\[ \PP^{\mathrm{sp}}(S_0)=\kk Q/\m^5,\qquad \PP^{\mathrm{sp}}(S_1)=\kk Q/\m^2, \]
of dimensions $15$ and $6$, respectively. For $S_1$, the map
\eqref{alg:eq:ordinary-to-flat} is the proper quotient
$\kk Q/\m^5\twoheadrightarrow\kk Q/\m^2$.
\end{example}

\subsection{Scalar extension}

\begin{proposition}\label{alg:thm:basechange}
Assume that $\mathcal P(A,S)$ is finite-dimensional. Let $\OO\to\OO'$ be a finite extension of complete characteristic-zero DVRs, with fraction fields $K\subseteq K'$. Extend scalars in the vertex algebra and arrow modules. Then
$\PP(S)\otimes_{\OO}\OO'\simeq\PP(S_{\OO'})$.
Reduced image orders and their prescribed comparison maps also commute with this base change.
\end{proposition}
\begin{proof}
Choose one arrow length at which the images of the arrow ideals in both
Jacobian algebras vanish. All maps then factor through finite tensor truncations. Flatness preserves the injection of $\PP(S)$ into the Jacobian algebra over $K$ and the surjection of the integral truncation onto its image. Their scalar extensions therefore have exactly the image defining the right-hand side. A splitting over $K$ carries the images of the same integral generators to their scalar extensions, proving the reduction assertion.
\end{proof}

The rings $\OO_i\otimes_{\OO}\OO'$ need not be normal or DVRs. The original free bases give freeness over the base-changed vertex
rings. Normalizing those rings would change the order.

\section{Splitting, integral orders, and QP mutation}\label{alg:sec:reduction}\label{alg:sec:nonsplit}
Quadratic reduction is performed over $K$, but the integral order must
be carried through the resulting Jacobian-algebra isomorphism. We first
make this comparison precise, then construct premutation and prove
that it preserves Jacobi-finiteness.

\subsection{Splitting and cancellation}\label{alg:subsec:generic-splitting}\label{alg:subsec:closed-context}
An abstract isomorphism between two reduced Jacobian algebras is not
enough: it must induce the prescribed identification of the image
orders. Theorem~\ref{alg:thm:cancel} provides such a right-equivalence,
and Proposition~\ref{alg:thm:nonsplit} applies it to the orders.
The linear splitting takes place over the semisimple algebra
$R\otimes_KR^{\mathrm{op}}$.

For a quadratic potential $S^{(2)}$, let $A_{\mathrm{triv}}$ be the
$R$-bimodule spanned by all its cyclic derivatives. A QP $(C,T)$ is
\emph{trivial} when $T$ is quadratic and its derivative image is all
of $C$; equivalently, $\mathcal P(C,T)=R$. Set
$A_{\red}=A/A_{\mathrm{triv}}$. We call these the trivial and reduced arrow bimodules, as in \cite{DWZ1}.

\begin{theorem}\label{alg:thm:generic-splitting}
Over the separable vertex algebra $R$, every QP is right-equivalent to
\[ (A_{\red},S_{\red})\oplus(A_{\mathrm{triv}},S_{\mathrm{triv}}), \]
where $S_{\red}$ is reduced and the second summand is trivial. Projection
to the reduced tensor algebra induces an isomorphism of Jacobian algebras.
If the Jacobian algebra is finite, the comparison of any two splittings can be
chosen to induce their prescribed isomorphism of Jacobian algebras.
\end{theorem}
\begin{proof}
Over a split vertex algebra, the quadratic terms between each unordered
pair of distinct vertices form a matrix pairing the two opposite arrow
spaces. Choose bases in the two spaces so that the pairing is
nondegenerate on a pair of complementary subspaces and vanishes on
the remaining subspaces. The nondegenerately paired subspaces form
the derivative image. The quadratic potential belongs to their tensor square and gives
a perfect pairing there.

For a separable $R$, its enveloping algebra
$R\otimes_KR^{\mathrm{op}}$ is semisimple. The derivative image therefore
has a bimodule complement. The preceding matrix statement can be checked
after a finite splitting field and descends: the quadratic potential is
supported on its derivative image, and its restricted duality is perfect.
The arrows in this image can be grouped into opposite pairs by the original
vertex idempotents. In particular a compatible coevaluation normalization
exists over $R$.

Choose a complement $A_0$ to the paired arrows $C$, and write
$S=Q+V+U$, where $Q$ is the perfect quadratic part, $V$ uses
only $A_0$, and every term of $U$ contains a paired arrow.
If $U$ has order $r\geq3$, decompose it by the number $q$ of
paired occurrences. The partial Euler identity, followed by the inverse
quadratic pairing, writes each part as a first variation $D_hQ$,
with $h(C)\subseteq\m^{r-1}$ and $h(A_0)=0$. Division by $q$
is allowed in characteristic zero. Substituting $c\mapsto c-h(c)$
cancels the mixed part to first order. The new mixed remainder has order
at least $2r-2$; pure terms are incorporated into $V$.
Iterating gives a convergent product of unitriangular automorphisms by
\eqref{alg:eq:unitriangular-estimate}.
The cyclic identities follow from Section~\ref{alg:subsec:cyclic-calculus};
a linear correction over a splitting field can be averaged before each
substitution to make it $R$-linear.
The paired derivatives of the resulting quadratic potential generate
exactly the paired-arrow ideal. Killing that ideal leaves $J(S_{\red})$,
which proves the assertion about Jacobian algebras. The assertion about the induced map of Jacobian algebras is proved in Theorem~\ref{alg:thm:cancel} below.
\end{proof}

The existence assertion of the splitting theorem
\cite{DWZ1,Nguefack} does not require Jacobi-finiteness.
The cancellation argument uses the following closed-image lemma.

\begin{lemma}\label{alg:lem:closed-image}
Give each finite-dimensional vector space over a field the discrete topology. Let $X$ be a closed linear subspace of a countable product of such spaces. Every continuous linear map from $X$ to another such product has closed image.
\end{lemma}
\begin{proof}
Suppose $y$ lies in the closure of the image. The solutions of the first
$n$ equations in the target coordinates form a nonempty closed affine subspace
$E_n$ of $X$, with $E_{n+1}\subseteq E_n$. For each fixed finite
set of source coordinates, the images of the $E_n$ form a descending
sequence of nonempty finite-dimensional affine spaces and eventually
stabilize. The stabilized images project onto one another: this follows
by applying the same descending-chain argument in each finite fibre.
Choose compatible points successively in these stabilized images. Their
limit belongs to every $E_n$, since each $E_n$ is closed, and maps to
$y$. Thus $y$ lies in the image.
\end{proof}

\begin{lemma}\label{alg:lem:closed-context}
Let $T=R\langle\!\langle A\rangle\!\rangle$ over a split vertex
algebra, let $I$ be a closed two-sided ideal, and let $J$ be the
closed ideal generated by finitely many elements $r_a$. The cyclic
image of the closed product $\overline{IJ}$ equals the cyclic image of
$\sum_a I r_a$. In particular, for $r_a=\partial_aS$, every cyclic
class in $\overline{IJ(S)}$ is a first variation of $S$ with
corrections belonging to $I$.
\end{lemma}
\begin{proof}
An algebraic product $g u r_a v$, with $g\in I$, is cyclically
equivalent to $vgu r_a$, and $vgu\in I$. Finite sums of these products are dense in $\overline{IJ}$. The linear map
$\bigoplus_a I\to\Cyc(T)$, $(h_a)\mapsto[\sum_a h_a r_a]$, has
closed image by Lemma~\ref{alg:lem:closed-image}; both spaces have the arrow
filtration with finite-dimensional quotients. Therefore taking the closure
adds no cyclic classes. Idempotents select the component of $h_a$
parallel to the arrow $a$, as required for a first variation.
Over a separable vertex algebra the same assertion, with its bimodule
conditions, follows by finite splitting and averaging the linear
correction maps.
\end{proof}

\begin{lemma}\label{alg:lem:square}
Let $A$ be a finite-dimensional loop-free $R$-bimodule, put
$T=\CT{R}{A}$, and let $V,V'$ be reduced potentials on $A$.
Suppose that $T/J(V)$ is finite-dimensional. If $V'-V$ is cyclically equivalent to an element of $J(V)^2$, then $J(V')=J(V)$. There is an $R$-unitriangular automorphism $\eta$ with
\[ \eta(V)\sim_{\cyc}V',\qquad \eta(b)-b\in J(V)\quad(b\in A). \]
Thus $\eta$ induces the identity on the Jacobian algebra.
\end{lemma}
\begin{proof}
For a split vertex algebra, this is \cite[Proposition 4.10]{DWZ1}.
Put $J=J(V)\subseteq\m^2$. The cyclic Leibniz rule gives
$\partial(J^2)\subseteq\overline{\m J+J\m}$, so $J(V')\subseteq J$.
Conversely,
$J\subseteq J(V')+\overline{\m J+J\m}$; iteration in the arrow
filtration and closedness give $J(V')=J$.

We construct the correction while preserving the ideal $J$. For $n\ge1$
put $I_n=J\cap\m^{n+1}$. Start with $W_1=V$ and maintain
\[ J(W_n)=J,\qquad V'-W_n\sim_{\cyc}E_n,\quad E_n\in\overline{I_nJ}. \]
The initial assertion holds because $I_1=J$. By
Lemma~\ref{alg:lem:closed-context}, choose $h_n(b)\in I_n$ such that
\[ V'-W_n\sim_{\cyc}\sum_b h_n(b)\partial_bW_n. \]
The substitution $\theta_n(b)=b+h_n(b)$ is unitriangular of depth
at least $n$ and induces the identity modulo $J$. In particular,
$\theta_n(J)=J$. Set $W_{n+1}=\theta_n(W_n)$. Since $W_n$ is reduced,
every nonlinear Taylor term contains at least two corrections and at
least one further arrow factor. Up to cyclic equivalence, the remainder
therefore lies in
$\overline{\m I_n^2}\subseteq\overline{I_{n+1}J}$.
This proves the induction step, including $J(W_{n+1})=J$ by the
chain rule. The product $\theta_n\cdots\theta_1$ converges in the
arrow filtration to a unitriangular automorphism $\eta$. Its error tends
to zero, and every finite product is the identity modulo the fixed closed
ideal $J$. Thus $\eta(V)\sim_{\cyc}V'$ and $\eta$ induces the identity
on $T/J$.

For general $R$, choose a finite Galois extension splitting its field
factors. At each stage the equation for $h_n$ is linear, and the ideal
and depth conditions $h_n(A)\subseteq I_n$ are Galois-stable. Average
a solution before making the substitution. The averaged solution is an
$R$-bimodule map with the same error and depth conditions. Repeating the
induction descends $\eta$. It is linear corrections that are averaged,
not automorphisms.
\end{proof}

\Needspace{8\baselineskip}
\begin{theorem}\label{alg:thm:cancel}
For $i=1,2$, let $(A_i,S_i)$ be reduced and let $(C_i,T_i)$ be
trivial. Suppose that $\Palg(A_i,S_i)$ is finite-dimensional for
$i=1,2$, and that
\[ \Phi:(A_1\oplus C_1,S_1+T_1)\xrightarrow{\sim}(A_2\oplus C_2,S_2+T_2) \]
is an $R$-right-equivalence. Projection away from the trivial summands
identifies these Jacobian algebras with $\Palg(A_i,S_i)$. Let
$\beta:\Palg(A_1,S_1)\to\Palg(A_2,S_2)$ be the map induced by
$\Phi$ under these identifications. Then a right-equivalence
$(A_1,S_1)\to(A_2,S_2)$ induces $\beta$.
\end{theorem}
\begin{proof}
Let $i_1$ include the first reduced tensor algebra and let $p_2$ kill the second trivial summand. Put $\psi=p_2\Phi i_1$. The derivatives of the quadratic part span $C_i$. Hence the linear part of $\Phi$ sends $C_1$ to $C_2$, and induces an isomorphism of quotient arrow modules. Thus $\psi$ has invertible linear part and is an isomorphism of completed tensor algebras.

Writing $\pi_i$ for the projections to the Jacobian algebras gives $\pi_2\psi=\beta\pi_1$, and hence $\psi(J(S_1))=J(S_2)$. Every trivial arrow belongs to the stabilized Jacobian ideal, so $p_2\Phi(C_1)\subseteq J(S_2)$. Projecting the equality of total potentials therefore gives
\[ S_2-\psi(S_1)\sim_{\cyc}p_2\Phi(T_1)\in J(S_2)^2. \]
By the chain rule,
$J(\psi(S_1))=\psi(J(S_1))=J(S_2)$. Apply
Lemma~\ref{alg:lem:square} with $V=\psi(S_1)$ and $V'=S_2$.
It gives a correction $\eta$ with $\eta\psi(S_1)\sim_{\cyc}S_2$
and inducing the identity on the target Jacobian algebra. Thus
$\eta\psi$ is a right-equivalence inducing the prescribed $\beta$.
\end{proof}

The same comparisons occur in \cite[Propositions 4.9--4.10 and 10.5]{DWZ1}.
The specified map of Jacobian algebras identifies the images of the integral order,
which is the assertion needed below.

\subsection{The image order under reduction}

Suppose $S$ is Jacobi-finite. Theorem~\ref{alg:thm:generic-splitting} gives a splitting over $K$
\[ F_i:(A,S)\longrightarrow(A_i,S_i)\oplus(C_i,T_i) . \]
Let $p_i$ kill $C_i$, put $\rho_i=p_iF_i$, and let $r_i:\Palg(S)\to\Palg(S_i)$ be the induced isomorphism. Define
\begin{equation}\label{alg:eq:reducedorder}
 L_i=r_i(\PP(S))
 =\im\bigl(\TT\xrightarrow{\rho_i}\CT{R}{A_i}\to\Palg(S_i)\bigr).
\end{equation}
A \emph{reduced Jacobian algebra with its integral order} is the triple $(A_i,S_i,L_i)$. Two such triples are equivalent if an $R$-right-equivalence of their potentials induces an isomorphism carrying the specified orders onto each other.

\begin{proposition}\label{alg:thm:nonsplit}
The triple $(A_i,S_i,L_i)$ is independent of the splitting over $K$ up to the preceding equivalence. More precisely, there is a right-equivalence $\Theta_{21}$ whose induced map of Jacobian algebras is
$\cl\Theta_{21}=r_2r_1^{-1}$,
and this map carries $L_1$ onto $L_2$. The maps between orders are integral $\RR$-algebra isomorphisms. They preserve the multiplication of the image order and are covariant under integral right-equivalence of the original ramified QP.
\end{proposition}
\begin{proof}
The kernel of the map in \eqref{alg:eq:reducedorder} is $\TT\cap J(S)$. Thus $r_i$ identifies $\PP(S)$ with $L_i$, including its multiplication. Theorem~\ref{alg:thm:contraction} gives finiteness and freeness over the vertex rings. Apply Theorem~\ref{alg:thm:cancel} to $F_2F_1^{-1}$. Its prescribed quotient map is $r_2r_1^{-1}$, proving the comparison. If $G$ is an integral right-equivalence between the original QPs,
compare a splitting of the source with a splitting of the target by
applying Theorem~\ref{alg:thm:cancel} to their composite with $G$.
The induced map carries the source order to the target order, and
therefore also identifies their reduced images.
\end{proof}

The right-equivalence may have poles on arrows, although its induced
map on the orders is integral. On the completed tensor algebras it gives
$\Theta_{21}\bigl(\rho_1(\TT)+J(S_1)\bigr) =\rho_2(\TT)+J(S_2)$.
Both sides are inverse images of the orders under the projections to the Jacobian algebras.

For three reductions, the maps $\beta_{ji}=r_jr_i^{-1}$ satisfy
\begin{equation}\label{alg:eq:reductioncocycle}
 \beta_{ii}=1,\qquad \beta_{ij}=\beta_{ji}^{-1},\qquad
 \beta_{kj}\beta_{ji}=\beta_{ki}.
\end{equation}
Two right-equivalences realizing the same $\beta_{ji}$ differ by one acting trivially on the Jacobian algebra. 
Reducing these order isomorphisms modulo $t$ gives compatible
identifications $L_i/tL_i\simeq \PP^{\mathrm{sp}}(S)$.

Removing trivial QP summands over $K$ in stages, in any grouping, is another splitting over $K$ of the original QP. Consequently successive reduction is equivalent to one-step reduction, with the comparisons \eqref{alg:eq:reductioncocycle}. 

Choose $N$ such that the $N$th power of the arrow ideal over $K$
vanishes in the Jacobian algebra. The images of integral source paths
of length less than $N$ span $L_i$ over $\OO$, so finitely many images
determine a basis of the order and its multiplication constants. If $\ell_1,\ldots,\ell_d$ is a basis of $\PP(S)$, the prescribed comparison $r_2r_1^{-1}$ sends $r_1(\ell_j)$ to $r_2(\ell_j)$. Thus the map on orders is determined by finitely many images, even
when the inducing right-equivalence is a formal series.

\paragraph{\em An integral quadratic pairing without a bimodule splitting.}
The tensor product in Example~\ref{run:ex:path} also gives a quadratic
pairing that splits over $K$ but not over the vertex orders.
\begin{example}\label{alg:ex:norm}\label{alg:subsec:norm}
Return to the tensor order $\mathscr D$ of Example~\ref{run:ex:path},
and put $H_2^{\wedge}=\OO[z]/(z^2-t)$.
The diagonal quotient and its Frobenius adjoint are
\begin{equation}\label{alg:eq:norm}
 q:\mathscr D\to H_2^{\wedge},\quad f(x,y)\mapsto f(z,z),
 \qquad
 j:H_2^{\wedge}\to \mathscr D,\quad h\mapsto(x+y)h.
\end{equation}
Here the source of $j$ has the diagonal action; the formula is well defined because $(x-y)(x+y)=0$. The image of $j$ is a nonzero proper ideal of the local ring $\mathscr D$. It cannot be a direct summand of the regular $\mathscr D$-module, since a summand ideal is generated by an idempotent. The quotient $q$ also has no bimodule section: its kernel is a nonzero
proper ideal of $\mathscr D$. Over $K$, it has the section
\[ s(1)=\tfrac12(1+xy/t),\qquad s(z)=\tfrac12(x+y). \]
Its pole is consistent with the nonsplitting over $\OO$.

These maps arise by premutating at vertex $2$ in
\[ \vcenter{\hbox{$\displaystyle \xymatrix@C=32pt@R=18pt{&2\ar[dr]^b&\\1\ar[ur]^a&&3\ar[ll]_c}$}} \qquad d=(2,1,2). \]
The arrows incident with $2$ are the standard bimodules,
and the arrow from $3$ to $1$ is the diagonal bimodule over the two
quadratic fields. The new composite is $\mathscr D$; the old opposite arrow is diagonal. In the coefficient-extraction Frobenius normalization, the quadratic part is
$xpc+pyc$,
with $p$ the full tensor composite and $c$ the diagonal opposite arrow, and its two derivative maps are \eqref{alg:eq:norm}. For $\alpha,\beta\in\kk$ with $\alpha\ne0$, the element $\theta=\alpha+\beta z$ is a unit. Multiplying the cubic coefficient by $\theta$ multiplies these maps by a unit, so neither map acquires an integral bimodule splitting. These cubic potentials span the cubic potential space over $K$; higher source terms do not change this quadratic part. After splitting the vertex fields, the source consists of two ordinary triangles sharing the scalar vertex. Nonzero triangle coefficients give a finite Jacobian algebra: the six paths of length two within the triangles vanish, the two crossing paths survive, and every path of length three vanishes. This realizes \eqref{alg:eq:norm} inside a Jacobi-finite QP.

\end{example}

\subsection{Premutation and the exchange matrix of the species}\label{alg:sec:mutation-invariants}\label{alg:sec:premutation}\label{alg:sec:families}
\label{alg:subsec:outside-invariant}\label{alg:subsec:reduced-mutation}
\label{alg:subsec:premutation-covariance}\label{alg:subsec:matrix-mutation}

Let $(A,S)$ be reduced. At a vertex $k$ with no incident oriented
two-cycle put 
\[ e=1-e_k,\qquad R_e=eRe,\qquad M=eAe_k,\qquad N=e_kAe,\qquad A^\circ=eAe. \]
Thus $M$ consists of incoming arrows, $N$ of outgoing arrows, and $A^\circ$ of arrows not incident with $k$. The composite module is $C=M\otimes_{K_k}N$. Put $U=M^\vee$ and $V=N^\vee$ for the reversed modules; the tensor--dual adjunction identifies $C^\vee$ with $V\otimes_{K_k}U$.
 The premutated arrow
bimodule is
$\widetilde A=A^\circ\oplus C\oplus U\oplus V$.
Replacing each passage through $k$ by its composite identifies
$e\CT{R}{A}e$ with $\CT{R_e}{A^\circ\oplus C}$; their arrow filtrations
are equivalent. Rotate cyclic representatives to a vertex other than
$k$, and write $w=[S]$ for the potential after this substitution. Put
$\widetilde S=w+\Delta_k$,
where $\Delta_k$ is the coevaluation pairing of composites and reverse
return paths. After splitting the vertex fields this is
$\sum_{a,b}[ab]b^*a^*$.
The separable tensor duality descends this expression to $K$.
Define $\mu_k(A,S)$ to be a reduced part of
$(\widetilde A,\widetilde S)$.

\begin{proposition}\label{alg:prop:generic-premutation-covariance}\label{alg:prop:double-premutation}
Premutation over $R$ respects right-equivalence. Reduced mutation
is an involution whenever the local no-two-cycle condition holds.
\end{proposition}
\begin{proof}
Work over a splitting field. Enumerate the incoming arrows by $a_i$
and the outgoing arrows by $b_j$. Factoring the last or first incident
arrow in the image under an augmented automorphism $\phi$ gives
$\phi(a_i)=\sum_j(C_\phi)_{ij}a_j$ and
$\phi(b_i)=\sum_jb_j(D_\phi)_{ji}$, where $C_\phi,D_\phi$ have entries in the
outside path algebra and invertible constant terms. In matrix notation,
these identities are $\phi(a)=C_\phi a$ and $\phi(b)=bD_\phi$. In the algebra with both reverse families
adjoined, extend $\phi$ by
$a^*\longmapsto a^*C_\phi^{-1},\ b^*\longmapsto D_\phi^{-1}b^*$.
The extension fixes the contractions $a^*a$ and $bb^*$. It preserves the
subalgebra generated by the outside arrows, the composites $ab$, and the
reversed arrows, as does its inverse. Its restriction carries the embedded
premutation potential $S+(bb^*)(a^*a)$ to the corresponding potential for
$\phi(S)$. These formulas use tensor products and duality and commute with Galois descent.
Cancellation proves independence of the reduced part.

For double premutation, write $\Gamma_{U,V}$ for the cyclic evaluation
term of the indicated dual arrow families, and put
$w(C,A^\circ)=[S]$. After identifying double
duals, the potential is cyclically equivalent to
\[ w(C,A^\circ)+\Gamma_{C,C^\vee}+\Gamma_{MN,C^\vee}. \]
Change the sign of one restored incident family and translate $C$ by $MN$.
The result is $S+\Gamma_{C,C^\vee}+R_C$, where every term of $R_C$ contains $C$
and none contains $C^\vee$. Cutting at a $C$-occurrence writes
$R_C\sim_{\cyc}\Gamma_{C,h}$. Subtract the corresponding components of $h$ from the arrows in
$C^\vee$. This removes $R_C$ and leaves the original QP together with
a trivial quadratic pair.
This proves involutivity over the vertex algebra
\cite[Sections 5--6]{DWZ1}.
\end{proof}

An $R$-bimodule $A$ is two-acyclic if
$e_iAe_i=0$ and, for distinct vertices $i,j$, at most one of
$e_iAe_j,e_jAe_i$ is nonzero. An $R$-QP is nondegenerate
if every successive reduced mutation is two-acyclic as an
$R$-bimodule. The vertices here are the field factors of $R$. Nondegeneracy is a
condition on a specified potential, whereas goodness in
Definition~\ref{gm:def:good} concerns the underlying modulation;
Theorem~\ref{gm:thm:realization} relates them.

\begin{proposition}\label{prop:intrinsic-matrix}
For a two-acyclic $R$-bimodule $A$, we put
\[ b_{ij}=\frac{\dim_K e_iAe_j-\dim_K e_jAe_i}{d_j}. \]
Then $B\diag(d_i)$ is skew-symmetric. Let
$B'=\mu_kB$ be the matrix defined by \eqref{alg:eq:FZmutation}.
If the reduced QP mutation at $k$ is two-acyclic over $R$, its
exchange matrix is $B'$. Thus a nondegenerate $R$-QP realizes the whole
matrix mutation pattern, without a separate choice of unfoldings.
\end{proposition}
\begin{proof}
Both dimensions in the numerator are divisible by $d_j$, using the
right action in the first summand and the left action in the second.
Reversal replaces each incident module by its dual and changes the
sign of the corresponding matrix entry. Away from $k$, the new
positive composite is
$e_iAe_k\otimes_{K_k}e_kAe_j$. Its right $K_j$-rank is
$[b_{ik}]_+[b_{kj}]_+$. The opposite composite contributes
$[-b_{ik}]_+[-b_{kj}]_+$ to the negative part. Consequently the
premutation signed dimensions satisfy the matrix mutation formula.
A trivial quadratic summand consists of two mutually dual arrow
modules. Their $K$-dimensions agree, so removing them leaves every
signed dimension unchanged. Two-acyclicity of the reduced result
then identifies its arrow ranks with the positive and negative
parts of the mutated matrix. Iterate.
\end{proof}

\begin{example}[Mutation depends on the simple arrow types]\label{run:ex:path-mutation}
Equip the quiver in Example~\ref{run:ex:path} with the zero potential.
Mutation at the scalar vertex $2$ makes its length-two bimodule into
an arrow bimodule, and the matrix becomes
\[
B_1=\mu_2B_0=
 \begin{pmatrix}0&-2&2\\1&0&-1\\-2&2&0\end{pmatrix}.
\]
The resulting simple arrow summands are
\[
\vcenter{\hbox{$\displaystyle \xymatrix@C=32pt@R=24pt{
 &2\ar[dl]_{a^*}&\\
 1\ar@/^/[rr]^{L_{\mathrm{id}}}\ar@/_/[rr]_{L_\tau}&&3\ar[ul]_{b^*}
}$}}.
\]
The bimodule on $1\to3$ is
$L_{\mathrm{id}}\oplus L_\tau$, whereas the standard GLS assignment
to $B_1$ uses $L_{\mathrm{id}}^{\oplus2}$. After splitting, the
target potential is the canonical cubic
\[ \widetilde S=\sum_{\alpha,\beta\in\{+,-\}} x_{\alpha\beta}b_\beta^*a_\alpha^*. \]
The actual composite has type vector $(1,1)$. Replacing it by the
standard GLS choice gives $(2,0)$ without changing $B_1$. On a second
mutation at $2$, the new opposite composite still has type vector
$(1,1)$, so the reassigned modulation leaves $(1,-1)$ after cancellation. Since a quadratic pair can
cancel only matching types, both original orientations survive for every
potential on the reassigned modulation. In contrast, double
mutation returns the original QP up to right-equivalence.
\end{example}

The distinction from geometric nondegeneracy is real. For example,
let all three vertex fields be $L=K(s)$, $s^2=t$, and take
arrows $a,b,c$ on
\[ \vcenter{\hbox{$\displaystyle \xymatrix@C=32pt@R=18pt{&2\ar[dr]^b&\\1\ar[ur]^a&&3\ar[ll]_c}$}}. \]
The first two arrows are diagonal and $s_3c=-cs_1$.
The potential $\frac12(abc)^2$ splits into the
ordinary potential of one oriented six-cycle. Its geometric
Jacobian algebra is $\overline KQ/\mathfrak m^5$, of dimension $30$,
and is rigid. Mutating the two sheets above vertex $3$ leaves a
diagonal $1\to2$ module and an antidiagonal $2\to1$ module.
There is no quadratic term to cancel either. In fact
$[s_1,s_1^{-1}abc]=2abc$, so a one-turn cycle has zero cyclic
class. Every potential on this modulation has arrow order at least
six; after bracketing at $3$ its order is at least four. Thus no
choice of coefficients on this modulation can supply the missing
quadratic cancellation. It admits no nondegenerate potential over
$R$. Thus geometric nondegeneracy and Jacobi-finiteness do not imply
nondegeneracy over $R$. The twisted arrow is essential to this example.

Section~\ref{sec:good-realization} uses goodness to choose a single
potential for which the preceding mutation construction is defined
throughout the mutation tree.

\subsection{Finiteness and deformation spaces}
Following DWZ, set $\overline e_k=1-e_k$ and write
$M_{\widehat k,\widehat k}=\overline e_kM\overline e_k$ for an
$R$-bimodule $M$. For an algebra $P$, the space
$P_{\widehat k,\widehat k}$ is a subalgebra with identity $\overline e_k$.
It contains paths whose two ends differ from $k$, including paths that
pass through $k$. The following result compares these subalgebras before
and after mutation.

\begin{proposition}\label{alg:prop:jacobian-subalgebra}
Let $(A,S)$ be reduced and let $k$ have no incident oriented two-cycle.
If $(A',S')=\mu_k(A,S)$, there is an isomorphism
\begin{equation}\label{alg:eq:outside-invariant}
 \mathcal P(A,S)_{\widehat k,\widehat k}
 \simeq\mathcal P(A',S')_{\widehat k,\widehat k}
\end{equation}
fixing $R_{\widehat k,\widehat k}$. The same assertion holds before
removing the trivial quadratic part.
\end{proposition}
\begin{proof}
Work first over a splitting field. Write incoming arrows as
$a:i\to k$, outgoing arrows as $b:k\to j$, and set
$c_{ab}=[ab]$. Let $w=[S]$ and $G_{ab}=\partial_{c_{ab}}w$.
The premutation potential is
$\widetilde S=w+\sum_{a,b}c_{ab}b^*a^*$.
Replacing each passage $ab$ through $k$ by $c_{ab}$ identifies
$\overline e_kT_K\overline e_k$ with the completed tensor algebra on
the arrows not incident with $k$ and the $c_{ab}$. In the premutated
algebra, paths beginning and ending away from $k$ also include the products $b^*a^*$.
The derivatives with respect to the composite arrows impose
\begin{equation}\label{alg:eq:return-elimination}
 b^*a^*=-G_{ab}.
\end{equation}
We will show that inclusion of the algebra on the old paths and the
substitution \eqref{alg:eq:return-elimination} induce inverse maps on
the two subalgebras of the Jacobian algebras.

The ideal $\overline e_kJ(S)\overline e_k$ is the closure of the
ideal generated by derivatives in arrows not incident with $k$ and by
all expressions
$a'\partial_aS$ and $(\partial_bS)b'$. Indeed, a product of an incident derivative with paths on either side,
whose initial and terminal vertices both differ from $k$, has an adjacent incoming or
outgoing arrow of this form. Replacing each passage through $k$ by its
composite gives
\[ \partial_f w,\qquad \sum_b c_{a'b}G_{ab},\qquad \sum_a G_{ab}c_{ab'}, \]
for all eligible indices. We likewise multiply reverse-arrow derivatives
of $\widetilde S$ by paths on either side so that the resulting products
have initial and terminal vertices different from $k$, and take the closure of their
span. After \eqref{alg:eq:return-elimination}, these products give the
same expressions, after renaming the summation indices. The composite derivatives
are killed by the substitution, and derivatives in the unchanged arrows agree. Hence the
two maps carry the complete closed ideals onto one another and are inverse
on the quotients. Projection away from a trivial quadratic pair preserves
the full Jacobian algebra by Theorem~\ref{alg:thm:generic-splitting}.

The maps use only tensor composition, duality and cyclic derivatives, so
they are compatible with the semilinear Galois action. They therefore
descend to $R$. This proves the species version of
\cite[Proposition 6.1]{DWZ1}.
\end{proof}

\begin{lemma}\label{alg:lem:subalgebra-finiteness}
Let $T=R\langle\!\langle A\rangle\!\rangle$, let $J\subseteq\m$
be a closed ideal, and suppose $e_kAe_k=0$. The quotient $T/J$ is
finite-dimensional if and only if $\overline e_k(T/J)\overline e_k$ is
finite-dimensional.
\end{lemma}
\begin{proof}
One implication follows because $\overline e_k(T/J)\overline e_k$ is a subspace of $T/J$. Every positive path beginning at
$k$ starts with an arrow leaving $k$, and every positive path ending
there ends with an arrow entering $k$. After removing these first and last arrows,
the remaining path has both ends outside $k$. There are only finitely
many incident arrow basis elements after splitting the vertex fields.
Put $P=T/J$. The spaces $e_kP\overline e_k$ and
$\overline e_kPe_k$ are images of finite direct sums of
$\overline e_kP\overline e_k$. The space $e_kPe_k$ is spanned by
$K_k$ and the subspaces $b(\overline e_kP\overline e_k)a$, as $a,b$
range over the finitely many incident arrows. These factorizations
also apply to formal series, by grouping them according to their first
and last incident arrows. Finite scalar extension detects finite
dimensionality, proving the assertion over $R$.
\end{proof}

\begin{corollary}\label{alg:cor:mutation-finiteness}
Under the hypotheses of Proposition~\ref{alg:prop:jacobian-subalgebra}, let
$(\widetilde A,\widetilde S)$ be the premutation and
$(A',S')$ a reduced part. Then
\[ \begin{split} \dim_K\mathcal P(A,S)<\infty &\ \Longleftrightarrow\ \dim_K\mathcal P(\widetilde A,\widetilde S)<\infty\\ &\ \Longleftrightarrow\ \dim_K\mathcal P(A',S')<\infty. \end{split} \]
\end{corollary}
\begin{proof}
Apply Lemma~\ref{alg:lem:subalgebra-finiteness} to the isomorphic
subalgebras in \eqref{alg:eq:outside-invariant} before reduction.
Removing the trivial quadratic part leaves the Jacobian algebra unchanged.
\end{proof}
The dimensions $\dim_K e_iPe_j$ with $i,j\ne k$ agree by
\eqref{alg:eq:outside-invariant}; the dimension of the whole Jacobian algebra can change.

\begin{lemma}\label{alg:lem:outside-trace}
For a loop-free QP with finite Jacobian algebra $\mathcal P=\mathcal P(A,S)$,
inclusion of $\mathcal P_{\widehat k,\widehat k}$ induces an isomorphism
\[ \Cyc(\mathcal P_{\widehat k,\widehat k})/ \Cyc(R_{\widehat k,\widehat k}) \xrightarrow{\sim}\Def(A,S). \]
\end{lemma}
\begin{proof}
Over a split vertex algebra, represent a cyclic class by closed paths.
Every positive closed path based at $k$ visits an outside vertex, since
there is no arrow loop at $k$. Rotate the path to such a vertex. Two
choices of this vertex differ by a commutator of paths whose ends are
outside $k$, so the resulting class in the cyclic quotient of the outside subalgebra is independent of the
choice.

This rotation also respects the Jacobian relations. By
Lemma~\ref{alg:lem:closed-context}, a cyclic relation is a sum of classes
$h_a\partial_aS$. At least one end of $\partial_aS$ is outside
$k$. Rotate each term to that end; it is then the cyclic class of an element
of the outside ideal $\overline e_kJ(S)\overline e_k$. This proves
that rotation descends to a map on $\Def(A,S)$, inverse to inclusion.
All classes can be evaluated after a finite arrow truncation, since the Jacobian algebra is finite-dimensional. Descent by
a finite splitting extension proves the general separable case.
\end{proof}

\begin{proposition}\label{alg:thm:def-invariance}
Under the hypotheses of Proposition~\ref{alg:prop:jacobian-subalgebra}, assume the
Jacobian algebra is finite-dimensional. Mutation induces an isomorphism
\[ \beta_k:\Def(A,S)\xrightarrow{\sim}\Def(A',S'). \]
It carries the distinguished potential class $[S]$ to $[S']$.
In particular, rigidity is invariant under mutation.
\end{proposition}
\begin{proof}
Use Lemma~\ref{alg:lem:outside-trace} on both sides of
\eqref{alg:eq:outside-invariant}. For the potential class, replacing passages through $k$ by composite
arrows carries $[S]$ to $[[S]]$. In the premutated Jacobian algebra, the coevaluation term
has zero class: multiply the derivative in each reversed incoming arrow
by that arrow and sum. Since $[S]$ contains no reversed
arrows, the sum is exactly $\Delta_k$, which belongs to the Jacobian
ideal. Therefore $[\widetilde S]=[[S]]$. A right-equivalence and
projection of a trivial quadratic pair preserve these classes. This
proves the claim and then the equivalence of vanishing deformation
spaces. The cyclic-quotient assertion is \cite[Proposition 6.9]{DWZ1}.
\end{proof}

\begin{corollary}\label{alg:cor:rigid-nondegenerate}
A rigid reduced Jacobi-finite ordinary QP is two-acyclic and is nondegenerate under all
ordinary vertex mutations.
\end{corollary}
\begin{proof}
If opposite arrows $a,b$ remain in a reduced QP, the nonzero cyclic
class of $ab$ has degree two. Every positive cyclic Jacobian relation has degree at least three:
a derivative has degree at least two, and the path that closes it has
positive length because there are no arrow loops.
Thus $[ab]$ is a nonzero deformation class, contradicting rigidity.
Apply this observation successively, using
Proposition~\ref{alg:thm:def-invariance}.
\end{proof}
The corollary concerns individual vertices of an ordinary quiver.
For species, goodness additionally requires compatible directions
between the orbits of split vertices; this is the subject of the next
section.

\section{Good ramified modulations}\label{sec:good-modulations}
This section isolates the obstruction carried by the arrow bimodules.
We first compute their simple types and tensor products, then formulate
goodness as sign coherence under type mutation. The final two
subsections turn this condition into a test along one reddening
sequence and a recursive construction by triangular extensions.

\subsection{Vertex fields and simple arrow bimodules}

Throughout this section let $K=\C((t))$, write
$\Kram_p=\C((t^{1/p}))$ and $\Oram_p=\C[[t^{1/p}]]$ with
compatible roots, and put $K_i=\Kram_{d_i}$ and $R=\prod_iK_i$. A modulation is a finite-dimensional
$K$-central $R$-bimodule $A$, with $e_iAe_i=0$. It is
\emph{two-acyclic at the original vertices} if, for $i\ne j$, at most
one of $e_iAe_j,e_jAe_i$ is nonzero. As in the introduction, isomorphisms fix $R$ pointwise. Changes of the
vertex-field identifications are considered separately in
Proposition~\ref{gm:prop:gauge}.

The matrix of a two-acyclic modulation is
\begin{equation}\label{gm:eq:Bconvention}
 b_{ij}=\frac{\dim_K e_iAe_j-\dim_K e_jAe_i}{d_j},\qquad
 B\diag(d_i)=-(B\diag(d_i))^{\mathsf T}.
\end{equation}
The matrix records total ranks. The coordinates below record the
individual simple bimodule types.

Let $E=\Kram_{\ell_0}$ with $\ell_0=\lcm(d_1,\ldots,d_n)$.
Index the factors of $E\otimes_KK_i\simeq E^{d_i}$, called
sheets, by $\alpha\in\Sigma_i=\Z/d_i$. A generator of $\Gal(E/K)$ simultaneously adds one to all sheet indices. Its orbits on the sheet pairs $(\alpha,\beta)$ over $i,j$ are indexed by
\[ s=\beta-\alpha\pmod {g_{ij}},\qquad g_{ij}=\gcd(d_i,d_j). \]

\begin{proposition}\label{gm:prop:types}
The finite-dimensional $K$-central $K_i$--$K_j$ bimodules form a semisimple category with $g_{ij}$ simple types $T_{ij}^{(s)}$, indexed by $s\in\Z/g_{ij}$. They satisfy
\[ \dim_K T_{ij}^{(s)}=\lcm(d_i,d_j),\qquad \rank_{K_j}T_{ij}^{(s)}=d_i/g_{ij},\qquad (T_{ij}^{(s)})^\vee\simeq T_{ji}^{(-s)}. \]
After extension to $E$, $T_{ij}^{(s)}$ has one arrow on every sheet pair with $\beta-\alpha=s\pmod {g_{ij}}$ and no other arrows.
\end{proposition}
\begin{proof}
The commutative separable algebra $K_i\otimes_KK_j$ is a product of $g_{ij}$ copies of $\Kram_{\lcm(d_i,d_j)}$. Its factors are the simultaneous Galois orbits just described. Simple modules over its factors give all simple bimodules. Reversing a sheet pair negates the difference of its indices. The dimension and rank assertions follow immediately. This is the cyclic-field case of \cite[Proposition 2.15]{GL}.
\end{proof}

A two-acyclic modulation is thus encoded by signed vectors
\[ m_{ij}=(m_{ij}(s))_{s\in\Z/g_{ij}},\qquad m_{ji}(s)=-m_{ij}(-s). \]
Positive entries count arrows $i\to j$; negative entries count the dual types of arrows $j\to i$. A collection of signed multiplicity vectors is \emph{admissible} if every $m_{ij}$ is entirely nonnegative or entirely nonpositive. This is \emph{sign coherence of the type multiplicities};
zero entries are allowed.

The split skew-symmetric matrix is
\[ C_{i\alpha,j\beta}=m_{ij}(\beta-\alpha),\qquad C_{i\alpha,i\beta}=0. \]
Formula \eqref{gm:eq:Bconvention} becomes
\begin{equation}\label{gm:eq:augmentation}
 b_{ij}=\frac{d_i}{g_{ij}}\sum_{s\in\Z/g_{ij}}m_{ij}(s).
\end{equation}
Thus $B$ records total ranks, whereas the multiplicity vectors record the twists.

\begin{proposition}\label{gm:prop:finitecandidates}
For fixed $(B,d)$ satisfying \eqref{gm:eq:Bconvention}, the isomorphism classes of two-acyclic modulations form a finite set $\Cand(B,d)$ of cardinality
\[ |\Cand(B,d)|=\prod_{i<j}\binom{N_{ij}+g_{ij}-1}{g_{ij}-1}, \qquad N_{ij}=\frac{|b_{ij}|g_{ij}}{d_i}. \]
\end{proposition}
\begin{proof}
Skew-symmetrizability makes $N_{ij}$ integral: writing $d_i=g_{ij}p$, $d_j=g_{ij}q$ with $\gcd(p,q)=1$ gives $p\mid b_{ij}$. On a nonzero pair, the sign is fixed by $B$, and the absolute multiplicities form a weak composition of $N_{ij}$ into $g_{ij}$ parts. A zero pair has the unique zero multiplicity vector. Semisimplicity shows that these multiplicities determine the isomorphism class.
\end{proof}

\subsection{Tensor products and mutation of type multiplicities}
For three distinct vertices $i,k,j$, put
$a=g_{ik}$, $b=g_{kj}$, $c=g_{ij}$, and
$h=\gcd(d_i,d_k,d_j)=\gcd(a,b)$.
For type vectors $x$ on $(i,k)$ and $y$ on $(k,j)$ define
\begin{equation}\label{gm:eq:convolution}
 (x\star_k y)(s)=\sum_{v=0}^{d_k-1}x(v\bmod a)\,y(s-v\bmod b),
 \qquad s\in\Z/c.
\end{equation}

\begin{proposition}\label{gm:prop:convolution}
For nonnegative multiplicity vectors, $x\star_k y$ is the multiplicity vector of the corresponding tensor product over $K_k$. Equivalently,
\begin{equation}\label{gm:eq:CRT}
 (x\star_k y)(s)=\frac{d_k}{\lcm(a,b)}
 \sum_{\substack{u\in\Z/a,\,v\in\Z/b\\u+v\equiv s\pmod h}}x(u)y(v).
\end{equation}
The expression is well defined modulo $c$, is bilinear, and is associative whenever three bimodule factors are composable.
\end{proposition}
\begin{proof}
After splitting the vertex fields, tensor composition is multiplication
of the matrices of sheet multiplicities. Between fixed outside sheets
$\alpha,\gamma$, the resulting entry is
\[ \sum_{\beta\in\Z/d_k}x(\beta-\alpha)y(\gamma-\beta). \]
Simultaneous translation of the sheet indices gives
\eqref{gm:eq:convolution}. For fixed $u\in\Z/a$ and $v\in\Z/b$,
the map $\beta\mapsto(\beta,s-\beta)$ has a fibre exactly when
$u+v\equiv s\pmod h$, and each nonempty fibre has
$d_k/\lcm(a,b)$ elements. This proves \eqref{gm:eq:CRT} and,
since $h\mid c$, well-definedness. Bilinearity and associativity are
those of matrix multiplication. The simple-type decomposition is
\cite[Proposition 2.12]{GL}.
\end{proof}

Write $x^+(s)=\max(x(s),0)$ and $x^-(s)=\max(-x(s),0)$. For an admissible collection, define the signed type mutation by
\begin{equation}\label{gm:eq:typemutation}
 m'_{ij}=\begin{cases}
 -m_{ij},& k\in\{i,j\},\\
 m_{ij}+m_{ik}^+\star_km_{kj}^+
          -m_{ik}^-\star_km_{kj}^-,&i,j\ne k.
 \end{cases}
\end{equation}

\begin{proposition}\label{gm:thm:typemutation}
For an admissible collection, \eqref{gm:eq:typemutation} agrees with
simultaneous ordinary matrix mutation at all sheets above $k$.
No arrows occur between sheets above the same vertex. Applying
\eqref{gm:eq:augmentation} to the resulting signed multiplicities gives
$\mu_kB$, whether or not these multiplicities are sign-coherent.
They define a two-acyclic modulation precisely when every type vector
is sign-coherent. Mutation is an involution on admissible collections
whose mutated collections are also admissible.
\end{proposition}
\begin{proof}
There are initially no arrows between sheets over $k$, so their mutations commute. For fixed sheets above $i$ and $j$, each intermediate sheet above $k$ contributes the product of the two arrow multiplicities. Summing over the $d_k$ intermediate sheets gives \eqref{gm:eq:typemutation}. An arrow between two sheets of $\Sigma_i$ would require arrows in both directions between $i$ and $k$ before mutation. Admissibility excludes this. Summing the signed type multiplicities as in \eqref{gm:eq:augmentation} counts all composites and cancels dual pairs. It therefore gives the matrix-mutation formula \eqref{alg:eq:FZmutation}. Equivalently, augment \eqref{gm:eq:convolution}; because the source multiplicity vectors have uniform signs, it gives the products $[b_{ik}]_+[b_{kj}]_+$ and $[-b_{ik}]_+[-b_{kj}]_+$. The split matrix uniquely determines the multiplicity vectors, so involutivity follows from the commuting ordinary involutions at the sheets.
\end{proof}

\begin{lemma}\label{gm:lem:obstruction}
Suppose $u_s$ arrows of type $T_{ij}^{(s)}$ face $v_s$ arrows of dual type in the reverse direction in a premutation. If quadratic reduction cancels $r_s$ pairs, then
$0\le r_s\le\min(u_s,v_s)$ and $(u_s-r_s)-(v_s-r_s)=u_s-v_s$.
If the vector $(u_s-v_s)_s$ has both signs, no potential can give a reduction that is two-acyclic over $R$. If a reduction is two-acyclic over $R$, cancellation in every matching type is necessarily maximal.
\end{lemma}
\begin{proof}
The quadratic derivative map preserves the decomposition into simple
bimodule types, since there are no maps between nonisomorphic types.
Its rank on type $s$ removes the same number of summands from the two
directions. Opposite strict signs force nonzero residual modules in opposite original directions, regardless of the ranks. Conversely any unmatched pair in the same type would already violate two-acyclicity of the species.
\end{proof}

Even for sign-coherent multiplicities, a particular potential may have
a quadratic derivative map of insufficient rank. A nondegenerate
potential cancels every possible opposite pair, but cannot make a
mixed-sign multiplicity vector sign-coherent. Ordinary split-quiver mutation
uses maximal cancellation on each sheet pair; hence it gives the same
criterion as typewise cancellation.

\subsection{Goodness and the mutation tree}
\begin{definition}\label{gm:def:good}
An admissible collection $m=(m_{ij})$ is \emph{good} if applying \eqref{gm:eq:typemutation} along any finite sequence of vertices gives an admissible collection after every step. Let $\Good(B,d)\subseteq\Cand(B,d)$ denote the good isomorphism classes.
\end{definition}

Equivalently, after every sequence of orbit mutations, the split quiver has no arrows between sheets above the same vertex, and all arrows between any two vertex orbits point in the same direction. An orbit mutation means the commuting mutations at the sheets above one vertex.

\begin{proposition}\label{gm:prop:bijections}
Every mutation induces a bijection
\[ \mu_k:\Good(B,d)\xrightarrow{\sim}\Good(\mu_kB,d). \]
In particular, the finite number $\nu(B,d)=|\Good(B,d)|$ is invariant
under matrix mutation with the vertex numbering fixed.
\end{proposition}
\begin{proof}
Any mutation sequence starting at $\mu_kA$ can be preceded by the mutation at $k$ to give a sequence starting at $A$. Thus the target is good. Involutivity in
Proposition~\ref{gm:thm:typemutation} supplies the inverse, and
Proposition~\ref{gm:prop:finitecandidates} gives finiteness.
\end{proof}

\begin{proposition}\label{gm:prop:gauge}
For changes of sheet indices at the vertices $q_i\in\Z/d_i$, the rule
$m_{ij}(s)\mapsto m_{ij}(s+q_j-q_i)$
commutes with type mutation and preserves goodness. Restriction to any subset of original vertices preserves goodness. Every two-vertex admissible modulation is good.
\end{proposition}
\begin{proof}
Changing the summation variable in \eqref{gm:eq:convolution} proves compatibility with the changes of sheet indices. Equivalently, these changes conjugate the split matrices by permutation matrices. For a sequence using only a fixed subset of vertices, the matrix entries between those vertices are the same whether we restrict before or after mutation. For two vertices, each mutation merely reverses the sole multiplicity vector, so no mixed signs can appear.
\end{proof}
The number of classes modulo changes of the vertex-field identifications
is also invariant under mutation. Two modulations related by such a change
need not be isomorphic as $R$-bimodules.

An initial good modulation determines its type multiplicities at every
node of the mutation tree. At each step we mutate this modulation,
rather than choose a new one from the resulting exchange matrix.

To disprove goodness, it suffices to find a mutation sequence producing a
multiplicity vector with both positive and negative entries. To prove goodness, it
suffices to include the initial modulation in a set of admissible
modulations closed under every mutation. A bounded search alone need
not certify goodness. Section~\ref{gm:sec:reddening} gives a sufficient
finite test along a specified reddening sequence.

\begin{proposition}\label{gm:thm:finitealgorithm}
Suppose the matrix mutation class with fixed vertex numbering $\mathcal M(B)$ is finite. All good modulations in that class can be classified by a finite computation.
\end{proposition}
\begin{proof}
Enumerate the finite set
$\mathscr V=\coprod_{B'\in\mathcal M(B)}\Cand(B',d)$.
Remove every modulation having an inadmissible mutation. Then repeatedly
remove each modulation whose mutation has already been removed. The process
terminates because $\mathscr V$ is finite. Every remaining modulation has
all its mutations in the remaining set and is therefore good. Following the recorded reason for each removal gives a finite sequence that makes the corresponding modulation inadmissible.
\end{proof}
Section~\ref{sec:good-realization} addresses the separate existence of a Jacobi-finite potential.

\subsection{Full tensor modulations and further examples}\label{gm:sec:uniform}
Constant multiplicity vectors give a simple source of good modulations.
\begin{proposition}\label{gm:thm:fulltensor}
If $d_i\mid b_{ij}$ for every pair, put $c_{ij}=b_{ij}/d_i$.
The modulation $A_{ij}=(K_i\otimes_KK_j)^{\oplus c_{ij}}$ for
$c_{ij}>0$, with no opposite arrows, is good.
\end{proposition}
\begin{proof}
Its type multiplicity vectors are constant, and
\[ (K_i\otimes_KK_k)\otimes_{K_k}(K_k\otimes_KK_j) \simeq(K_i\otimes_KK_j)^{\oplus d_k}. \]
Every entry in an outside target multiplicity vector is consequently
\[ c'_{ij}=c_{ij}+d_k\bigl([c_{ik}]_+[c_{kj}]_+ -[-c_{ik}]_+[-c_{kj}]_+\bigr), \]
while incident multiplicity vectors are negated. Constant multiplicity vectors cannot have
mixed signs, so induction proves the assertion.
\end{proof}
\begin{remark}\label{gm:thm:relative}
If $h\mid d_i$ for all $i$ and $d_i/h\mid b_{ij}$, replace $K$ by
$\Kram_h$ in the tensors and use $h b_{ij}/d_i$ copies. The same proof uses
$d_k/h$ in place of $d_k$; compatible vertex twists are allowed.
The full-tensor divisibility class appears in \cite[Corollary 3.6]{LA}, after transposing the matrix convention; the pairwise-coprime setting is treated in \cite{LZ}.
Theorem A also applies to modulations outside these tensor constructions.
\end{remark}

\begin{example}\label{ex:triangle-good-choice}
	Consider the exchange matrix obtained by reversing the arrows of the
	triangle in Example~\ref{run:ex:path-mutation}:
	\begin{equation}\label{gm:eq:cureB}
		d=(2,1,2),\qquad
		B=\begin{pmatrix}0&2&-2\\-1&0&1\\2&-2&0\end{pmatrix}.
	\end{equation}
	Let $L=\Kram_2$, with nontrivial automorphism $\tau$. The two edges
	incident with the scalar vertex have unique types, while the two
	reverse-arrow summands between the quadratic vertices can have
	multiplicity vectors $(2,0),\ (1,1)$, and $(0,2)$
	corresponding respectively to
	$L_{\mathrm{id}}^{\oplus2}$,
	$L_{\mathrm{id}}\oplus L_\tau$, and
	$L_\tau^{\oplus2}$.
	Mutation at vertex $2$ creates the composite
	$L\otimes_KL$, whose type vector is $(1,1)$. After cancellation the
	three possibilities therefore leave
	$(-1,1),\ (0,0)$, and $(1,-1)$.
	Hence only the middle choice is two-acyclic. It is the full-tensor
	choice, so Proposition~\ref{gm:thm:fulltensor} makes it good under
	every mutation sequence. Thus \eqref{gm:eq:cureB} has a unique good
	modulation. A concrete Jacobi-finite nondegenerate potential for it is
	recorded in Remark~\ref{gm:ex:cure-potential}.
\end{example}

\begin{example}\label{ex:four-vertex-type-obstruction}
The next matrix fails for every choice of simple arrow types. Take
\[
d=(6,2,6,3),\qquad
 B=\begin{pmatrix}
 0&3&0&-2\\-1&0&1&0\\0&-3&0&2\\1&0&-1&0
 \end{pmatrix}.
\]
This is the cyclic-species obstruction of
\cite[Examples 3.13 and 3.25]{GL}, written in our convention. Every
edge must be one simple bimodule. After mutation at $2$, the
$1\to3$ multiplicity vector is the indicator of one parity class in $\Z/6$,
with three entries. Mutation at $4$ subtracts the indicator of one
class modulo three, with two entries. The two classes intersect in
exactly one entry. Their difference therefore has two positive entries
and one negative entry, for every choice of the four twists. Thus
$\Good(B,d)$ is empty, and Proposition~\ref{gm:prop:bijections} excludes
its entire matrix mutation class for these degrees. The dual pair
will instead have a good realization and a character basis;
see Example~\ref{ex:nonunfoldable-basis} and
Remark~\ref{ex:four-vertex-duality}.

\end{example}

\begin{remark}\label{gm:rem:matrix-return}
Let $L=K(t^{1/2})$, with nontrivial automorphism $\tau$, and take
\[ d=(2,2,1),\qquad B=\begin{pmatrix}0&1&0\\-1&0&2\\0&-1&0\end{pmatrix}. \]
There are two isomorphism classes over $R=L\times L\times K$:
$A_{12}=L_{\mathrm{id}}$ or $L_\tau$, with
$A_{23}={}_LL_K$ and $A_{13}=0$. Both are good by
Corollary~\ref{gm:cor:class-P} and the acyclic observation following it. The mutation sequence
$w=(1,3,2,3,2)$ returns $B$ to itself but interchanges the two classes.
Indeed, formula~\eqref{gm:eq:typemutation} sends the triple
$(m_{12},m_{13},m_{23})=((1,0),0,1)$ to $((0,1),0,1)$.
The two resulting modulations become isomorphic after changing one
vertex-field identification, but not by an $R$-bimodule isomorphism.
Thus we follow the actual modulation along a mutation sequence, even
when its exchange matrix returns.
\end{remark}

\subsection{A test along a reddening sequence}\label{gm:sec:reddening}
Start matrix mutation with principal coefficients $(B\mid-I_n)$.
A vertex is \emph{green} when its coefficient row is nonpositive and
\emph{red} when that row is nonnegative. The rows are nonzero and
sign-coherent. A \emph{reddening sequence} is a finite mutation sequence
whose final coefficient rows are all nonnegative. Intermediate red
mutations are allowed; if every mutation is green, the sequence is
\emph{maximal green}. These conventions agree with successive sink
mutations for an acyclic matrix \cite[Section 2]{KellerGreen}.

\begin{theorem}\label{gm:thm:reddening}
Let $w$ be a reddening sequence of $B$. A two-acyclic ramified modulation
realizing $(B,d)$ is good if and only if its signed type multiplicities
remain sign-coherent after every prefix of $w$.
\end{theorem}
The sequence tests the actual bimodules by
\eqref{gm:eq:typemutation}, not only the exchange matrices. For
sufficiency, we compare the split and folded scattering diagrams.
The admissible lift of the prescribed reddening sequence identifies
their total wall-crossing elements. Positivity then ensures that
averaging cannot erase a nontrivial local factor. A first failure of
sign coherence along any other sequence would contradict the resulting
comparison of chamber interiors and walls. This argument does not
assume a potential on the ramified modulation.

\paragraph{Averaging the split data.}
Let $I_i=\Z/d_i$, $I=\coprod_iI_i$, and let $B^{\mathrm{spl}}$ be the
split matrix of the initial modulation. The cyclic group $\Gamma$
translates the sheets in every $I_i$. On the lattices
$N^{\mathrm{spl}}=\Z^I$ and $N=\Z^n$, define
\[ q(e_a)=e_i\quad(a\in I_i),\qquad s(e_i)=\frac1{d_i}\sum_{a\in I_i}e_a. \]
The map $s$ identifies $N_{\Q}$ with $(N^{\mathrm{spl}}_{\Q})^\Gamma$.
Put $\omega^{\mathrm{spl}}(e_a,e_b)=B^{\mathrm{spl}}_{ab}$ and
$\omega(u,v)=\omega^{\mathrm{spl}}(s(u),s(v))$. Then
\begin{equation}\label{gm:eq:scattering-form}
 \omega(e_i,e_j)=\frac{b_{ij}}{d_i}=\frac{H_{ij}}{d_j},
 \qquad H=D^{-1}BD=-B^{\mathsf T}.
\end{equation}
Here $H_{ij}=\sum_{b\in I_j}B^{\mathrm{spl}}_{ab}$ for $a\in I_i$:
$H$ uses row sums, whereas $B$ uses column sums. If $C_v$ denotes the
coefficient matrix after a sequence $v$ from $(B\mid-I_n)$, then the
coefficient matrix for $H$ is $D^{-1}C_vD$, by the mutation formula.
Thus the same sequence is reddening for $B$ and $H$.

For a based lattice with skew form $\omega$, use the complete graded
Lie algebra and its exponential group
\[ \widehat{\mathfrak g}_{\omega}=\prod_{0\ne u\in\Z_{\ge0}^n}\Q X_u, \qquad [X_u,X_v]=\omega(u,v)X_{u+v},\qquad G_{\omega}=\exp\widehat{\mathfrak g}_{\omega}. \]
The completion is by total degree, and multiplication in $G_\omega$
is given by the Baker--Campbell--Hausdorff formula. Use superscript
$\mathrm{spl}$ for the split versions. All positive degrees, including
central ones, are kept.

\begin{lemma}\label{gm:lem:scattering-average}
The degree-preserving coefficient-sum map
\begin{equation}\label{gm:eq:scattering-average}
 \Phi:(\widehat{\mathfrak g}_{\omega^{\mathrm{spl}}})^\Gamma
 \longrightarrow\widehat{\mathfrak g}_{\omega},\qquad
 \sum_u a_uX_u\longmapsto\sum_u a_uX_{q(u)}
\end{equation}
is a continuous Lie homomorphism.
\end{lemma}
\begin{proof}
Each nonnegative degree has finitely many lifts under $q$.
For an invariant sum $F=\sum_{q(u)=v}a_uX_u$, let $a=\sum a_u$.
The vector $\sum a_uu$ is invariant and projects to $av$, so it equals
$as(v)$. For another such sum of projected degree $v'$ and coefficient
sum $a'$, the coefficient of $X_{v+v'}$ in the projected bracket is
\[ \sum_{u,u'}a_ua'_{u'}\omega^{\mathrm{spl}}(u,u') =\omega^{\mathrm{spl}}(as(v),a's(v'))=aa'\omega(v,v'). \]
This proves the assertion degree by degree. Total degree is preserved,
so the map extends continuously to the completions.
\end{proof}
The map \eqref{gm:eq:scattering-average} induces a group homomorphism,
also denoted by $\Phi$. Its domain consists of invariant elements;
no homomorphism on the entire split Lie algebra is asserted.

\paragraph{Local scattering factors.}
We recall the scattering facts needed in the proof. For
$\theta\in N_{\mathbb R}^*$, every $g\in G_\omega$ has a unique
factorization
\begin{equation}\label{gm:eq:local-factor}
 g=g_{\theta,+}\,g_{\theta,0}\,g_{\theta,-},
\end{equation}
where the three logarithms have degrees of positive, zero, and negative
$\theta$-value, respectively \cite[Lemma 2.1]{MouScattering}.
This factorization is defined at every $\theta$, including intersections
of walls.

A \emph{scattering diagram} attaches group elements to rational
codimension-one cones, called \emph{walls}, each perpendicular to the
degrees of its attached logarithm. Only finitely many walls are nontrivial
in each total-degree quotient. Crossing a wall applies its element
or its inverse according to the crossing direction. The diagram is
\emph{consistent} if the ordered product depends only on the two ends
of a generic piecewise smooth path, transverse to the walls and avoiding
their intersections. Its total element $g$ is the product from the positive
orthant to the negative orthant; it determines the consistent diagram
up to equality of all ordered wall-crossing products \cite[Theorem 1.17]{GHKK}.
The factor $g_{\theta,0}$ is its
\emph{local factor} at $\theta$. At a finite degree truncation it is
computed by a sufficiently short transverse path from the positive to the negative
side of the walls through $\theta$; this also applies at their
intersections \cite[Sections 2.3--2.5]{MouScattering}.

In our Lie convention, the cluster diagram associated with
\eqref{gm:eq:scattering-form} has initial factors
\[ \mathcal K(e_i)^{d_i},\qquad \mathcal K(u)=\exp\!\left(\sum_{r\ge1}\frac{(-1)^{r-1}}{r^2}X_{ru}\right). \]
Equivalently, the GHKK convention uses the skew form $-\omega$ and
the sublattice with basis $d_ie_i$. We use its faithful
principal-coefficient realization, so no central degree is discarded.
The split initial factors have exponent one. A common divisor $h$ of
the $d_i$ is removed by replacing $(\omega,d)$ with $(h\omega,d/h)$:
the isomorphism $X_u\mapsto h^{-1}X_u$ carries the initial factors to
those for the new normalization. The following properties
follow from positivity, mutation invariance, and the facet description
in \cite[Theorems 1.13 and 1.24, Construction 1.30, Lemma 5.12]{GHKK}.
The cluster chamber at a seed has inequalities
$\theta(\beta_i)>0$, where the $\beta_i$ are the negatives of the
horizontal principal coefficient rows for $H$. Its local factor is
one in the interior and nontrivial on every proper face. Crossing
its $k$th facet has factor
$\mathcal K(\epsilon\beta_k)^{\epsilon d_k}$, where
$\epsilon\beta_k\in\Z_{\ge0}^n$.
The split diagram has a representative with positive wall functions
$(1+z^u)^c$, where $c>0$ and $z^u$ is a monomial of positive degree. Consequently the first nonzero logarithmic term
of each wall element is a positive rational multiple of a positive
degree $X_v$.

\begin{lemma}\label{gm:lem:local-nonvanishing}
Let $g^{\mathrm{spl}}$ be the total element of the split cluster diagram.
For every $\theta\in N_{\mathbb R}^*$,
\begin{equation}\label{gm:eq:local-nonvanishing}
 \Phi(g^{\mathrm{spl}}_{q^*\theta,0})
       =\Phi(g^{\mathrm{spl}})_{\theta,0},\qquad
 \Phi(g^{\mathrm{spl}}_{q^*\theta,0})=1
       \ \Longleftrightarrow\ g^{\mathrm{spl}}_{q^*\theta,0}=1.
\end{equation}
\end{lemma}
\begin{proof}
Uniqueness of the cluster diagram \cite[Theorem 1.12]{GHKK} makes
$g^{\mathrm{spl}}$ invariant.
Since $q^*\theta$ is invariant, uniqueness of
\eqref{gm:eq:local-factor} makes its three factors invariant as well.
Their degrees have the same signs after applying $q$. Applying
$\Phi$ and using uniqueness proves the first equality.

For the second assertion, truncate at a degree that detects the local
factor. Only finitely many walls remain. Take a small ball about
$q^*\theta$ meeting no wall whose support does not contain that point.
Choose $\lambda$ strictly positive on every split basis vector and
perturb the segment from $q^*\theta+\varepsilon\lambda$ to
$q^*\theta-\varepsilon\lambda$ inside this ball. The perturbation
can avoid all intersections and be kept strictly decreasing on every
basis vector. It therefore crosses every wall with the same sign.
Its ordered product is the local factor: the zero-degree factor in
\eqref{gm:eq:local-factor} is computed by this path, also when
$q^*\theta$ is an intersection of several walls.

If the product is nontrivial, at least one crossed factor is nontrivial.
At the least degree of the crossed factors, their logarithms add and
commutators have higher degree. Positivity makes all coefficients of
this term have the same sign, positive or negative according to the
crossing convention. The map $\Phi$ preserves total degree and adds
coefficients, so it cannot annihilate that term. Passing through the
finite truncations proves the assertion.
\end{proof}

\begin{proof}[Proof of Theorem~\ref{gm:thm:reddening}]
\emph{Lifting the prescribed sequence.}
Only sufficiency needs proof. Let $g$ be the total element of the diagram
on $N$. At every admissible stage, the split normal vectors and those
for $H$ satisfy
\begin{equation}\label{gm:eq:normal-projection}
 q(\beta^{\mathrm{spl}}_a)=\beta_i\qquad(a\in I_i).
\end{equation}
Initially this is clear. If the common sign at the mutating orbit is
$\epsilon$, then, for $b\in I_i$, $i\ne k$,
\[ (\beta^{\mathrm{spl}}_b)'=\beta^{\mathrm{spl}}_b+ \sum_{a\in I_k}[-\epsilon B^{\mathrm{spl}}_{ba}]_+ \beta^{\mathrm{spl}}_a. \]
Admissibility allows the positive part to be taken after summation,
which gives $\beta_i'=\beta_i+[-\epsilon H_{ik}]_+\beta_k$.
The normals at $I_k$ change sign. This proves
\eqref{gm:eq:normal-projection}, including the first inadmissible
result if one occurs. Ordinary sign coherence and cyclic symmetry
give a common sign at each split orbit; projection by $q$ preserves it.
Thus the lift of the given reddening sequence ends in the negative
split chamber.

\emph{Comparing local factors.}
The normals at a mutating orbit have pairwise zero skew pairing, since
its vertices are mutually nonadjacent. Their wall factors therefore
commute, and the invariant product maps under $\Phi$ to
$\mathcal K(\epsilon\beta_k)^{\epsilon d_k}$. Multiplying these factors along the given sequence, with the indicated
negative exponent at a red step, yields
$\Phi(g^{\mathrm{spl}})=g$. Hence
\eqref{gm:eq:local-nonvanishing} gives
\begin{equation}\label{gm:eq:invariant-support}
 g^{\mathrm{spl}}_{q^*\theta,0}=1
 \quad\Longleftrightarrow\quad g_{\theta,0}=1.
\end{equation}

\emph{Excluding a first obstruction.}
Suppose another sequence first produces a mixed-sign pair at a seed
$t$. There are still no arrows within any split orbit: the preceding
seed was admissible, so a path between two sheets of one orbit through
the last mutating orbit was impossible. Formula
\eqref{gm:eq:normal-projection} still holds at $t$.
Choose mixed orbits $I_k,I_i$, and let $\epsilon$ be the common sign
of the normals at $I_k$. We may mutate its split vertices once more,
since they have no internal arrows. Write $B^{\mathrm{spl}}$ for the
split matrix at $t$. For a fixed
$b\in I_j$, $j\ne k$, put
\[ a_j=\sum_{a\in I_k}[\epsilon B^{\mathrm{spl}}_{ab}]_+, \qquad a_j^0=\left[\epsilon\sum_{a\in I_k}B^{\mathrm{spl}}_{ab}\right]_+ =[-\epsilon(H_t)_{jk}]_+. \]
These values are independent of $b$ by cyclic symmetry. We have
$a_j\ge a_j^0$ for every $j$, and
\begin{equation}\label{gm:eq:strict-excess}
 a_i-a_i^0=
 \min\!\left\{\sum_{a\in I_k}[B^{\mathrm{spl}}_{ab}]_+,
               \sum_{a\in I_k}[-B^{\mathrm{spl}}_{ab}]_+\right\}>0.
\end{equation}
Both signs occur at each fixed $b\in I_i$, since translating an arrow
can move its target to any sheet of $I_i$.

Use the independent coordinates $t_j=\theta(\beta_j)$ at $t$.
The neighboring chamber for $H_t$ has inequalities
$t_k<0$ and $t_j+a_j^0t_k>0$ for $j\ne k$.
The restriction of the actual split neighboring chamber instead has
$t_k<0$ and $t_j+a_jt_k>0$. Set $t_k=-1$, $t_i=a_i$, and
$t_j>a_j+1$ for $j\ne i,k$. By \eqref{gm:eq:strict-excess}, this
point is inside the first chamber, whereas its lift lies on a proper
face of the second. Their local factors are respectively trivial and
nontrivial, contradicting \eqref{gm:eq:invariant-support}.
Thus no inadmissible result can occur.
\end{proof}

A supplied reddening sequence therefore gives a finite classification
of $\Good(B,d)$: enumerate $\Cand(B,d)$ and discard a candidate at its
first inadmissible step. Unlike
Proposition~\ref{gm:thm:finitealgorithm}, this does not require a finite
matrix mutation class. The test concerns the modulation; it does not
assert nondegeneracy of a particular potential merely because that
potential can be mutated along the chosen sequence.

\subsection{Triangular extensions}\label{gm:sec:extensions}
A \emph{triangular extension} of matrices on disjoint vertex sets
$I,J$ is a matrix whose restrictions are the given matrices and whose
arrows between the two sets all point in one direction. A
\emph{sink--source extension} is the case in which one set consists of
one new vertex. Connecting bimodules are allowed to have any simple
types compatible with the prescribed matrix.

We first recall the coefficient calculation underlying triangular
extensions \cite[Theorem 3.2 and Corollary 3.3]{CaoLi}. In the next lemma
we use vertical principal coefficients, so the coefficient columns
start at the standard basis and are nonpositive at the end of a
reddening sequence.

\begin{lemma}\label{gm:lem:connecting-coefficients}
Let $Q=\begin{pmatrix}Q_I&M\\-M^{\mathsf T}&Q_J\end{pmatrix}$ be skew-symmetric with $M\geq0$.
Perform any mutation sequence using only vertices in $J$. After its
first $\ell$ steps, denote the full matrix by $Q^{(\ell)}$ and the
coefficient matrix for $\binom{Q_J}{I_J}$ by $C_J^{(\ell)}$, where
$I_J$ is the identity matrix indexed by $J$. Then
\begin{equation}\label{gm:eq:connecting-coefficients}
 Q^{(\ell)}_{I,J}=M C_J^{(\ell)},\qquad Q^{(\ell)}_{I,I}=Q_I.
\end{equation}
Each column of $M C_J^{(\ell)}$ has one weak sign. If the sequence is
reddening for $Q_J$, then $M C_J^{(\ell)}\leq0$ at its end.
\end{lemma}
\begin{proof}
Write $c_k$ for the $k$th coefficient column and $\epsilon$ for its
sign. Ordinary sign coherence gives
\[ c'_k=-c_k,\qquad c'_j=c_j+[\epsilon q_{kj}]_+c_k\quad(j\ne k). \]
If $r_k=M c_k$, nonnegativity of $M$ implies $\epsilon r_k\geq0$,
even when $r_k=0$. The matrix mutation formula therefore gives
$r'_k=-r_k$ and $r'_j=r_j+[\epsilon q_{kj}]_+r_k$.
This proves the first identity in
\eqref{gm:eq:connecting-coefficients} by induction. All arrows joining
$I$ to the current vertex $k$ have one direction, so no path of length
two from $I$ back to $I$ passes through $k$. Hence $Q_I$ is unchanged.
The last assertion follows from $C_J^{(\ell)}\leq0$.
\end{proof}

For a split skew-symmetric matrix, vertical coefficients are the
negative transpose of our horizontal coefficients. For an original
matrix $B$, the corresponding vertical convention uses
$H=-B^{\mathsf T}=D^{-1}BD$. Positive diagonal conjugation preserves
the signs in the coefficient pattern. Thus these conventions use the
same reddening sequences; a triangular extension with arrows from $I$
to $J$ is treated by mutating $J$ first.

\begin{theorem}\label{gm:thm:triangular-extension}
Let $B$ be a triangular extension of $B_I$ and $B_J$, with fixed
positive integral degrees $d$ satisfying $BD=-(BD)^{\mathsf T}$.
Suppose both restrictions admit reddening sequences. A two-acyclic
modulation $A$ of $(B,d)$ is good if and only if its restrictions
$A_I,A_J$ are good. There is no further condition on the connecting
bimodules.
\end{theorem}
\begin{proof}
Restriction preserves goodness by Proposition~\ref{gm:prop:gauge}.
Conversely, suppose the two restrictions are good and the connecting
arrows point from $I$ to $J$. Choose reddening sequences $w_I,w_J$.
Their concatenation $w_Jw_I$ is reddening for $B$ by
\cite[Theorem 4.5 and Remark 4.6]{CaoLi}, applied in the convention
just described.

Split all vertex fields over one common cyclic extension. Denote the
sheet sets by $\widetilde I,\widetilde J$ and their connecting matrix
by $M\geq0$. Goodness of $A_J$ makes every prefix of the orbit lift
of $w_J$ admissible inside $\widetilde J$. By \eqref{gm:eq:normal-projection}, each terminal split normal projects
to the corresponding folded normal. A nonzero sign-coherent vector
has the same sign as its projection, so the lifted sequence is
reddening. During its individual mutations,
Lemma~\ref{gm:lem:connecting-coefficients} gives the connecting
matrix $M C_J^{(\ell)}$ and keeps the quiver on $\widetilde I$ unchanged.

At the end of every complete orbit mutation, the principal coefficient
matrix is equivariant under the cyclic action on its mutable and frozen
sheet indices. Its columns above one original vertex are therefore
permutations of one another. They are nonzero and sign-coherent, so
they all have the same sign. Multiplication by $M\geq0$ preserves
this weak sign. Thus all connecting arrows between each pair of
original vertices have a common direction. There are no arrows within
an orbit in $\widetilde I$, since its whole quiver is unchanged.
Together with goodness on $J$, this proves admissibility of the entire
modulation after every prefix of $w_J$.

At the end of $w_J$, the connecting matrix is nonpositive, so all
connecting arrows point from $J$ to $I$. The restriction to $I$ is
still the original good modulation. Apply the same argument to $w_I$,
with $J$ now fixed and $I$ mutating. The entire modulation remains
admissible along $w_Jw_I$. Theorem~\ref{gm:thm:reddening} proves
that it is good.
\end{proof}

The theorem also counts the good isomorphism classes over the fixed
vertex algebra. Write $d_I,d_J$ for the restricted degree vectors and put $g_{ij}=\gcd(d_i,d_j)$ and
$N_{ij}=|b_{ij}|g_{ij}/d_i$, as in
Proposition~\ref{gm:prop:finitecandidates}. Restriction and the
independent choices of connecting types give
\begin{equation}\label{gm:eq:triangular-count}
 \nu(B,d)=\nu(B_I,d_I)\nu(B_J,d_J)
 \prod_{i\in I,\,j\in J}
 \binom{N_{ij}+g_{ij}-1}{g_{ij}-1}.
\end{equation}
Every factor counts a weak composition of $N_{ij}$ into $g_{ij}$
parts; a zero connecting pair contributes one. The count is not a
quotient by vertex-field automorphisms.

Let $\mathfrak P$ be the smallest class containing the one-vertex zero
matrix and closed under mutations and triangular extensions of two
members. Let $\mathfrak P'$ be defined using sink--source extensions
instead. These are the classes denoted $\mathcal P,\mathcal P'$ in
\cite[Section 2.2]{Cao}. We include the empty matrix with its zero
modulation. Thus $\mathfrak P'\subseteq\mathfrak P$.

\begin{corollary}\label{gm:cor:class-P}
For every $B\in\mathfrak P$ and every positive integral $d$ with
$BD$ skew-symmetric, $\Good(B,d)$ is nonempty. A construction of $B$
by extensions and mutations gives its good modulations recursively
and computes $\nu(B,d)$ by \eqref{gm:eq:triangular-count} and mutation
invariance. Every matrix in $\mathfrak P$ admits a reddening sequence.
\end{corollary}
\begin{proof}
Existence of a reddening sequence is preserved by triangular extension
\cite[Theorem 4.5 and Remark 4.6]{CaoLi} and by mutation. 
The one-vertex sequence is reddening, so induction proves the final assertion.

Now fix the degrees at the terminal matrix and follow a construction
backwards to assign their restrictions at each extension step.
Mutation keeps the same degrees. At a single vertex the modulation is
unique and good. At an extension step choose the previously constructed
good restrictions and any connecting multiplicities. Such choices
exist because every $N_{ij}$ is integral. Theorem~\ref{gm:thm:triangular-extension}
proves goodness and \eqref{gm:eq:triangular-count} gives the number of
classes. At a mutation step use the bijection of
Proposition~\ref{gm:prop:bijections}. The reverse implication in the
extension theorem shows that no good class is omitted.
\end{proof}

For an acyclic matrix, repeatedly removing a sink gives a construction
by sink--source extensions. The extension theorem then shows that
\emph{every} two-acyclic modulation of that matrix is good. In
contrast, a matrix in $\mathfrak P$ can have other candidates that are
not good: the triangle in Example~\ref{ex:triangle-good-choice} is
mutation-equivalent to an acyclic quiver, but only one of its three
candidates is good. More generally, the formula
\eqref{gm:eq:triangular-count} applies whenever the two restrictions
admit reddening sequences, whether or not they belong to $\mathfrak P$.

\section{Realizations of good modulations}\label{sec:good-realization}
We now construct nondegenerate potentials on a good modulation. Each
finite mutation sequence imposes open conditions on finitely many
coefficients. We show that these conditions can be satisfied
simultaneously, together with Jacobi-finiteness when a finite-dimensional
realization exists.

A \emph{finite jet} of a potential consists of its terms through a
fixed arrow degree. An \emph{open cylinder} is the inverse image of a
Zariski-open subset of a finite jet space. We call a countable
intersection of nonempty open cylinders a \emph{very-general locus}.

\subsection{Open conditions on finitely many coefficients}\label{alg:subsec:nondegeneracy}
We first show that nondegeneracy along a finite sequence and finite-dimensionality of the Jacobian algebra are open conditions on finite jets.

\begin{proposition}\label{alg:prop:wordchart}
Let $Q$ be a two-acyclic ordinary quiver over an infinite field. For each finite mutation sequence, there is a nonempty open subset of its potential space, determined by finitely many coefficients, on which all successive reduced mutations are two-acyclic. On suitable principal open subsets, every fixed finite jet of the resulting potential is a regular function of the initial coefficients.
\end{proposition}
\begin{proof}
For one mutation, the quadratic matrices pair arrows not incident with the mutated vertex with the newly created composites. Their entries are coefficients of cyclic paths of length three in the initial quiver. These coefficients are independent. Choose a maximal minor in each matrix that is not identically zero. Where these minors do not vanish, Gaussian elimination and the unitriangular substitutions express each fixed finite jet of the reduced potential by regular functions.
Proceed by induction on the length of the sequence. On the quiver
after the first mutation, intersect the open subset supplied by the
induction hypothesis with a one-step open subset for the inverse mutation.
Both are nonempty open subsets of the same affine jet space after passing
to a common finite jet, so their intersection is nonempty. Mutating a
point in this intersection back, and using involutivity, gives an initial
potential for which the entire sequence is defined. Choose minors that
do not vanish at this point and pull back the remaining open conditions.
Clearing denominators gives a nonzero polynomial in the initial
coefficients. This is the argument of \cite[Proposition 7.3]{DWZ1}.
\end{proof}

\begin{proposition}\label{alg:prop:finitenessopen}
Let $Q$ be a loop-free quiver over $K$. If $S_*$ has a finite-dimensional completed Jacobian algebra of dimension $D$, then an open cylinder containing $S_*$ consists of potentials whose completed Jacobian algebras have dimension at most $D$. Pulling back this cylinder gives the same assertion for any algebraic family of potentials containing $S_*$.
\end{proposition}
\begin{proof}
Choose $n\ge D$ with $\m^n\subseteq J(S_*)$. In $KQ/\m^{n+1}$,
form the matrix whose columns are the cyclic derivatives multiplied by
paths on both sides. Its entries depend polynomially on the potential
coefficients through degree $n+1$. Choose a minor of size equal to the
rank of this matrix at $S_*$ that does not vanish there. On its
nonvanishing locus the truncated quotient has dimension $q\le D\le n$.
Its arrow ideal is nilpotent and has $q$th power zero. The argument of
Proposition~\ref{alg:prop:finite-algorithm} gives $\m^q\subseteq J(S)$,
so the truncated quotient equals the completed Jacobian algebra.
\end{proof}

On an irreducible finite-jet family in this open cylinder, restrict further to the open subset where the relation matrix in each $e_iT_{\le n,K}e_j$ has maximal rank. The dimensions $\dim_K e_i\mathcal P(A,S)e_j$ are constant there, and the corresponding quotients form locally free sheaves. At other points their dimensions can be larger.

\subsection{The fixed-modulation realization theorem}

Write $\Good_{\mathrm{fin}}(B,d)$ for the good modulations whose
ordinary split quivers admit a Jacobi-finite potential. Both conditions
refer to the same modulation. The general realization theorem treats
this finite-dimensionality condition separately. For matrices with a
reddening sequence, Corollary~\ref{gm:cor:reddening-finite} will show
that it follows from goodness. Corollary~\ref{gm:cor:class-P} then
supplies such modulations for every matrix in $\mathfrak P$.

\begin{proposition}\label{gm:prop:exact-opens}
For each finite mutation sequence $w$ of vertices of the species, the locus
$\mathcal N_w(A)$ of potentials nondegenerate along $w$ is a possibly
empty open cylinder depending on finitely many cyclic coefficients.
The nondegenerate locus is exactly
\[ \mathcal N(A)=\bigcap_w\mathcal N_w(A). \]
If $A$ is good, every $\mathcal N_w(A)$ contains a nonempty principal
open cylinder defined over $K$.
\end{proposition}
\begin{proof}
For a single step, two-acyclic reduction is the maximal-rank condition
on the quadratic pairings in each simple type. If the resulting signed multiplicities have both signs on a vertex pair, Lemma~\ref{gm:lem:obstruction} makes this locus empty. Otherwise it is the union of the
opens on which an appropriate maximal minor is invertible.
On each such open subset, a finite jet of the reduced potential is
rational in finitely many source coefficients, with denominators given
by the chosen minors. Along a fixed finite sequence, only finitely many
jet degrees and choices of minors occur. Pulling back the rank
conditions therefore gives an open cylinder in a sufficiently large
source jet space.
This proves the first assertion and the asserted equality.

If $A$ is good, apply Proposition~\ref{alg:prop:wordchart} to the sequence of ordinary mutations at its split vertices. The full ramified cyclic coefficient
space is a semilinear $K$-form of the entire split coefficient space.
Write a nonzero polynomial defining a sufficient principal open in
$K$-coordinates after splitting, and take the product of its coefficient
Galois conjugates. It is nonzero over $K$. Goodness makes the resulting reduction two-acyclic at the vertices of the species as well,
so its nonvanishing is sufficient for membership in $\mathcal N_w(A)$.
\end{proof}

\begin{theorem}\label{gm:thm:realization}
For the separable cyclic calculus of Section~\ref{alg:subsec:cyclic-calculus},
the following are equivalent for an arbitrary two-acyclic modulation $A$:
\begin{enumerate}
\item $A$ is good;
\item $\mathcal N(A)\ne\varnothing$;
\item $\mathcal N_w(A)\ne\varnothing$ for every finite mutation sequence $w$ of vertices of the species;
\item the full ramified potential space contains a nonempty very-general
locus of nondegenerate potentials, obtained by avoiding
countably many proper Zariski-closed conditions, each depending on
finitely many coefficients.
\end{enumerate}
When these equivalent conditions hold, the locus contains a Jacobi-finite potential if and only if the split quiver of $A$ admits a Jacobi-finite potential.
\end{theorem}
\begin{proof}
For (ii)$\Rightarrow$(i), suppose a nondegenerate potential exists. Then every reduced stage is two-acyclic over $R$. Lemma~\ref{gm:lem:obstruction} forces maximal typewise cancellation at each stage. Its multiplicity vectors must therefore follow \eqref{gm:eq:typemutation}, and every sequence remains admissible.

For (i)$\Rightarrow$(iv), choose a finite Galois splitting extension $E/K$. Degreewise ramified cyclic potential spaces become the entire ordinary cyclic potential spaces after extension to $E$. They are semilinear $K$-forms, not just constant-coefficient invariant subspaces. For every finite sequence of mutations of the ordinary split quiver, DWZ supplies a nonzero polynomial in finitely many potential coefficients whose nonvanishing implies nondegeneracy along that sequence \cite[Proposition 7.3 and Corollary 7.4]{DWZ1}. Express it in degreewise $K$-coordinates and multiply all Galois conjugates. This gives a nonzero polynomial over $K$. There are countably many sequences. The coefficients of all these polynomials generate a countable subfield of $K$. Choose the countably many coordinates algebraically independent over it. The resulting ramified potential is geometrically nondegenerate. Goodness makes its orbit-mutated quivers two-acyclic over $R$.

The implications (iv)$\Rightarrow$(ii)$\Rightarrow$(iii) follow immediately.
For (iii)$\Rightarrow$(ii), choose a nonempty principal open cylinder inside each
$\mathcal N_w(A)$ using Proposition~\ref{gm:prop:exact-opens}. There
are countably many coefficients among their defining polynomials;
algebraically independent coordinates over the subfield they generate
avoid every zero set. This proves all the equivalences.

A Jacobi-finite potential gives one additional nonempty finite-jet open by Proposition~\ref{alg:prop:finitenessopen}. Include the Galois norm of a nonzero defining polynomial among the conditions above. Conversely a Jacobi-finite ramified potential remains Jacobi-finite after splitting.
\end{proof}

The construction first gives a geometrically nondegenerate potential;
goodness ensures that its orbit mutations are two-acyclic as species.
Jacobi-finiteness contributes one further nonempty open condition on a
finite jet, which the same Galois-norm argument incorporates.

\begin{corollary}\label{gm:cor:finite-bijections}
The bijections of Proposition~\ref{gm:prop:bijections} restrict to
$\mu_k:\Good_{\mathrm{fin}}(B,d) \xrightarrow{\sim}\Good_{\mathrm{fin}}(\mu_kB,d)$.
In particular, $|\Good_{\mathrm{fin}}(B,d)|$ is a finite
mutation invariant.
\end{corollary}
\begin{proof}
Choose a Jacobi-finite nondegenerate potential by
Theorem~\ref{gm:thm:realization}. Mutation preserves finiteness by
Corollary~\ref{alg:cor:mutation-finiteness}. Its target has the good
mutated type multiplicity vector, so it supplies the required Jacobi-finite potential.
Apply the inverse mutation for the converse.
\end{proof}

\begin{corollary}\label{gm:cor:reddening-finite}
Suppose $B$ admits a reddening sequence. Every good modulation realizing
$(B,d)$ admits a Jacobi-finite nondegenerate potential. In particular,
$\Good_{\mathrm{fin}}(B,d)=\Good(B,d)$.
\end{corollary}
\begin{proof}
The proof of Theorem~\ref{gm:thm:reddening} lifts the sequence to a
reddening sequence of the ordinary split quiver. Choose a geometrically
nondegenerate ramified potential as in
Theorem~\ref{gm:thm:realization}. A nondegenerate potential on an
ordinary quiver admitting a reddening sequence has finite-dimensional
completed Jacobian algebra \cite[Corollary 5.2]{KellerGreen}.
This gives finiteness after splitting, hence over $K$ by faithful
scalar extension.
\end{proof}

For every $B\in\mathfrak P$ and every compatible $d$,
Corollaries~\ref{gm:cor:class-P} and~\ref{gm:cor:reddening-finite}
therefore give
\[ \varnothing\ne\Good(B,d)=\Good_{\mathrm{fin}}(B,d). \]
Thus the componentwise construction of Section~\ref{gm:sec:extensions}
supplies both existence and the number of isomorphism classes admitting
Jacobi-finite nondegenerate realizations. It does not test an individual
potential for nondegeneracy.

\subsection{Integral potentials on a fixed arrow lattice}\label{gm:sec:integral}
Fix a full $\RR$-stable arrow lattice $\AA\subset A$ such that each $e_i\AA e_j$ is free over its two vertex rings. Put $\TT=\RR\langle\!\langle\AA\rangle\!\rangle$, with $\OO=\C[[t]]$ and $\OO_i=\C[[t^{1/d_i}]]$.

\begin{theorem}\label{gm:thm:integral}
Suppose $A$ is good. Fix any integral potential representative $S_0\in\mathscr T$ of arrow order at least three, and any integer $N\ge1$. There exists an integral nondegenerate potential $S$ on the same arrow lattice such that
$S-S_0\in t^N\mathscr T$.

If, additionally, the same modulation admits a Jacobi-finite potential after splitting, $S$ can be chosen Jacobi-finite. The congruence is coefficientwise in the arrow completion. It does not assert equality of contracted special Jacobian algebras.
\end{theorem}
\begin{proof}
For each arrow degree $r$, take the image $\Lambda_r$ of the integral cyclic tensors in the finite-dimensional $K$-space $\Pot_r(A)$ of cyclic tensors. The image is a finite torsion-free $\OO$-module spanning $\Pot_r(A)$ over $K$, hence a full free lattice. Choose an $\OO$-basis and integral tensor representatives $c_{r,a}$ of it.

Consider all potentials
\[ S(z)=S_0+t^N\sum_{r\ge3}\sum_a z_{r,a}c_{r,a}. \]
For each finite jet, this is an invertible affine change of potential coordinates over $K$. Thus every nonzero polynomial condition used in Theorem~\ref{gm:thm:realization} remains a nonzero polynomial in finitely many $z_{r,a}$, including the single finiteness condition when a Jacobi-finite potential is supplied. After Galois norms these have coefficients in $K$. The union of their coefficient sets is countable; let $F_0\subset K$ be the subfield they generate over $\mathbb Q$.

Choose the countably many $z_{r,a}$ in $\C$ algebraically independent over $F_0$. This is possible recursively: the elements of $K$ algebraic over any countable subfield form a countable set, while $\C$ is uncountable. Adjoining the previously chosen finitely many coordinates leaves that subfield countable. All required polynomials are nonzero at this tuple. The series $S(z)$ is integral and satisfies the congruence. Theorem~\ref{gm:thm:realization} then gives the desired properties.
\end{proof}

\begin{remark}
The approximation is coefficientwise $t$-adic. Prescribing the terms
of small arrow degree can obstruct nondegeneracy, whereas a cubic term
divisible by $t^N$ can still change the quadratic pairing after mutation
over $K$. Congruence modulo $t^N$ also need not identify the contracted
orders; see Example~\ref{alg:ex:completion}.
\end{remark}

If $B$ admits a reddening sequence, the additional finiteness
hypothesis of Theorem~\ref{gm:thm:integral} is automatic by
Corollary~\ref{gm:cor:reddening-finite}. In particular, for every good
modulation constructed in Section~\ref{gm:sec:extensions}, one may
choose the approximating potential both Jacobi-finite and nondegenerate,
without changing the prescribed arrow lattice or congruence.

Every simple $K_i$--$K_j$ bimodule type admits a lattice free over both vertex rings. With $g=g_{ij}$ and $\zeta_g$ a primitive $g$th root of unity, one may take
\[ \mathscr T_{ij}^{(s)}= \OO[\varpi_i,\varpi_j]/(\varpi_i^{d_i}-t,\ \varpi_j^{d_j/g}-\zeta_g^s\varpi_i^{d_i/g}). \]
The monic relations give bases over the left and right vertex rings.
Taking direct sums gives a lift of each admissible multiplicity vector.
Different integral lattices can have the same bimodule after tensoring with $K$.

For any chosen arrow lattice and Jacobi-finite potential $S$,
Theorem~\ref{alg:thm:contraction} gives a full order whose bimodules
$e_i\PP(S)e_j$ are free over both vertex rings.

\begin{remark}\label{gm:ex:cure-potential}
	The unique good modulation of Example~\ref{ex:triangle-good-choice}
	has an explicit Jacobi-finite nondegenerate potential. After splitting,
	write $a_\alpha:1_\alpha\to2$, $b_\beta:2\to3_\beta$ and
	$c_{\beta\alpha}:3_\beta\to1_\alpha$, with
	$\alpha,\beta\in\{+,-\}$. 
	The Galois-invariant potential
	$W=\sum_{\alpha,\beta}a_\alpha b_\beta c_{\beta\alpha}$
	descends to the ramified species. At vertex $2$, with
	$x_{\alpha\beta}=[a_\alpha b_\beta]$, its premutation is
	$\sum_{\alpha,\beta}	x_{\alpha\beta}
	\bigl(c_{\beta\alpha}+b_\beta^*a_\alpha^*\bigr)$.
	The change of arrows
	$c'_{\beta\alpha}=c_{\beta\alpha}+b_\beta^*a_\alpha^*$
	splits off the perfect quadratic part
	$\sum_{\alpha,\beta}x_{\alpha\beta}c'_{\beta\alpha}$,
	leaving the acyclic species $3\to2\to1$ with zero potential.
	Mutation back therefore shows that $W$ is Jacobi-finite and
	nondegenerate.
\end{remark}

\subsection{The mutation pattern}

\begin{hypothesis}\label{hyp:initial-realization}
Fix $(B_0,d)$ and a good modulation $A$ of $B_0$ whose split
quiver admits a Jacobi-finite potential. Choose a potential $S$
on $A$ which is Jacobi-finite and nondegenerate, as
provided by Theorem~\ref{gm:thm:realization}. Fix also a full initial
$\OO$-order $\PP_0\subseteq\mathcal P(A,S)$ containing $\RR$.
The contracted order from any specified arrow lattice free over the two vertex rings
is an available choice. Theorem~\ref{gm:thm:integral} permits an
integral potential on any such lattice, but integrality of a cyclic
derivative calculus is not assumed.
\end{hypothesis}

\begin{corollary}\label{thm:good-global-pattern}
Under Hypothesis~\ref{hyp:initial-realization}, for every finite
mutation sequence $w$ the reduced QP
$(A_w,S_w)=\mu_w(A,S)$ is defined up to $R$-right-equivalence.
Its type multiplicity vector and matrix are
\[ m(A_w)=\mu_w^{\mathrm{type}}m(A),\qquad B(A_w)=\mu_wB_0, \]
and $\mathcal P(A_w,S_w)$ is finite-dimensional. Every next
mutation of the species is allowed. The local representation and character formulas proved below can therefore be iterated whenever their remaining hypotheses hold.
\end{corollary}
\begin{proof}
Since $S$ is nondegenerate, quadratic reduction is maximal in each simple type, by Lemma~\ref{gm:lem:obstruction}. Its arrow multiplicity vectors
follow Proposition~\ref{gm:thm:typemutation}; goodness keeps them
sign-coherent. Proposition~\ref{prop:intrinsic-matrix} identifies their matrices with $\mu_wB_0$.
Corollary~\ref{alg:cor:mutation-finiteness} preserves finite-dimensionality
at each step. The splitting and premutation covariance results in
Section~\ref{alg:sec:reduction} give the right-equivalence assertions.
\end{proof}

For assertions along arbitrary mutation sequences, we use Hypothesis~\ref{hyp:initial-realization}. The initial order defines the valued presentation families studied next. We mutate these representations together with the QP, using the subalgebra isomorphisms of Proposition~\ref{alg:prop:jacobian-subalgebra}.

We now study the decorated representations of these Jacobian algebras
and their presentation spaces. Generality under mutation will allow
Part~\ref{part:characters} to attach characters to the universal
presentations associated with an order.

\section{Decorated representations and mutation invariants}
\label{sec:local}\label{sec:rep-invariants}
We use Hypothesis~\ref{hyp:initial-realization} for statements about
arbitrary mutation sequences. Corollary~\ref{thm:good-global-pattern}
then gives a Jacobi-finite two-acyclic QP at every seed. For one local mutation, it is enough that the source QP be reduced
and Jacobi-finite and that the chosen vertex be incident to no oriented
two-cycle.

\subsection{The local construction}

Let $(A,S)$ be a Jacobi-finite QP over $R=\prod_iK_i$,
and write $\mathcal P=\mathcal P(A,S)$.
A \emph{decorated representation} is a pair $\mathcal M=(M,V)$, where
$M$ is a finite-dimensional right $\mathcal P$-module and
$V=\bigoplus_i V_i$ is a finite-dimensional right $R$-module.
Put $M_i=Me_i$, $\mathbf r_M=(\dim_{K_i}M_i)_i$,
$r_i(M)=(\mathbf r_M)_i$, and $v_i=\dim_{K_i}V_i$.
A morphism of decorated representations is a pair of module maps on
$M$ and on $V$; direct sums are taken componentwise.
The negative simple is $E_i^-=(0,K_i)$, supported at $i$.

A \emph{right-equivalence of decorated representations}
$(A,S,M,V)$ and $(A',S',M',V')$ consists of a QP right-equivalence
$\varphi$, an $R$-module isomorphism $\psi:M\to M'$ satisfying
\[ \psi(mp)=\psi(m)\,\overline\varphi(p)\qquad(p\in\mathcal P(A,S)), \]
and an $R$-module isomorphism $V\to V'$. Here
$\overline\varphi$ is the induced isomorphism of Jacobian algebras. In particular,
transport along $\varphi$ gives $M$ the target action
$m\cdot p'=m\,\overline\varphi^{-1}(p')$.

\paragraph{\em Reduction of a representation.}
Choose a splitting
\[ \Phi:(A,S)\xrightarrow{\sim}(A^{\red},S^{\red})\oplus(A^{\mathrm{triv}},S^{\mathrm{triv}}). \]
Projection of the trivial arrows induces an isomorphism
$r_\Phi:\mathcal P(A,S)\xrightarrow{\sim}\mathcal P(A^{\red},S^{\red})$.
The \emph{reduced part} of $(M,V)$ for this splitting is $(M^{\red},V)$,
where $M^{\red}=M$ as an $R$-module and
\[ m\cdot_{\red}p'=m\,r_\Phi^{-1}(p'). \]
After transport by $\Phi$, the trivial arrows act by zero because their
linear cyclic derivatives generate their arrow space. Reduction changes
the arrow actions; it keeps the vertex spaces and the decoration.
This is the right-module form of \cite[Definition 10.4]{DWZ1}.

For the local construction, assume now that $(A,S)$ is reduced and
that $k$ is incident to no oriented two-cycle. Set
\[ M_{\mathrm{in}}=\bigoplus_{i\ne k}M_i\otimes_{K_i}A_{ik}, \qquad M_{\mathrm{out}}=\bigoplus_{j\ne k}\Hom_{K_j}(A_{kj},M_j). \]
The right $K_k$-action on the Hom summand is $(f\lambda)(b)=f(\lambda b)$ for $\lambda\in K_k$.
The action of the incident arrows and its adjoint give
\[
\begin{aligned}
 \alpha:M_{\mathrm{in}}&\longrightarrow M_k,
 &\alpha(m\otimes a)&=ma,\\
 \beta:M_k&\longrightarrow M_{\mathrm{out}},
 &\beta(m)(b)&=mb.
\end{aligned}
\]
For composable arrows $a:i\to k$ and $b:k\to j$ after splitting the
vertex fields, write $\partial_{[ab]}S$ for the \emph{second cyclic derivative}:
it is the sum of the paths left after removing each consecutive occurrence
of $ab$ from a cyclic monomial of $S$, with its coefficient.
Equivalently, differentiate the substituted potential $[S]$ with respect
to $[ab]$, then replace every remaining composite arrow by its two-arrow path.
The resulting path goes from $j$ to $i$. Evaluating these derivatives on $M$
and using tensor--Hom adjunction gives
$\gamma:M_{\mathrm{out}}\to M_{\mathrm{in}}$.
After splitting, its $(a,b)$-entry is $(\partial_{[ab]}S)_M$.
These maps are defined by tensor contraction with the second cyclic
derivatives. They are compatible with the semilinear Galois action on
the split data and therefore descend to $K$.
The first-derivative relations give
\[ \alpha\gamma=0,\qquad\gamma\beta=0. \]
The composite map $\beta\alpha$ need not vanish. The local maps are
\[ \vcenter{\hbox{$\displaystyle \xymatrix@C=32pt@R=20pt{&M_k\ar[dr]^\beta&\\M_{\mathrm{in}}\ar[ur]^\alpha&&M_{\mathrm{out}}\ar[ll]_\gamma}$}}. \]

The premutated middle module and decoration are
\begin{equation}\label{eq:newmodules}
 \begin{split}
 M'_k&=\ker\gamma/\im\beta\ \oplus\ \im\gamma\ \oplus\
       \ker\alpha/\im\gamma\ \oplus\ V_k,\\
 V'_k&=\ker\beta/(\ker\beta\cap\im\alpha).
 \end{split}
\end{equation}
All terms are $K_k$-vector spaces. Choose a retraction
$\rho:M_{\mathrm{out}}\to\ker\gamma$ and a section
$\varsigma:\ker\alpha/\im\gamma\to\ker\alpha$, and define
\[
\alpha'=\begin{pmatrix}-\pi\rho\\-\gamma\\0\\0\end{pmatrix}:M_{\mathrm{out}}\to M'_k,
 \qquad
 \beta'=\begin{pmatrix}0&\iota&\iota\varsigma&0\end{pmatrix}:M'_k\to M_{\mathrm{in}},
 \qquad \gamma'=\beta\alpha.
\]
Here $\pi:\ker\gamma\to\ker\gamma/\im\beta$ is the quotient
map and each $\iota$ is the indicated inclusion. The second component
of $\alpha'$ has codomain $\im\gamma$.
Arrows not incident with $k$ keep their actions, and a new composite
$[ab]$ acts by $b_Ma_M$. The reverse-arrow actions are obtained from
$\alpha'$ and $\beta'$ by tensor--Hom adjunction. The representation
spaces and decorations at all other vertices are unchanged. We denote these data over the premutated QP by
$\widetilde\mu_k\mathcal M=(\widetilde M,\widetilde V)$.
Choose a reduced part $(A',S')$ of the premutated QP as in
Section~\ref{alg:sec:premutation}. We define the \emph{reduced mutation}
by the construction just given:
\[ \mu_k\mathcal M=(\widetilde\mu_k\mathcal M)^{\red}. \]
The next theorem verifies the relations and proves independence of the
choices up to right-equivalence.

\begin{theorem}\label{thm:representation-mutation}
The formulas above define a finite-dimensional representation of the
premutated QP. Its reduced right-equivalence class is independent of
$\rho,\varsigma$, of the splitting of the QP, and of the chosen
representative of the source right-equivalence class. Reduced mutation
is an involution on the classes for which the QP mutations are defined,
and commutes with direct sums. 
\end{theorem}
\begin{proof}
The local identities are
$\beta'\alpha'=-\gamma,\ \alpha'\beta=0$, and $\alpha\beta'=0$.
They are respectively the composite and reverse-arrow derivative
relations. The unchanged derivative relations are those of the original
representation after replacing each passage through $k$ by its composite arrow. Together these identities verify all cyclic derivative relations. A long path in the
premutation groups into outside arrows, composites, and reverse return
paths. By the first identity, return paths become positive-order
original paths. After at most two end pieces, a sufficiently long path
therefore acts by zero. The action extends to the completed algebra.

For two choices, the difference $\rho_1-\rho$ vanishes on
$\ker\gamma$ and therefore factors through $\gamma$. Similarly,
$\varsigma_1-\varsigma$ takes values in $\im\gamma$. Thus write
$\rho_1=\rho+\xi\gamma$ and $\varsigma_1=\varsigma+\iota\eta$, where
$\xi:\im\gamma\to\ker\gamma$ and
$\eta:\ker\alpha/\im\gamma\to\im\gamma$ are $K_k$-linear.
The change of sections is intertwined by
\[
\Psi=\begin{pmatrix}
 1&\pi\xi&0&0\\0&1&-\eta&0\\0&0&1&0\\0&0&0&1
 \end{pmatrix}.
\]
Indeed $\Psi\alpha'=\alpha'_1$ and $\beta'_1\Psi=\beta'$.
This is a $K_k$-linear isomorphism. The other vertex maps are identities.
For two splittings $\Phi_1,\Phi_2$, the map
$r_{\Phi_2}r_{\Phi_1}^{-1}$ identifies the two reduced module actions.
Theorem~\ref{alg:thm:cancel} lifts it to a right-equivalence of the
reduced QPs. The identity on the vertex spaces, together with the
unchanged decoration, gives the required right-equivalence of reduced
representations.

For the double local construction, put
$M_k^0=\ker\beta\cap\im\alpha$. The identities
\[ \ker\alpha'=\im\beta,\qquad \im\beta'=\ker\alpha, \qquad \gamma'=\beta\alpha \]
give the twice-mutated positive space
\[ M_k^0\oplus(\im\alpha/M_k^0)\oplus M_k/(\ker\beta+\im\alpha)\oplus (\ker\beta+\im\alpha)/\im\alpha. \]
After interchanging the last two summands, this is the associated
graded space of the filtration
$0\subseteq M_k^0\subseteq\im\alpha\subseteq\ker\beta+\im\alpha\subseteq M_k$,
and the twice-mutated decoration is $V_k$. Compatible sections of this
filtration identify the incident actions with the originals, with the
sign change on the restored outgoing family in the double-QP identity
of Proposition~\ref{alg:prop:double-premutation}.

Compatibility with a source right-equivalence follows by comparing the
incoming and outgoing actions and their duals, as in
\cite[Proposition 10.10]{DWZ1}. It is defined over the vertex algebra:
factor images of incident arrows at their first or last incident factor
and transform the reverse actions by the inverse adjoint matrices,
as in Proposition~\ref{alg:prop:generic-premutation-covariance}.
The chain rule transforms $\gamma$ accordingly; the remaining
changes are changes of the sections just treated. Equivalently the
calculation may be checked after finite separable splitting, where it
is the ordinary DWZ calculation, and the tensor and duality maps descend.
Finally kernels, images and the mutation construction commute with
finite direct sums.
\end{proof}

Taking dimensions, with all dimensions and ranks over $K_k$, gives
\[
\begin{aligned}
 \dim M'_k
 &=\dim M_{\mathrm{in}}+\dim M_{\mathrm{out}}
   -\rank\alpha-\rank\beta-\rank\gamma+v_k,\\
 \dim V'_k
 &=\dim M_k-\rank\alpha-\rank\beta+\rank(\beta\alpha).
\end{aligned}
\]

\subsection{Homomorphisms and the presentation invariant}

As in \cite[Propositions 6.1--6.2]{DWZ2}, mutation preserves Hom
spaces modulo the maps supported at the mutation vertex.

Let $\mathcal M=(M,V)$ and $\mathcal N=(N,W)$ be finite
decorated representations and suppose mutation at $k$ is defined.
Write $\Hom^{[k]}(M,N)$ for the module maps whose
components at all vertices other than $k$ vanish. Define the
\emph{$h$-vector} and its dual by
\[
\begin{aligned}
 \mathbf h(M)&=(h_i(M))_i,&h_i(M)&=-\dim_{K_i}\coker\alpha_{i;M},\\
 \check{\mathbf h}(M)&=(\check h_i(M))_i,&\check h_i(M)&=-\dim_{K_i}\ker\beta_{i;M}.
\end{aligned}
\]
Both vectors depend only on the positive representation. The minus signs
agree with the convention of \cite[(3.2)]{DWZ2}: because we count quotient
ranks, our $\mathbf h(M)$ is their $h$-vector of the dual representation
over the opposite algebra. In particular, $h_i(M)\leq0$ and
$\check h_i(M)\leq0$. We use
$N_{\mathrm{in}}$ and $N_{\mathrm{out}}$ for the spaces defined from $N$
by the same construction.

\begin{proposition}\label{thm:hom-mutation}
Restriction to the mutation vertex gives
\begin{equation}\label{eq:confined-hom}
 \Hom^{[k]}(M,N)\simeq
 \Hom_{K_k}(\coker\alpha_{k;M},\ker\beta_{k;N})
\end{equation}
as $K_k$-vector spaces. There is a $K$-linear isomorphism, compatible
with composition,
\begin{equation}\label{eq:hom-quotient-mutation}
 \frac{\Hom(M,N)}{\Hom^{[k]}(M,N)}
 \simeq
 \frac{\Hom(M',N')}{\Hom^{[k]}(M',N')},
 \qquad (M',V')=\mu_k\mathcal M,\ (N',W')=\mu_k\mathcal N.
\end{equation}
Consequently, with dimensions over the scalar field,
\begin{equation}\label{eq:hom-number-mutation}
 \hom(M',N')-\hom(M,N)
 =d_k\bigl(h_k(M')\check h_k(N')-h_k(M)\check h_k(N)\bigr).
\end{equation}
On full subcategories carried to each other by mutation, the induced
functor is an equivalence after quotienting morphisms by the ideal of
maps supported at $k$ and by all morphisms between the decorations.
\end{proposition}
\begin{proof}
A map supported at $k$ is a vertex-linear map $f_k$ satisfying
$f_k\alpha_M=0$ and $\beta_Nf_k=0$. This proves
\eqref{eq:confined-hom}.

Let $e=1-e_k$. A map $f:Me\to Ne$ of representations of the common
subalgebra induces maps
$f_{\mathrm{in}}:M_{\mathrm{in}}\to N_{\mathrm{in}}$ and
$f_{\mathrm{out}}:M_{\mathrm{out}}\to N_{\mathrm{out}}$.
It extends to a map $M\to N$ if and only if
\begin{equation}\label{eq:extension-flags}
 f_{\mathrm{in}}(\ker\alpha_M)\subseteq\ker\alpha_N,
 \qquad f_{\mathrm{out}}(\im\beta_M)\subseteq\im\beta_N.
\end{equation}
Here it already commutes with the composite actions, so
$f_{\mathrm{out}}\beta_M\alpha_M=\beta_N\alpha_Nf_{\mathrm{in}}$.
To prove sufficiency, first define $f_k$ on $\im\alpha_M$ by
$f_k\alpha_M=\alpha_Nf_{\mathrm{in}}$. On a complement of $\im\alpha_M$ in
$M_k$, lift $f_{\mathrm{out}}\beta_M$ through the split surjection
$N_k\to\im\beta_N$. The complements and lifts exist over $K_k$. The resulting $f_k$
shows that the image of restriction consists exactly of the maps
satisfying \eqref{eq:extension-flags}.

For the mutated diagram the two flags are interchanged:
$\ker\alpha'_M=\im\beta_M$ and
$\im\beta'_M=\ker\alpha_M$. The outside algebra and its actions are
identified by premutation. Thus the two restriction images are the
same, with the same composition. Their kernels are the supported maps,
which proves \eqref{eq:hom-quotient-mutation}.
Reduction transports this statement by the prescribed quotient
isomorphisms of Section~\ref{alg:sec:reduction}.

The two vector spaces on the right of \eqref{eq:confined-hom} have
$K_k$-dimensions $-h_k(M)$ and $-\check h_k(N)$. The space of maps
therefore has $K$-dimension $d_kh_k(M)\check h_k(N)$. This proves \eqref{eq:hom-number-mutation}.
The reverse construction gives essential surjectivity of this quotient
functor.
\end{proof}

\paragraph{\em The presentation invariant.}
Write $P_i=e_i\mathcal P$ and $P(a)=\bigoplus_iP_i^{a_i}$, and
let $I_i=\Hom_K(\mathcal P e_i,K)$ be the corresponding right
injective. Recall that $\rad\mathcal P$ is the image of the arrow ideal.
A \emph{projective cover} $Q\twoheadrightarrow N$ has kernel
contained in $Q\rad\mathcal P$. A minimal projective presentation uses
covers of $N$ and of the first kernel. Dually, a minimal injective
copresentation uses the \emph{injective envelope} of $N$ and of its cokernel;
an envelope is an essential embedding into an injective module, meaning
that every nonzero submodule of the latter meets the image of $N$. In a minimal
projective presentation of a decorated representation, put its decoration
in the direct summand $P(v)\xrightarrow{0}0$. Its index is
$\delta_{\mathcal M}=[P_+^{\mathcal M}]-[P_-^{\mathcal M}]$.
We also use the injective index
$\check\delta_{\mathcal M}=[I_+^{\mathcal M}]-[I_-^{\mathcal M}]$,
where the decoration is the direct summand $0\xrightarrow{0}I(v)$.
Brackets denote the vectors of indecomposable-projective or
indecomposable-injective multiplicities, called \emph{Betti multiplicities}
in a minimal (co)presentation.

For the decorated minimal presentation
$f_{\mathcal M}:P_-^{\mathcal M}\to P_+^{\mathcal M}$, define the
\emph{presentation invariant space} and its dimension by
\begin{equation}\label{eq:E-pres}
 \op{E}(\mathcal M,\mathcal N)
 =\coker\bigl(\Hom(P_+^{\mathcal M},N)
                  \longrightarrow\Hom(P_-^{\mathcal M},N)\bigr),
 \qquad \op{e}(\mathcal M,\mathcal N)=\dim_K\op{E}(\mathcal M,\mathcal N).
\end{equation}
The map is precomposition with $f_{\mathcal M}$. Thus $\op{E}$ denotes
 a $K$-vector space and $\op{e}$ its dimension. We omit the algebra
when it is understood, and write
$\op{E}(\mathcal M)=\op{E}(\mathcal M,\mathcal M)$ and
$\op{e}(\mathcal M)=\op{e}(\mathcal M,\mathcal M)$ for equal arguments.
Set
\[ \langle u,v\rangle_D=\sum_i d_i u_i v_i. \]
Since $\dim_K\Hom(P_i,N)=d_i r_i(N)$ and the kernel in
\eqref{eq:E-pres} is $\Hom(M,N)$, one has
\begin{equation}\label{eq:E-dimension}
 \op{e}(\mathcal M,\mathcal N)
 =\hom(M,N)-\langle\delta_{\mathcal M},\mathbf r_N\rangle_D.
\end{equation}
The sign follows from $\delta=[P_+]-[P_-]$.
There is also a description involving two presentations:
\begin{equation}\label{eq:E-two-presentations}
 \op{E}(\mathcal M,\mathcal N)\simeq
 \frac{\Hom(P_-^{\mathcal M},P_+^{\mathcal N})}
 {\{a f_{\mathcal M}-f_{\mathcal N}b:
 a:P_+^{\mathcal M}\to P_+^{\mathcal N},\
 b:P_-^{\mathcal M}\to P_-^{\mathcal N}\}}.
\end{equation}
Indeed, projectivity lifts maps $P_-^{\mathcal M}\to N$ through
$P_+^{\mathcal N}\to N$. The ambiguity consists of maps factoring
through $f_{\mathcal N}$; lifting maps from $P_+^{\mathcal M}$ gives
the other part of the denominator.
For projective presentations $f,g$, we also write $\op{E}(f,g)$ for
this quotient of the Hom space and put
$\op{e}(f,g)=\dim_K\op{E}(f,g)$.

\begin{proposition}\label{prop:opposite-duality}
Let $\mathbb D=\Hom_K(-,K)$. Reversing the arrow bimodules and the
words in the potential gives $(A^{\mathrm{op}},S^{\mathrm{op}})$,
whose Jacobian algebra is $\mathcal P^{\mathrm{op}}$ and whose exchange
matrix is $-B$, with the same symmetrizer. If
$g_{\mathcal N}:I_+^{\mathcal N}\to I_-^{\mathcal N}$ is the
decorated minimal injective copresentation, there is a natural
isomorphism
\begin{equation}\label{eq:E-opposite-duality}
 \coker\bigl(\Hom(M,I_+^{\mathcal N})\longrightarrow
                    \Hom(M,I_-^{\mathcal N})\bigr)
 \simeq
 \op{E}_{\mathcal P^{\mathrm{op}}}
                    (\mathbb D\mathcal N,\mathbb D\mathcal M).
\end{equation}
Moreover,
\begin{equation}\label{eq:opposite-indices}
 \delta^{\mathrm{op}}_{\mathbb D\mathcal M}
     =\check\delta_{\mathcal M},\qquad
 \check\delta^{\mathrm{op}}_{\mathbb D\mathcal M}
     =\delta_{\mathcal M}.
\end{equation}
Opposite duality commutes with mutation up to right-equivalence.
\end{proposition}
\begin{proof}
Dualizing $I_+^{\mathcal N}\to I_-^{\mathcal N}$ gives the
projective presentation
$\mathbb DI_-^{\mathcal N}\to\mathbb DI_+^{\mathcal N}$ of
$\mathbb DN$. For $\varepsilon\in\{+,-\}$, the natural identifications
$\Hom_{\mathcal P}(M,I_\varepsilon^{\mathcal N}) \simeq\Hom_{\mathcal P^{\mathrm{op}}} (\mathbb DI_\varepsilon^{\mathcal N},\mathbb DM)$
identify the two cokernels, including their decoration summands.
The same multiplicities give \eqref{eq:opposite-indices}.
Reversing the words identifies the two Jacobian ideals and changes
$B$ to $-B$. Finally, duality exchanges the incoming and outgoing
maps and interchanges the kernel and quotient sequences in
\eqref{eq:newmodules}. Dual sections give the opposite mutation;
the resulting sign changes on the reversed arrows define right-equivalences.
\end{proof}

Thus the cokernel defining the injective presentation invariant is
$\op{E}$ over the opposite algebra, with the two arguments dualized and
interchanged.
Its dimension need not equal $\op{e}(\mathcal M,\mathcal N)$ for the same
ordered pair. We use \eqref{eq:E-opposite-duality} rather than introduce
a second pairing. In particular,
\begin{equation}\label{eq:E-dual-dimension}
 \op{e}_{\mathcal P^{\mathrm{op}}}(\mathbb D\mathcal N,\mathbb D\mathcal M)
 =\hom(M,N)-\langle\mathbf r_M,\check\delta_{\mathcal N}\rangle_D.
\end{equation}

The local Betti calculation extends to finite-dimensional completed
Jacobian algebras without the extra hypothesis on the uncompleted derivative
ideal imposed in \cite[(10.3), Proposition 10.4]{DWZ2}. We give the
completed-algebra argument, since this distinction is essential in
Example~\ref{alg:ex:completion}.

\begin{lemma}\label{lem:completed-betti}
Let $\mathcal P$ be the completed Jacobian algebra of a reduced Jacobi-finite
ramified QP, and let $M$ be a finite right $\mathcal P$-module.
For the local maps
$M_{\rm out}\xrightarrow{\gamma_i}M_{\rm in}
 \xrightarrow{\alpha_i}M_i$, the multiplicities in a minimal
projective presentation of $M$ satisfy
\[ \beta_{0,i}(M)=\dim_{K_i}\coker\alpha_i,\qquad \beta_{1,i}(M)=\dim_{K_i}\bigl(\ker\alpha_i/\im\gamma_i\bigr). \]
\end{lemma}
\begin{proof}
First pass to a finite separable splitting field $F$. We work with
one vertex of the resulting ordinary quiver, still denoted by $i$.
Write $T=F\langle\!\langle Q\rangle\!\rangle$, let $\mathfrak m$
be its arrow ideal, let $J$ be the closed Jacobian ideal, and put
$P=T/J$. Choose $N$ with $\mathfrak m^N\subseteq J$; this is possible
because the completed Jacobian algebra is finite-dimensional. Denote the left
simple at $i$ by $L_i=P e_i/(\rad P)e_i$.

With paths multiplied from left to right, there is a complex of left
projectives
\begin{equation}\label{eq:left-simple-complex}
 \bigoplus_{b:i\to k}P e_k
 \xrightarrow{d_2}
 \bigoplus_{a:j\to i}P e_j
 \xrightarrow{d_1}P e_i\longrightarrow L_i\longrightarrow0.
\end{equation}
Here $d_1((p_a)_a)=\sum_a p_a a$. The second cyclic derivatives defined
in Section~\ref{sec:local} give
\[ \bigl(d_2((p_b)_b)\bigr)_a=\sum_{b:i\to k}p_b\,\partial_{[ab]}S. \]
Indeed, cutting the last arrow in each first derivative gives
\[ \partial_bS=\sum_{a:j\to i}(\partial_{[ab]}S)a\qquad(b:i\to k). \]
All expressions are evaluated in $P$, so the formal series have finite
images. In particular, the Jacobian relations give $d_1d_2=0$.

We verify exactness at the middle term, keeping the closed ideal $J$.
Lift $(p_a)_a\in\ker d_1$ to a tuple in $T$ and put
$r=\sum_a\widetilde p_a a\in J$. Since $J$ is the closed ideal generated by the cyclic derivatives,
modulo $\mathfrak m^{N+1}$ the element $r$ is a finite sum of products
$u(\partial_c S)v$, with $u,v$ paths, all ending at vertex $i$.
Separate the terms where $v$ has length zero from those where it has
positive length. Removing the last arrow of a positive-length $v$
produces a coefficient in
$J$, hence zero in $P$. When $v$ has length zero, the derivative contributes only
when $c$ starts at $i$, and removing its last arrow gives an image
of $d_2$. Finally, cutting off the last arrow of the remainder in
$\mathfrak m^{N+1}$ gives coefficients in $\mathfrak m^N$, which
also vanish in $P$. Thus $\ker d_1=\im d_2$.
Exactness at $P e_i$ follows by factoring every positive-length path
ending at $i$ by its last arrow. In particular,
\eqref{eq:left-simple-complex} is the beginning of a projective
resolution; no injectivity of $d_2$ is asserted.

Tensoring this resolution with the right module $M$ identifies its
last two differentials with $\gamma_i$ and $\alpha_i$. Therefore
\[ M\otimes_P L_i\simeq\coker\alpha_i,\qquad \op{Tor}_1^P(M,L_i)\simeq\ker\alpha_i/\im\gamma_i. \]
For a minimal right-projective resolution of $M$, tensoring with
$L_i$ kills its radical differentials. The dimensions of these two
spaces are consequently $\beta_{0,i}(M)$ and $\beta_{1,i}(M)$.

Finite separable scalar extension is faithfully flat and commutes with
these kernels, images, and finite initial segments of projective
resolutions. Each vertex projective splits into the projectives at its
sheets, with the same multiplicity at each sheet. The ordinary calculation
therefore gives the stated dimensions over $K_i$.
\end{proof}

Adding the summand $P(v)\xrightarrow{0}0$ for the decoration gives
\begin{equation}\label{eq:generic-index-local}
 \delta_i=r_i-v_i-\dim_{K_i}\coker\gamma_i,
 \qquad
 \check\delta_i=r_i-v_i-\dim_{K_i}\ker\gamma_i.
\end{equation}
The first equality follows by subtracting the two Betti
multiplicities in Lemma~\ref{lem:completed-betti}; the second
follows by Proposition~\ref{prop:opposite-duality}. These identities
hold for every finite decorated representation, not only for a
general presentation. Since the arrow ideal is nilpotent in
$\mathcal P$, a finite jet of the potential suffices to evaluate
these matrices in the completed Jacobian algebra.
Consequently
\begin{equation}\label{eq:indices-E}
 \check\delta_{\mathcal M}=\delta_{\mathcal M}+\mathbf r_M B.
\end{equation}
We call $\mathcal M$ \emph{$E$-rigid} if $\op{E}(\mathcal M)=0$, or equivalently
$\op{e}(\mathcal M)=0$. This condition on a
representation is distinct from rigidity of its potential. A finite
collection is \emph{$E$-compatible} if its direct sum is $E$-rigid. Put
$\op{e}^{\mathrm{sym}}(\mathcal M,\mathcal N) =\op{e}(\mathcal M,\mathcal N)+\op{e}(\mathcal N,\mathcal M)$.
Equations~\eqref{eq:E-dimension}, \eqref{eq:E-dual-dimension} and
\eqref{eq:indices-E}, together with skew-symmetry of $BD$, show that
$\op{e}(\mathcal M)$ and $\op{e}^{\mathrm{sym}}(\mathcal M,\mathcal N)$
are unchanged by opposite duality.

\begin{theorem}\label{thm:E-mutation}
For mutation of finite decorated representations at $k$,
\begin{equation}\label{eq:Ep-mutation}
 \op{e}(\mathcal M',\mathcal N')-\op{e}(\mathcal M,\mathcal N)
   =d_k\bigl(h_k(N')h_k(M)-h_k(N)h_k(M')\bigr).
\end{equation}
Thus $\op{e}(\mathcal M)$ and $\op{e}^{\mathrm{sym}}(\mathcal M,\mathcal N)$
are mutation invariant. The directed invariant need not be.
\end{theorem}
\begin{proof}
Lemma~\ref{lem:completed-betti} identifies the presentation index
with the local expression \eqref{eq:generic-index-local}. The local
mutation calculation of \cite[Lemma 5.2]{DWZ2}, applied to that expression
in our conventions, gives
\begin{equation}\label{eq:index-unsimplified}
 \delta'_k=-\delta_k,\qquad
 \delta'_j=\delta_j-[b_{kj}]_+h_k+[-b_{kj}]_+h'_k\quad(j\ne k).
\end{equation}
After a splitting extension, mutate successively at the sheets above
$k$. There are no arrows between them. Summing their ordinary index
formulas gives \eqref{eq:index-unsimplified}, since the incident
multiplicities at any other sheet are the indicated bimodule ranks.
Here $h_k=h_k(M)$ and $h'_k=h_k(M')$; no disjoint-support assumption is used.

The local construction gives
\[
\begin{aligned}
 h_k(N)&=\dim_{K_k}N_{\mathrm{in}}-r_k(N)-r_k(N')-\check h_k(N'),\\
 h_k(N')&=\dim_{K_k}N_{\mathrm{out}}-r_k(N)-r_k(N')-\check h_k(N).
\end{aligned}
\]
Substitute these identities, \eqref{eq:index-unsimplified}, and the
Hom comparison \eqref{eq:hom-number-mutation} into
\eqref{eq:E-dimension}. Using $d_jb_{kj}=-d_kb_{jk}$ gives
\eqref{eq:Ep-mutation}. Its right side is antisymmetric in the two
representations, proving both invariance statements.
\end{proof}

For the injective mutation statement, apply \eqref{eq:Ep-mutation}
over $\mathcal P^{\mathrm{op}}$ to
$(\mathbb D\mathcal N,\mathbb D\mathcal M)$. Opposite duality replaces
$\mathbf h$ by $\check{\mathbf h}$ and reverses the two arguments.
For presentations with disjoint supports, $h_k(M)=-[\delta_{\mathcal M,k}]_+$
and $h_k(M')=-[-\delta_{\mathcal M,k}]_+$, so
\[ \op{e}(\mathcal M',\mathcal N')-\op{e}(\mathcal M,\mathcal N) =d_k\bigl([\delta_{\mathcal M,k}]_+[-\delta_{\mathcal N,k}]_+ -[-\delta_{\mathcal M,k}]_+[\delta_{\mathcal N,k}]_+\bigr). \]
Thus the directed invariant is unchanged when the two $k$th weight
coordinates have the same weak sign. No rigidity assumption on the
potential is used.

\paragraph{\em Sign coherence.}
Write $\beta_{0,i},\beta_{1,i}$ for the multiplicities in a decorated
minimal projective presentation. Its differential is radical, so every map in the denominator of
\eqref{eq:E-two-presentations}, with $\mathcal N=\mathcal M$, has image
in $P_+^{\mathcal M}\rad\mathcal P$. There is therefore a surjection
\[ \op{E}(\mathcal M)\twoheadrightarrow \Hom_R\bigl(P_-^{\mathcal M}/P_-^{\mathcal M}\rad\mathcal P, P_+^{\mathcal M}/P_+^{\mathcal M}\rad\mathcal P\bigr). \]
The dimension of the right-hand side gives
\begin{equation}\label{eq:E-lower-bound}
 \op{e}(\mathcal M)\geq\sum_i d_i\beta_{0,i}\beta_{1,i}.
\end{equation}

\begin{corollary}\label{thm:sign-coherence}
If $\op{e}(\mathcal M)=0$, then at every allowed mutation vertex $i$,
\begin{equation}\label{eq:E0-betti}
 h_i(M)=-[\delta_i]_+,\qquad h_i(M')=-[-\delta_i]_+.
\end{equation}
If $\op{e}(\bigoplus_j\mathcal M_j)=0$, the presentation weights of the
summands are coordinatewise sign-coherent. Under Hypothesis~\ref{hyp:initial-realization}, these conclusions hold
after every common finite sequence of species mutations; locally they hold
whenever the successive mutations are defined. Opposite duality
gives the same assertions for injective weights and outgoing kernels.
\end{corollary}
\begin{proof}
The multiplicities of the decorated minimal projective presentation
at $i$ are $-h_i(M)$ and $-h_i(M')$, by
Lemma~\ref{lem:completed-betti}, including the decoration,
and \eqref{eq:newmodules}.
Their difference is $\delta_i$. Inequality~\eqref{eq:E-lower-bound}
forces their product to vanish when $\op{e}(\mathcal M)=0$, proving
\eqref{eq:E0-betti}. Minimal multiplicities add under direct sums,
so opposite signs cannot occur at one coordinate among compatible
summands. Apply Theorem~\ref{thm:E-mutation} after mutation.
Opposite duality preserves $\op{e}(\mathcal M)$, which proves the
injective assertions.
\end{proof}

The direct sum of the negative simples has $\op{e}=0$. Its successive
mutations therefore form sign-coherent collections, whose characters
will give cluster monomials in Section~\ref{sec:cluster-application}.

\section{General presentations and valued families}\label{sec:weights}
We separate two steps. The algebraic step proves that mutation
preserves general represented objects, up to changes of presentation
bases; it uses the fixed vertex fields and does not require an
integral order. The valued step uses the full initial order in
Hypothesis~\ref{hyp:initial-realization} to select a universal
presentation. The algebraic result then applies to this one valued
object along every finite mutation sequence.

\subsection{Presentation weights and mutation}\label{subsec:presentation-spaces}

For a finite algebra $P$ with the fixed vertex idempotents, put $P_i=e_iP$ and
$P(\eta)=\bigoplus_iP_i^{\eta_i}$ for $\eta\in\mathbb Z_{\geq0}^I$.
Our projective convention is
\begin{equation}\label{eq:presweight}
 \delta=[P_+]-[P_-],\qquad
 \PHom(\delta)=\Hom_P(P([-\delta]_+),P([\delta]_+)).
\end{equation}
We write $\PHom_P(\delta)$ when the algebra needs to be specified.
We call $\delta$ the \emph{presentation weight}, or the \emph{$\delta$-vector},
of the presentation.
A \emph{contractible summand} is an isomorphism $Q\xrightarrow{\sim}Q$
appearing in the two degrees; it has zero cokernel and zero weight.
After removing these summands, write a presentation $f$ as a direct sum
of a minimal projective presentation and $P(v)\xrightarrow{0}0$.
Its \emph{decorated cokernel} is $(\coker f,\bigoplus_iK_i^{v_i})$.
Thus a projective summand on which the differential vanishes contributes
decoration, not a positive module.
The source and target in \eqref{eq:presweight} have \emph{disjoint supports}:
no indecomposable projective summand occurs in both.
An assertion about a \emph{general presentation} means that it holds on
a nonempty Zariski-open subset of this affine space. A general decorated representation of weight $\delta$ is its decorated
cokernel, considered up to isomorphism. On a suitable nonempty open,
the cokernel ranks and decoration dimensions are constant. We abbreviate
$X_\delta=\PHom(\delta)$ and write
\[ K_\delta^{\mathrm{pres}}:=K(X_\delta) \]
for its rational function field. If $X_\delta$ is a point, this field is
$K$. A $K$-basis $f_1,\ldots,f_q$ identifies $X_\delta$ with
$\mathbb A_K^q$; its \emph{universal presentation} is the tautological
map $f_{\mathrm{univ}}=\sum_jz_jf_j$ over $K[z_1,\ldots,z_q]$.
Its pullback to $K_\delta^{\mathrm{pres}}=K(z_1,\ldots,z_q)$ is the
algebraically generic presentation.
The group of changes of presentation bases
$G_\delta=\Aut_P(P([\delta]_+))\times\Aut_P(P([-\delta]_+))$
acts by $(g_+,g_-)f=g_+fg_-^{-1}$. The \emph{orbit saturation} of a
subset is the union of the orbits meeting it. On an open subset with constant cokernel and decoration ranks, choose
bases at the vertices. The closure of the union of all vertex-basis
orbits of the resulting decorated cokernels is the \emph{presentation
family} of weight $\delta$; it need not be a single orbit. This is the presentation-family construction of \cite{DF}. An
\emph{algebraically generic presentation} over a
field $F/K$ means an embedding $K_\delta^{\mathrm{pres}}\hookrightarrow F$
and the pullback of this universal map. It lies in every nonempty
$K$-defined open of the presentation space, not just in one prescribed
open.

In particular $P_i$ has weight $\mathbf e_i$ and $E_i^-$ has weight $-\mathbf e_i$.
The injective convention is the dual one: an injective copresentation
$M\to I_+\to I_-$ has weight $[I_+]-[I_-]$, with the same negative
decoration convention. We denote this weight by $\check\delta$.
Neither weight is a dimension or rank vector.

For a finite-dimensional Jacobian algebra $P=\mathcal P(A,S)$, the
local Betti calculation and opposite duality give
\[ \check\delta=\delta+\mathbf r_MB,\qquad \delta_k=h_k(M')-h_k(M),\qquad \check\delta_k=\check h_k(M')-\check h_k(M). \]
We now identify these $h$-coordinates for presentations with disjoint supports.

\begin{proposition}\label{prop:disjoint-support-indices}
Let $\mathcal P=\mathcal P(A,S)$ be a finite Jacobian algebra over
$R=\prod_iK_i$, and let $f:P_-\to P_+$ be a projective
presentation in $\PHom(\delta)$: the multiplicities of $P_i$ in
$P_+$ and $P_-$ are $[\delta_i]_+$ and $[-\delta_i]_+$.
Let $\mathcal M$ be its decorated cokernel. For every allowed mutation
at $k$, the $h$-coordinates satisfy
$h_k(M)=-[\delta_k]_+$ and $h_k(M')=-[-\delta_k]_+$.
For an injective copresentation with disjoint supports, the corresponding
identities are $\check h_k(M)=-[\check\delta_k]_+$ and
$\check h_k(M')=-[-\check\delta_k]_+$.
\end{proposition}
\begin{proof}
We first note that every entry of $f$ is radical: its source and target have no
common indecomposable projective. Thus $P_+\to M$ is a projective
cover. Since $P_-\to\im f$ is surjective, a change of basis in $P_-$ separates
its projective cover from a complementary summand on which $f$ vanishes.
That complementary summand is precisely the decoration. If $\beta_0,\beta_1$ are the minimal projective
multiplicities of $M$ and $v$ is its decoration, then
$\beta_0=[\delta]_+,\ \beta_1+v=[-\delta]_+$.
Lemma~\ref{lem:completed-betti} identifies the dimensions of
$\coker\alpha_k$ and $\ker\alpha_k/\im\gamma_k$ with
$\beta_{0,k}$ and $\beta_{1,k}$, respectively. The local mutation formula gives
$\coker\alpha'_k\simeq \ker\alpha_k/\im\gamma_k\oplus V_k$.
Taking $K_k$-dimensions proves the assertion. The injective identities
follow from Proposition~\ref{prop:opposite-duality}.
\end{proof}

Let $B=(b_{ij})$ have $b_{ij}>0$ for arrows $i\to j$ and
$B\diag(d_i)$ skew-symmetric. The right $K_j$-ranks
of the outgoing arrow module and incoming dual module at $k$ are
$[b_{kj}]_+$ and $[-b_{kj}]_+$, respectively.

\begin{proposition}\label{prop:weight-rule}
Let $\mathcal M$ be the decorated cokernel of a projective presentation
with disjoint supports and $\delta$-vector $\delta$. Its mutation at $k$
has $\delta$-vector
\begin{equation}\label{eq:weight-p}
 (\mu_k\delta)_j=
 \begin{cases}
 -\delta_k,&j=k,\\
 \delta_j+[b_{kj}]_+[\delta_k]_+
       -[-b_{kj}]_+[-\delta_k]_+,&j\ne k.
 \end{cases}
\end{equation}
For an object given by an injective copresentation with disjoint supports, write
$\check\mu_{k,B}=\mu_{k,-B}$ for its injective weight transformation.
Equivalently,
\[ \check\mu_{k,B}(\check\delta)=\mu_{k,-B}(\check\delta). \]
This is \eqref{eq:weight-p} for the dual representation over the
opposite QP.
These statements specify the index of the mutated object;
its generality is established below.
\end{proposition}
\begin{proof}
For projective, respectively injective, presentations with disjoint
supports, Proposition~\ref{prop:disjoint-support-indices} gives
\begin{equation}\label{eq:disjoint-support-betti}
 \begin{aligned}
 h_k=-[\delta_k]_+,\quad h'_k=-[-\delta_k]_+
     &\quad\text{(projective presentation)},\\
 \check h_k=-[\check\delta_k]_+,\quad \check h'_k=-[-\check\delta_k]_+
     &\quad\text{(injective copresentation)}.
 \end{aligned}
\end{equation}
Insert the first pair in \eqref{eq:index-unsimplified} for the projective
formula. Apply Proposition~\ref{prop:opposite-duality} and the second pair
for the injective formula. 
\end{proof}

The coefficients $b_{kj}$ already include the vertex degrees through
the bimodule ranks. Both transformations are involutive when the matrix
is mutated as well. For a projective
presentation with disjoint supports the rank identities of Section~\ref{sec:local} read
\[ r'_k=\sum_i[-b_{ik}]_+r_i-r_k-\check h_k(M)+[-\delta_k]_+, \qquad r'_i=r_i\quad(i\ne k). \]
When the same decorated object also has an injective
copresentation with disjoint supports, then
$-\check h_k(M)=[\check\delta_k]_+$. Without that assumption we use
$-\check h_k(M)=\dim_{K_k}\ker\beta_k$.

\subsection{Rigid and reachable presentations}
\label{subsec:rigid}\label{subsec:reachable}
Recall that $\mathcal M$ is $E$-rigid when
$\op{E}(\mathcal M)=0$. 

\begin{lemma}\label{lem:isomorphism-descent}
Let $P$ be a finite-dimensional algebra over an infinite field $F$. If
finite $P$-modules $M,N$ become isomorphic after a field extension, they
are already isomorphic over $F$. The same holds for decorated modules,
provided the decoration dimensions agree. The resulting isomorphism need not be the specified isomorphism over
the larger field.
\end{lemma}
\begin{proof}
The intertwining conditions define the finite-dimensional $F$-space
$\Hom_P(M,N)$, and its defining linear equations commute with scalar
extension. In an $F$-basis $T_1,\ldots,T_q$ of that space, the determinant
of $\sum_j u_jT_j$ on the underlying coefficient spaces is a polynomial.
The given geometric isomorphism shows that it is not zero. Since $F$ is
infinite, it has a nonzero value at a tuple in $F^q$. This is an invertible
$P$-linear map. Decorations with the same dimensions can be compared
independently over their vertex fields.
\end{proof}

\begin{proposition}\label{prop:rigid-weight}
For a fixed finite ramified Jacobian algebra $P$, an $E$-rigid decorated
representation is determined up to $P$-isomorphism by its projective
weight. Its presentation in $\PHom(\delta)$ lies in the unique geometric open
orbit of the corresponding full presentation space. In particular, the
universal presentation of weight $\delta$ associated with any full
initial order has a decorated cokernel isomorphic to the scalar
extension of this representation. Their monodromy $F$-polynomials agree.
\end{proposition}
\begin{proof}
Write $\beta_0,\beta_1$ for the minimal projective multiplicities with
the decoration as the summand $P(v)\xrightarrow{0}0$. The lower bound
\eqref{eq:E-lower-bound} implies $\beta_{0,i}\beta_{1,i}=0$ for all $i$.
Thus these multiplicities are $[\delta_i]_+$ and $[-\delta_i]_+$.
Two $E$-rigid objects of weight $\delta$ therefore have presentations
in the same affine space $X_\delta$.

At a presentation $f$, the tangent map of the action by changes of
presentation bases is $(a,b)\longmapsto af-fb$.
Its cokernel is $\op{E}(\mathcal M,\mathcal M)$ by
\eqref{eq:E-two-presentations}. This cokernel vanishes after scalar
extension as well. In characteristic zero the orbit is therefore open
in the whole geometrically irreducible affine presentation space. Two
such open orbits coincide. The decorated cokernels are geometrically
isomorphic; Lemma~\ref{lem:isomorphism-descent} descends their
isomorphism over the fixed valued field.

The universal point with its Gauss valuation from any initial full order is algebraically
generic in $X_\delta$ by Proposition~\ref{prop:universal-invariance}; it
lies in that same open orbit. Compare it to the scalar extension of the
fixed rigid presentation using the same lemma. Extension from $K$ to
$\overline{\C(\mathbf z)}((t))$ preserves $t$ and its tame generator,
so the representation Grassmannians have the same inertia traces. This proves equality of the entire monodromy $F$-polynomial at $E$-rigid presentation
weights. Theorem~\ref{newton:thm:order-independence} will give equality
of Newton polytopes at every weight.
\end{proof}

We apply the preceding uniqueness statement to the collections obtained
from negative simples.
Fix the mutation tree of the chosen good Jacobi-finite nondegenerate realization.
For a node $\mathbf s$ and its sequence $w$ from the root, pull the negative
simples of its Jacobian algebra back to the root:
\[ \mathcal T_{i;\mathbf s}=\mu_w^{-1}E_i^-(\mathbf s),\qquad \mathcal T_{\mathbf s}=\bigoplus_{i=1}^n\mathcal T_{i;\mathbf s}. \]
These are representations over the fixed ramified vertex fields,
up to the prescribed reduction comparisons. A decorated representation is
\emph{reachable} here if it is a summand of a finite direct sum of one
such collection. 

\begin{proposition}\label{thm:reachable}
Every collection $\mathcal T_{\mathbf s}$ has the following properties.
\begin{enumerate}
\item Its $n$ summands are pairwise nonisomorphic and indecomposable, and
$\op{e}(\mathcal T_{\mathbf s})=0$. In particular, $\op{E}(\mathcal T_{i;\mathbf s},\mathcal T_{j;\mathbf s})=0$
for every ordered pair $(i,j)$.
\item The mutable presentation weights of its summands form a $\Z$-basis
of $\Z^n$. At every common mutation vertex the weights of a compatible
collection have a common sign, with zeros allowed.
\item It is maximal: if $\op{e}(\mathcal T_{\mathbf s}\oplus\mathcal N)=0$,
then $\mathcal N$ belongs to $\add\mathcal T_{\mathbf s}$,
the class of direct summands of finite direct sums of $\mathcal T_{\mathbf s}$.
\item Each reachable isomorphism class is the general decorated
cokernel of its weight. Two reachable objects with the same mutable
presentation weight are isomorphic over the fixed Jacobian algebra over $K$.
\end{enumerate}
These conclusions require neither rigidity of $S$ nor the seed reachability condition. 
\end{proposition}
\begin{proof}
The sum of all negative simples has $\op{e}=0$. Mutation preserves
$\op{e}(\mathcal M)$, direct sums, and nonzero indecomposables, and is involutive.
This proves (i), including nonisomorphism. Additivity of the nonnegative
dimensions of the spaces $\op{E}$ shows that every ordered cross term vanishes.

At each stage, Corollary~\ref{thm:sign-coherence} gives sign coherence
of the $k$th coordinates of the weights in the collection. On that collection
the piecewise-linear rule \eqref{eq:weight-p} is consequently one
integral linear transformation: it negates coordinate $k$ and adds
$[b_{kj}]_+$ times that coordinate when its sign is positive, or
$[-b_{kj}]_+$ times it when its sign is negative. Its determinant is
$-1$. Starting with the negative unit weights and composing these
unimodular transformations proves (ii).

For (iii), transport $\mathcal T_{\mathbf s}\oplus\mathcal N$ to
$\mathbf s$. It becomes $(\bigoplus_iE_i^-)\oplus\mathcal N'$, still
with $\op{e}=0$. Formula~\eqref{eq:E-pres} gives
$\op{e}(E_i^-,\mathcal N') =\dim_K\Hom(P_i,N')=d_i\dim_{K_i}N'_i$.
All these terms vanish, so the positive part of $\mathcal N'$ is zero.
It is a sum of negative simples. Mutating back proves the assertion.
Finally apply Proposition~\ref{prop:rigid-weight} to each reachable
object and to sums from a common collection. This proves (iv).
\end{proof}

Theorem~\ref{thm:cluster-model} identifies the summands of each such
collection with the variables of one cluster, and their direct sums
with cluster monomials.

\subsection{Generality under mutation}
To apply the character formula to every $\delta$-vector, we need
mutation to preserve generality, not only individual isomorphism
classes. We first describe how rational choices of presentations and
changes of bases recover a dense family. Theorem~\ref{thm:generic-species-generality}
then applies these facts to mutation.

A \emph{constructible subset} is a finite union of locally closed
subsets. A \emph{constructible stratification} is a finite partition
into locally closed subsets; we may refine a constructible partition
to obtain one. 

All parameter varieties below are of finite type over $K$.
An \emph{integral variety} is irreducible and reduced; it is
\emph{geometrically integral} if this remains true after extension
to an algebraic closure.
A rational map is \emph{dominant} if its image is dense; in that case
it induces an injection of the target function field into the source
function field. A rational choice of presentation need not be dominant.
We therefore also consider the map obtained by adjoining a general
change of bases in its two projective terms.

\begin{lemma}\label{lem:dominance}
Let $G_X,G_Y$ act on irreducible presentation spaces $X,Y$, respectively,
and let $X^\circ\subseteq X$ be a nonempty open. Suppose a constructible
correspondence $\Gamma\subseteq X^\circ\times Y$ has the following
geometric fibres. For every $x\in X^\circ$, the fibre $\Gamma_x$ is
one entire $G_Y$-orbit of codimension $e$ in $Y$. For every $y$ in the
second projection, the reverse fibre $\Gamma^y$ is the intersection
with $X^\circ$ of one $G_X$-orbit of codimension $e$ in $X$.
Then the second projection of $\Gamma$, which is already orbit-saturated,
is dense in $Y$ and contains a nonempty open.
\end{lemma}
\begin{proof}
The first projection is surjective with fibres of dimension $\dim Y-e$,
so $\dim\Gamma=\dim X+\dim Y-e$. A nonempty intersection of a source
orbit with $X^\circ$ is open in that orbit; hence every nonempty fibre
of the second projection has dimension $\dim X-e$. The fibre-dimension
formula gives a target image of dimension $\dim Y$. The image is
constructible, so irreducibility of $Y$ implies that it contains a
nonempty open.
\end{proof}

For $f:P_-\to P_+$, the orbit codimension is
\[ \op{e}(f,f)=\dim_K\coker\bigl( \End(P_+)\oplus\End(P_-)\longrightarrow\Hom(P_-,P_+) \bigr),\qquad (a,b)\longmapsto af-fb. \]
Thus inverse orbit correspondences preserving $\op{e}(f,f)$ give dominance.
For an irreducible component of an ordinary decorated representation
variety, write $e_{\rm orb}$ for the generic codimension of a vertex-basis
orbit inside that component. Such a component is \emph{generically
$\tau$-reduced} when $e_{\rm orb}$ equals the generic value of $\op{e}(\mathcal M,\mathcal M)$ in the
projective convention; \emph{generically $\tau^{-}$-reduced} uses the
injective convention, or the opposite dual. Here $\tau$ denotes Auslander--Reiten translation, recalled in
Appendix~\ref{app:AR}. The term concerns the stated dimension equality,
rather than scheme-theoretic reducedness. The ordinary component mutation results
\cite[Theorems 5.18 and 5.20]{GLFS} use the latter convention.

\begin{lemma}\label{lem:descent}
Let $P$ be a finite $K$-algebra with separable vertex algebra
$R=\prod_iK_i$, and let $E/K$ be a finite Galois splitting field.
For $\delta\in\Z^n$, put $\delta^{\rm sp}_{i,a}=\delta_i$ on every
sheet above $i$. Then
\[ E\otimes_K\PHom_P(\delta) =\PHom_{E\otimes_KP}(\delta^{\rm sp}), \qquad E\otimes_K\End_P(P_\pm)=\End_{P_E}((P_\pm)_E). \]
The corresponding unit groups base-change as well. If a rational choice of target presentations defined over $K$ has
a dense union of target orbits after scalar extension, that union is
already dense over $K$. Finitely many required rank opens, defined over the chosen finite
splitting field, can be replaced by nonempty $K$-defined opens.
\end{lemma}
\begin{proof}
Hom out of a finite projective commutes with scalar extension. Also
$(e_iP)_E$ is the direct sum of the indecomposable projectives over all
sheets above $i$, proving the first assertions; invertibility is an open
condition in the endomorphism algebra. For density, adjoin a general element of the product of the two target
automorphism groups. The resulting rational map is defined over $K$.
Dominance is preserved and reflected by field extension, so dominance
after extension to $E$ proves the assertion over $K$. For a required principal rank open over $E$, multiply its nonzero
defining polynomial by all Galois conjugates. The product is a nonzero
$K$-polynomial. Intersect the finitely many resulting opens.
\end{proof}

We use this descent for the presentation spaces and their rank opens.
The represented objects are subsequently evaluated over the fixed
valued fields of Section~\ref{subsec:valued-families}.

\begin{lemma}\label{lem:cover-charts}
Let $P$ be a fixed finite ramified Jacobian algebra. A rational family of finite
$P$-modules in fixed free vertex modules over a geometrically integral
coefficient chart admits, on a nonempty open where its first two Betti
multiplicities are constant, a rational family of minimal projective
presentations. A decoration of rank $v$ is represented by adjoining $P(v)\xrightarrow{0}0$. When the generic decorated presentation has disjoint supports, this
family takes values in its fixed full presentation space.
\end{lemma}
\begin{proof}
Let $J=\rad P$, and choose a finite multiplication table and bases of
each bimodule $e_iJe_j$. The action map $M\otimes_RJ\to M$ has image $MJ$. On the open
subset where a chosen maximal minor is nonzero, its matrix gives
rational bases of $MJ$ and $M/MJ$. Lifts of a basis of $M/MJ$ define
the projective cover $P(\beta_0)\to M$, which is surjective by
Nakayama's lemma. Its kernel also has a rational basis on this open
subset. Applying the same construction to the top of that kernel gives
the minimal relation cover $P(\beta_1)\to P(\beta_0)$.

These operations are vertexwise over the finite separable fields $K_i$.
The norms of the finitely many chosen nonzero minors define a
nonempty open subset over $K$. Add the decoration as the zero map from its
projective source summand. The source and target multiplicities obtained this way are the minimal
Betti multiplicities; the construction gives no further disjointness assertion. Finally, for the local mutation family, finite
dimension of the Jacobian algebra truncates all arrow evaluations, so these are finite
rational matrices even when $S$ is a formal potential.
\end{proof}

\begin{lemma}\label{lem:orbit-open}
Let $X,Y$ be geometrically integral varieties over $K$, let $G$ act on
$Y$, and let $q:X\dashrightarrow Y$ be rational. If
$\Phi:X\times G\dashrightarrow Y$, $(x,g)\mapsto gq(x)$, is dominant,
then $q^{-1}(U)$ is a nonempty open in the domain of $q$ for every
nonempty $G$-invariant open $U\subseteq Y$. For an arbitrary nonempty
open $U$, the source meets its saturation $GU$. Thus finitely many successive invariant open conditions may be pulled
back to the initial presentation space.
\end{lemma}
\begin{proof}
For an invariant open,
$\Phi^{-1}(U)=q^{-1}(U)\times G$ on the domain of $q$.
Dominance makes this inverse image nonempty; rationality makes it open.
For a non-invariant open replace it by $GU$, a nonempty invariant open.
For iteration, take the invariant open subsets on which the ranks and
Betti multiplicities have their generic values, or the orbit saturation
of the union of the rational presentation charts. Pulling them back
successively gives finitely many dense open subsets of the source.
Their intersection is nonempty; on it the required charts may be chosen.
\end{proof}

\begin{lemma}\label{lem:finite-data}
Fix a finite mutation sequence through Jacobi-finite QPs and a presentation
weight. After a finite separable splitting, the finite Jacobian algebras,
local mutation matrices, and any chosen finite collection of rational
projective-presentation charts admit models over a finitely generated
characteristic-zero subfield. This includes the finite expressions of the relations
and the nonvanishing minors used by the charts. Nonemptiness of their
geometric rank conditions and dominance of the maps obtained by
adjoining changes of presentation bases may be tested after embedding
that subfield in $\C$. All quotients defining the Jacobian algebras
in this statement are completed quotients.
\end{lemma}
\begin{proof}
At one seed write $T=F\langle\!\langle Q\rangle\!\rangle$ and
choose $r$ with $\m^r\subseteq J(S)$. Let $S_{\le d}$ be the potential
truncated at degree $d\ge r+1$. The derivative tails lie in $\m^d$,
so the two derivative ideals have the same image modulo $\m^{r+1}$.
In particular,
$\m^r\subseteq J(S_{\le d})+\m^{r+1}$.
Lemma~\ref{alg:lem:toplayer} gives $\m^r\subseteq J(S_{\le d})$.
The derivative tails now lie in both closed ideals, and hence
$J(S_{\le d})=J(S)$.
The inclusion modulo $\m^{r+1}$ is expressed by finitely many linear
identities in the truncated tensor algebra. Those identities and a
finite potential jet describe the completed finite algebra over a
field generated by finitely many coefficients. 

For the fixed sequence choose a nilpotence cutoff at each of its finitely
many seeds. Evaluation of arrow series and their local derivative
matrices then uses only finite jets. Each prescribed finite reduction
jet likewise uses finitely many coefficients; substitution of composite arrows, changes of arrows,
and inversion of a map with invertible linear part require only a
sufficiently high jet of the original map. The partial quadratic splittings can
be completed without changing these selected jets. Projective-cover and
kernel charts use a finite multiplication table and finitely many
Gaussian eliminations, as in Lemma~\ref{lem:cover-charts}.
Adjoin all these coefficients, the coefficients expressing the relations, and the
chosen nonzero minors to $\Q$. This gives a finitely generated
field $F_0\subseteq F$ over which the finite calculation is defined.
The nilpotence relations remain valid over $F_0$ by the same
closed-ideal argument.

Representation and presentation spaces, their basis-change groups, and
the rational charts with their changes of bases are defined by these
finite algebraic data. After
passing to algebraic closures, dimensions, ranks, and dominance are
preserved by scalar extension. Embedding $F_0$ in $\C$ therefore permits
application of the ordinary component and presentation theorems to the
same finite calculation. Only the finite algebraic data are extended; the valued coefficient
field remains fixed.
\end{proof}

The next theorem concerns general represented objects, rather than the
coordinates of a particular normal form. We organize its proof by the
diagram
\[
\vcenter{\hbox{$\displaystyle \xymatrix@C=30pt@R=22pt{
 X_\delta\ar@{-->}[r]^q\ar[d]&X_{\mu_k\delta}\ar[d]\\
 \{\text{represented source classes}\}\ar[r]_{\mu_k}&
 \{\text{represented target classes}\}
}$}}.
\]
Here $q$ chooses a presentation rationally, and the vertical
maps take its decorated cokernel up to isomorphism. The lower horizontal
map is mutation of these isomorphism classes; no geometric quotient is
needed. The dominant map is
\[ X_\delta\times G_{\mu_k\delta}\dashrightarrow X_{\mu_k\delta}, \qquad(f,g)\longmapsto gq(f). \]
The proof first applies ordinary component mutation to the full
presentation space after splitting. Lemma~\ref{lem:finite-data}
justifies the characteristic-zero field comparison, and
Lemma~\ref{lem:descent} descends orbit-saturated dominance to $K$.
Finally, Lemma~\ref{lem:orbit-open} pulls the invariant rank conditions
back to the initial presentation space. These are statements about
represented objects; no dominance of a chosen normal form is asserted.

\begin{theorem}\label{thm:generic-species-generality}
Let $(A,S)$ be a Jacobi-finite QP over $R=\prod_iK_i$. Fix a finite sequence
whose successive reduced mutations are defined and two-acyclic at the
original vertices. Mutation takes a general projective
presentation of weight $\delta$, up to changes of presentation bases,
to a general presentation of weight $\mu_w\delta$ given by
\eqref{eq:weight-p}. The dual assertion holds for injective weights.
The requisite Betti and rank conditions hold on a nonempty $K$-defined
open of the initial presentation space.  Under
Hypothesis~\ref{hyp:initial-realization}, this applies to every finite
mutation sequence.
\end{theorem}
\begin{proof}
\emph{Ordinary presentations and component mutation.}
Over an algebraically closed field of characteristic zero, the full
projective-presentation space $\PHom(\delta)$, with projective source summands on which the differential vanishes recorded as decoration, parametrizes a generically
$\tau$-reduced decorated component. For the positive part without
decoration, this is the \emph{dually strongly reduced} convention of
\cite{Plamondon}. The converse passage uses generic minimal projective
covers with disjoint supports; the decoration is treated below. These are the
presentation--component correspondence in \cite[Section 2.6 and
Corollary 2.17]{Plamondon}, with the presentation framework of \cite{DF}.
Regular local mutation charts are supplied by \cite[Theorem 5.18]{GLFS};
\cite[Theorem 5.20]{GLFS} identifies the closure of their target
orbit saturation as a generically $\tau^{-}$-reduced component
and makes component mutation involutive. In that theorem, density
concerns the union of target orbits, not the chosen map of representation
spaces alone.

GLFS uses the injective convention. Apply its statements to the opposite
QP and dualize. A projective presentation $P_-\to P_+\to M$ becomes
$\mathbb DM\to\mathbb DP_+\to\mathbb DP_-$, and the injective weight
in our convention is $\delta$; the usual injective $g$-vector is
$-\delta$. The decoration is carried to the opposite algebra as well. The
injective copresentation used here is that of $\mathbb DM$, obtained
by dualizing the projective presentation of $M$.

At the target, the converse presentation--component correspondence gives
disjoint projective supports. To include decoration, observe that it
does not change orbit dimensions. On a generically reduced decorated
component we have
\[ e_{\rm orb}=\op{e}((M,V))=\op{e}((M,0))+\sum_i v_i\dim M_i,\qquad e_{\rm orb}\le \op{e}((M,0)). \]
The last inequality is the component inequality of
\cite[Remark 5.6]{GLFS}, applied to the positive part and, when needed,
to the opposite algebra. It follows that $v_i\dim M_i=0$ and
$e_{\rm orb}=\op{e}((M,0))$. The positive component therefore has disjoint
minimal projective supports by \cite[Corollary 2.17]{Plamondon}.
If $v_i>0$, then $M_i=0$, so its projective cover has no summand $P_i$.
Adjoining $P(v)\xrightarrow{0}0$ therefore preserves disjointness of the source and target supports. The index is
computed by Proposition~\ref{prop:weight-rule}.

Lemma~\ref{lem:finite-data} transfers these ordinary results to an
algebraic closure of $K$: for the fixed sequence and weight, descend the
finite completed-algebra calculation to a finitely generated
characteristic-zero field, embed that field in $\C$, and extend back.
Only these finite data are descended; the formal valued field and its
tame generator are not replaced.

\emph{Scalar extension and descent.}
Choose a finite Galois splitting field $E/K$. Lemma~\ref{lem:descent}
gives the entire ordinary presentation space, not a space of
Galois-invariant matrices:
$E\otimes_K\PHom_P(\delta)=\PHom_{P_E}(\delta^{\rm sp})$ with $\delta^{\rm sp}_{i,a}=\delta_i$.
It also identifies the groups of changes of presentation bases.
The mutations at the sheets over $k$ commute. There are no arrows
between those sheets, and mutation at one creates none incident to
another sheet of the same orbit. Their product is the species
kernel--image construction after scalar extension. The source weight is constant on the sheets over each vertex. Since
there are no arrows between the sheets over $k$, mutation at one does
not change the weights at the others. The common direction of the
incident arrows therefore lets us sum their ordinary contributions:
we obtain $[b_{kj}]_+$ and $[-b_{kj}]_+$ at every sheet over $j$.
The resulting weight is constant on each original orbit and obeys
\eqref{eq:weight-p}.

Lemma~\ref{lem:cover-charts} constructs rational projective
presentations of the mutated modules over $K$. Over the splitting
field, the union of the target orbits meeting their image is dense by
ordinary mutation theory. Lemma~\ref{lem:descent} gives the same density over $K$. Products of
the nonvanishing minors over the finite Galois group define the required
rank opens over $K$.

\emph{Iteration.}
Lemma~\ref{lem:orbit-open} pulls the invariant rank and Betti
conditions back through each successive mutation. The intersection of
the resulting dense source opens proves the assertion for the entire
sequence. Adjoining a general change of target presentation bases gives
algebraically generic target coordinates, as in
Section~\ref{sec:arithmetic-transport}. For a valued source presentation, its mutation matrices can be evaluated
in the same valued field. Adjoining independent changes of presentation
bases is needed only when a representative with algebraically generic
target coordinates is desired. Opposite duality proves the injective statement.
\end{proof}

\subsection{Initial orders and valued universal presentations}
\label{subsec:valued-families}

A \emph{discrete valuation} on a field $F$ is a surjective map
$v:F^\times\to\Z$ with $v(ab)=v(a)+v(b)$ and
$v(a+b)\geq\min(v(a),v(b))$, with $v(0)=+\infty$.
Its valuation ring is $\OO_v=\{v\geq0\}$, its maximal ideal is
$\mathfrak m_v=\{v>0\}$, and its residue field is
$\kappa(v)=\OO_v/\mathfrak m_v$. An element of value one is a
uniformizer. We normalize $v_t(t)=1$.

Fix a full order $\PP_0\subset\mathcal P$ containing $\RR$.
For a nonnegative vector $a$, write
$\PP_0(a)=\bigoplus_i(e_i\PP_0)^{a_i}$. For a weight $\delta$, put
\[ \PHOM(\delta)=\Hom_{\PP_0} \bigl(\PP_0([-\delta]_+),\PP_0([\delta]_+)\bigr). \]
We write $\PHOM_{\PP_0}(\delta)$ when the order must be specified.
The calligraphic $\mathscr P$ indicates the integral lattice; its
generic fibre is the presentation space $\PHom(\delta)$.
This lattice is a finite free $\OO$-module, and
\begin{equation}\label{eq:full-presentation-lattice}
 \PHOM(\delta)\otimes_{\OO}K=\PHom(\delta).
\end{equation}
Both assertions follow by writing the projectives as sums of
$e_i\PP_0$: the Hom modules are finite direct sums of the free
$\OO$-modules $e_i\PP_0e_j$.

Choose an integral basis $f_1,\ldots,f_q$ of $\PHOM(\delta)$
and form $f_{\mathrm{univ}}=\sum_jz_jf_j$. On $K(z_1,\ldots,z_q)$
use the Gauss valuation
\begin{equation}\label{eq:Gauss}
 v\left(\sum_\alpha c_\alpha z^\alpha\right)
       =\min_\alpha v_t(c_\alpha).
\end{equation}
Extend to rational functions by $v(a/b)=v(a)-v(b)$.
Its residue field is $\C(z_1,\ldots,z_q)$; the images of the $z_i$
there are algebraically independent.  After completion and
extension of the residue field to an algebraic closure, we work over
$L_\delta=\overline{\C(z_1,\ldots,z_q)}((t))$.
The decorated cokernel of $f_{\mathrm{univ}}$ there is denoted
$\mathcal M_{\delta,\mathrm{univ}}$. Thus $P$ and $\delta$ determine the
function field $K_\delta^{\mathrm{pres}}$, while the order selects its
Gauss valuation through the integral lattice $\PHOM(\delta)$.
Proposition~\ref{prop:coefficient-extension} proves that the monodromy
$F$-polynomial is unchanged by these compatible unramified scalar
extensions.

\begin{proposition}\label{prop:universal-invariance}\label{lem:generic-valued-embedding}
The valued universal presentation is independent of the integral basis,
up to the induced valued-field isomorphism, and is preserved by a
prescribed isomorphism of orders and QPs over $K$. Moreover,
$K(z_1,\ldots,z_q)\to L_\delta$ is injective. The induced algebraically generic point belongs to every nonempty
$K$-defined open in the presentation space.
\end{proposition}
\begin{proof}
Normalize a nonzero polynomial by the least valuation of its
coefficients. Its residue is a nonzero polynomial in algebraically
independent variables. It remains nonzero after an integral invertible
linear coordinate change, since the residue change is invertible.
Applying the inverse coordinate change gives equality of the two Gauss
valuations. The same argument proves injectivity into $L_\delta$.
Every nonempty open in the affine presentation space contains a
principal open defined by the nonvanishing of a nonzero polynomial.
The algebraically generic point belongs to that open. An order isomorphism identifies
the two integral Hom lattices and their universal maps.
\end{proof}

A different Hom lattice can select a different valuation; for example
$\OO f_1\oplus\OO f_2$ and $\OO f_1\oplus t\OO f_2$ do so.
Section~\ref{sec:order-sensitivity} studies this dependence and the
valuation transported under mutation.

\begin{corollary}\label{thm:fixed-field-generality}
Under Hypothesis~\ref{hyp:initial-realization}, keep $L_\delta$ and its
ramified vertex extensions fixed. Along every finite mutation sequence, mutation of $\mathcal M_{\delta,\mathrm{univ}}$ is general
of weight $\mu_w\delta$, in the orbit-saturated sense of
Theorem~\ref{thm:generic-species-generality}.
\end{corollary}
\begin{proof}
For a fixed sequence, that theorem and Lemma~\ref{lem:orbit-open} give
a nonempty open defined over $K$ of the initial presentation space on which
the required invariant rank and Betti conditions hold at every prefix.
It contains the initial universal point by
Proposition~\ref{prop:universal-invariance}. Choose rational presentation charts successively on this open subset.
Their maps and the prescribed QP reductions are evaluated in
$L_\delta$ and its fixed vertex extensions. The global QP condition is
Corollary~\ref{thm:good-global-pattern}. Every finite sequence has such an
initial open, and the same algebraic universal point lies in all of
them, so the assertion holds for every finite sequence.
\end{proof}

\paragraph{\em Auxiliary changes and the initial order.}
The initial order is fixed throughout the comparisons below. With
the prescribed comparison maps and compatible tame generators, the
listed auxiliary changes preserve the indicated data. Changing the
initial order is a different operation.
\begin{center}
\small\renewcommand{\arraystretch}{1.12}
\begin{tabular}{@{}>{\raggedright\arraybackslash}p{.26\linewidth}>{\raggedright\arraybackslash}p{\dimexpr.74\linewidth-2\tabcolsep\relax}@{}}
\toprule
Auxiliary change & Data preserved \\\midrule
$R$-right-equivalence & The vertex algebra $R$; the potential up to cyclic equivalence
and the representation through the induced isomorphism of Jacobian algebras. \\
Reduction & The initial order through its image in the reduced Jacobian
algebra, with the prescribed comparison map.\\
Integral coordinate change & The integral Hom lattice and its Gauss valuation.\\
Unramified scalar extension & The normalization $v(t)=1$, the compatible
tame generator, and the monodromy traces.\\\bottomrule
\end{tabular}
\end{center}

\part{Monodromy Characters and Cluster Bases}\label{part:characters}
The valued representations of Part~\ref{part:ramified} now give
monodromy $F$-polynomials. Their mutation formula leads to the cluster
basis theorem. We then study factorization, geometric orders, and the
arithmetic data carried by mutation.

\section{Ramified \texorpdfstring{$F$}{F}-polynomials and mutation}\label{sec:monodromy}
We attach a polynomial to each finite ramified representation, with
quotient ranks as exponents and monodromy traces as coefficients.
The mutation formula is followed by Example~\ref{run:ex:one-arrow},
the Newton-polytope results, and fixed-field transport.

\subsection{Monodromy traces and ramified \texorpdfstring{$F$}{F}-polynomials}

Let $K=\C((t))$ and $R=\prod_iK_i$, with
$K_i=K[\varpi_i]/(\varpi_i^{d_i}-t)$. Fix compatible roots of $t$ in
$\overline K$ and the tame generator $\sigma$ sending $t^{1/m}$
to $\exp(2\pi\sqrt{-1}/m)t^{1/m}$. It cyclically permutes the
$d_i$ embeddings of $K_i$. We use these choices throughout the construction for a finite-dimensional
decorated representation $\mathcal M=(M,V)$.

For a prime $\ell$, let $\Q_\ell$ denote the $\ell$-adic coefficient
field and $H_c^j$ compactly supported \'etale cohomology in degree $j$.
For a separated finite-type $K$-scheme $X$, put
\[ \chis(X)=\sum_j(-1)^j \Tr\bigl(\sigma\mid H_c^j(X_{\overline K},\mathbb Q_\ell)\bigr). \]
The cohomology localization sequence gives trace
additivity for Galois-invariant decompositions. The K\"unneth formula,
which identifies the cohomology of a product with the graded tensor
product, gives multiplicativity.
The function $\chis$ is integer-valued and independent of $\ell$ by
\cite[Corollary 6.5]{Nicaise}. For convergent algebraic families it
is the alternating trace of ordinary topological monodromy; the
cohomological definition also allows formal coefficients.

For a finite field extension $E/F$, restriction of scalars is the
functor $(\Res_{E/F}X)(T)=X(T\times_F E)$ on $F$-schemes $T$.
It is represented for the Grassmannians used here.
Set $\mathbf r_M=(r_i)_i$, where $r_i=\dim_{K_i}M_i$.
For $0\leq\mathbf a\leq\mathbf r_M$, the
\emph{representation Grassmannian} $\Gr^R_{\mathbf a}(M)$ parametrizes
subrepresentations $N\subseteq M$ with $\dim_{K_i}N_i=a_i$. It is
the closed subscheme of
\[ \prod_i\Res_{K_i/K}\Gr_{a_i}(M_i) \]
cut out by preservation under every arrow action. Over a base scheme,
the $N_i$ are subbundles with locally free quotients; arrow stability
means that each action map $N_i\otimes_{K_i}A_{ij}\to M_j/N_j$ vanishes,
after extension of scalars to that base.
This construction gives a projective scheme. For an ordinary quiver,
it is the usual \emph{quiver Grassmannian}, with its natural scheme
structure.

The index $\mathbf a$ in $\Gr^R_{\mathbf a}(M)$ records submodule ranks.
For the $F$-polynomial we use the complementary quotient ranks
$\mathbf e=\mathbf r_M-\mathbf a$. Its coefficient at
$\mathbf y^{\mathbf e}$ is therefore the trace on
\[ \Gr^R_{\mathbf r_M-\mathbf e}(M)\subseteq\prod_i\Res_{K_i/K}\Gr_{r_i-e_i}(M_i). \]
After splitting the vertex fields, the submodule has dimension
$r_i-e_i$ and the quotient has dimension $e_i$ on every sheet over
$i$. The Grassmannians in the definition of $F_M$ have these balanced ranks.
When we discuss their points as quotients, we use this complementary-rank
identification. In \eqref{eq:Fp-mon}, the sum is over
$0\leq\mathbf e\leq\mathbf r_M$ componentwise.

\begin{definition}\label{def:monodromy-F}
We define the \emph{ramified $F$-polynomial} by
\begin{equation}\label{eq:Fp-mon}
 F_{\mathcal M}(\mathbf y)=F_M(\mathbf y)
   =\sum_{\mathbf e}\chis\bigl(\Gr^R_{\mathbf r_M-\mathbf e}(M)\bigr)
                           \mathbf y^{\mathbf e}.
\end{equation}
The polynomial depends only on $M$. The decoration affects the mutated
positive module and hence the target $h$-vector and $F$-polynomial.
\end{definition}

This is a finite polynomial with integer coefficients. Its constant
coefficient and the coefficient of $\mathbf y^{\mathbf r_M}$ are one;
positivity of the other coefficients is not assumed. The algebra may
be indicated by a subscript when several algebras are in use. We
write $F^{\mathrm{op}}_{\mathbb DM}$ for the same construction on
the dual representation over $\mathcal P^{\mathrm{op}}$.
Opposite duality gives
\begin{equation}\label{eq:F-opposite-duality}
 F^{\mathrm{op}}_{\mathbb DM}(\mathbf y)
   =\sum_{\mathbf e}\chis\bigl(\Gr^R_{\mathbf e}(M)\bigr)
                             \mathbf y^{\mathbf e},
\end{equation}
because $N\subseteq M$ dualizes to the quotient
$\mathbb DM\twoheadrightarrow\mathbb DN$. This is an isomorphism
of schemes over $K$, hence respects monodromy. Complementary ranks
then give
\begin{equation}\label{eq:mon-reciprocity}
 F_M(\mathbf y)=\mathbf y^{\mathbf r_M}
                F^{\mathrm{op}}_{\mathbb DM}(\mathbf y^{-1}).
\end{equation}
Thus submodule counting is the quotient construction on the opposite
algebra. A right-equivalence carrying the represented module identifies
the representation Grassmannians over $K$, hence their Galois actions and monodromy $F$-polynomials.

A valuation ring is \emph{henselian} if every simple root of a polynomial
modulo its maximal ideal lifts to a root in the ring. It is
\emph{strictly henselian} if, in addition, its residue field is separably closed. A completed strict henselization first
adjoins the unramified residue extensions and then completes. In residue
characteristic zero the resulting residue field is algebraically closed.

\begin{proposition}\label{prop:coefficient-extension}
Let $\Omega\subseteq\Omega'$ be algebraically closed extensions of
$\C$, and put $K_\Omega=\Omega((t))$ and
$K_{\Omega'}=\Omega'((t))$. Choose compatible roots of $t$ and the
generators $\sigma_\Omega,\sigma_{\Omega'}$ sending $t^{1/e}$ to
$\zeta_e t^{1/e}$, where $\zeta_e=\exp(2\pi\sqrt{-1}/e)$. For every separated finite-type $K_\Omega$-scheme $X$,
$\chi_{\sigma_\Omega}(X) =\chi_{\sigma_{\Omega'}}(X_{K_{\Omega'}})$.
Formation of representation Grassmannians with balanced ranks commutes with this extension,
so their monodromy $F$-polynomials agree. The same comparison holds for
valuation-preserving embeddings of completed strictly henselian
discretely valued fields of residue characteristic zero when $t$ remains
a common uniformizer and the compatible roots of $t$ are fixed.
\end{proposition}
\begin{proof}
Over each such field, all finite extensions are tame and are obtained by
adjoining roots of $t$: units have all roots by Hensel's lemma and the
residue field is algebraically closed. The generator on the larger
field therefore restricts to the prescribed generator on the smaller
one. Choose compatible algebraic closures. Invariance of geometric
compact-support \'etale cohomology under algebraically closed field
extension identifies the cohomology spaces equivariantly for these
generators, which proves the trace equality. Representation Grassmannians commute
with field extension by their moduli description, or directly by their
Grassmannian stability equations. Apply the same argument to each
balanced rank. The argument for completed strictly henselian fields is
identical. A ramified scalar extension with $t=u^a$, $a>1$, would instead
restrict the generator to $\sigma^a$ and is not covered by the statement.
\end{proof}

The proposition applies in particular to every Gauss field $L_\delta$.

\begin{lemma}\label{lem:trace-strata}
Let a separated finite-type $K$-scheme have a finite geometric constructible
stratification, defined over a finite extension and permuted by $\sigma$.
Its alternating compact-support trace is the sum of the traces of the
strata fixed by $\sigma$. A nontrivial orbit of strata contributes zero.
A fixed locally closed stratum, with its reduced induced structure,
descends to $K$. The assertion applies to stratifications
of relative representation Grassmannians by their ranks on the split vertex sheets.
\end{lemma}
\begin{proof}
Taking the underlying reduced strata leaves their \'etale cohomology unchanged. Refine the
stratification by the finitely many Galois translates of its closures,
using invariant closed subsets and their complements. Localization
triangles give an equivariant finite filtration for computing the
alternating compact-support class. For an orbit of $d>1$ pieces its
associated summand is a direct sum $W_0\oplus\cdots\oplus W_{d-1}$ in
each cohomological degree, with $\sigma$ carrying $W_i$ to $W_{i+1}$.
Each $W_i$ is carried into a different summand, so its projection
back to $W_i$ is zero. Hence the trace is zero, regardless of the action
of $\sigma^d$ on $W_i$. Trace additivity through
the filtration proves the formula for the original stratification.
For a fixed piece the stabilizer is open and contains the topological
generator, hence is the whole Galois group; its defining constructible
conditions descend. Rank strata are obtained by finite determinantal
conditions, so they satisfy these hypotheses.
\end{proof}

We use additivity of alternating cohomology classes, not a decomposition
of the cohomology of the union. The fixed-stratum formula is for the
chosen generator; a power of $\sigma$ can fix additional strata.

\begin{lemma}\label{lem:relative-grassmannian}
Let $\Kram_d=K(t^{1/d})$, let $Z$ be a separated finite-type $K$-scheme,
and let $\mathcal U$ be a rank-$m$ vector bundle on $Z_{\Kram_d}$.
For $0\le e\le m$, we have
\begin{equation}\label{eq:relative-trace}
 \chis\bigl(\Res_{Z_{\Kram_d}/Z}\Gr_e(\mathcal U)\bigr)
       =\binom me\chis(Z).
\end{equation}
\end{lemma}
\begin{proof}
Put $W=\Res_{Z_{\Kram_d}/Z}\Gr_e(\mathcal U)$. After splitting $\Kram_d/K$,
the proper morphism $W\to Z$ is a product of $d$ Grassmannian bundles.
Let
\[ \mathsf V_{e,m} =\bigoplus_{\lambda\subseteq e\times(m-e)} \mathbb Q_\ell(-|\lambda|)[-2|\lambda|], \]
where $\lambda\subseteq e\times(m-e)$ is a partition with at most
$e$ parts, each at most $m-e$, and $|\lambda|$ is their sum.
The associated \emph{Schubert classes} form a basis of Grassmannian cohomology
in even degrees $2|\lambda|$. The shift $[-2a]$ puts a class in degree
$2a$. The \emph{Tate twist} $\Q_\ell(-a)$ is the one-dimensional
representation defined by the inverse $a$th power of the cyclotomic
character; it is trivial on inertia over our fields.
The \emph{tautological bundle} has fibre the subspace parametrized by a
Grassmannian point. The \emph{Schur classes} of its dual, indexed by the
partitions above, give the Schubert cohomology basis. They are polynomials
in the Chern classes of that dual bundle and are independent of local
trivializations. They give the following equality of alternating
cohomology classes:
\begin{equation}\label{eq:relative-class}
 \bigl[\mathrm R\Gamma_c(W_{\overline K},\mathbb Q_\ell)\bigr]
 =\bigl[\mathrm R\Gamma_c(Z_{\overline K},\mathbb Q_\ell)
                   \otimes\mathsf V_{e,m}^{\otimes d}\bigr].
\end{equation}
Here $\mathrm R\Gamma_c$ is the compact-support cohomology complex.
Its class means the alternating sum of its cohomology representations;
classes add in exact sequences, without asserting a direct-sum splitting.
On the right, $\sigma$ acts on the base cohomology and cyclically
permutes the $d$ Schubert factors. This follows equally from the
proper Leray filtration: the relative Schur classes identify its
cohomology sheaves and are compatible with local trivializations and
permutations of the embeddings. The resulting cohomology-class identity
therefore descends to $K$.

In the tensor product of the Schubert cohomology bases, every
nonconstant tuple belongs to a nontrivial cyclic orbit. The span of
such an orbit has trace zero under the cyclic permutation. There are
$\binom me$ constant tuples.
All Schubert degrees are even, so permuting factors introduces no sign;
the Tate twists have trivial inertia. Each constant tuple contributes $\chis(Z)$ to
\eqref{eq:relative-class}. Taking alternating traces proves the
formula.
\end{proof}

For the representation supported only at vertex $i$, with space $K_i^m$,
this gives $F_{K_i^m}=(1+y_i)^m$.
Ordinary geometric Euler characteristic would give the rank-$e$
coefficient $\binom me^{d_i}$. The cyclic trace gives $\binom me$.

\begin{example}\label{run:ex:first-exchange}
For the initial quiver of Example~\ref{run:ex:path}, the positive simple
at vertex 1 has presentation weight $\delta=(1,-2,0)$: its projective
cover is $P_1$ and its relation cover is $P_2^{\oplus2}$. Its quotient
polynomial is $1+y_1$. Thus the associated Laurent expression, defined as a cluster character in
Section~\ref{sec:cluster-application}, is
$x_1^{-1}x_2^2(1+x_2^{-2})=\frac{x_2^2+1}{x_1}$.
This is the exchange relation $x_1x_1'=1+x_2^2$. The exponent two
is the right $K_2$-dimension of the arrow bimodule.
\end{example}

\begin{proposition}\label{prop:mon-directsum}
For finite ramified representations,
$F_{M\oplus N}=F_MF_N$.
\end{proposition}
\begin{proof}
Stratify kernels in $M\oplus N$ by their intersection with $M$
and projection to $N$. Nonconstant sheet-rank patterns have zero
alternating trace by Lemma~\ref{lem:trace-strata}. For balanced
patterns, the base is a product
of kernel Grassmannians. Over fixed kernels, the possible graphs
form the vector space of representation maps from the projected
submodule to the quotient of $M$. On constant-rank strata these
are vector bundles. A rank-$a$ affine fibre has compact-support
cohomology $\Q_\ell(-a)[-2a]$, whose alternating inertia trace is one.
The quotient rank is the sum of the two complementary ranks.
Additivity gives this product.
\end{proof}

If every $d_i=1$ and the QP and representation are extended from
$\C$, inertia acts trivially and $F$ is the ordinary quotient
Euler-characteristic polynomial. The condition $d_i=1$ alone is
insufficient. The variety $\Spec K[s]/(s^2-t)$ has two geometric
points but trace zero; it is the $(1,1)$-Grassmannian of the
Kronecker representation
$a=I_2$, $b=\left(\begin{smallmatrix}0&t\\1&0\end{smallmatrix}\right)$.
This distinguishes ramification of the representation from that
of the vertex fields.

\subsection{The mutation formula}
The source and target representations have the same restriction to
the common subalgebra $P_{\widehat k,\widehat k}$. This restriction
includes paths passing through $k$, not only arrows away from $k$.
Over a fixed quotient of it, the possible kernels at $k$ form an
interval of subspaces. We compare the two intervals after splitting
the vertex fields. Nonconstant sheet-rank tuples occur in nontrivial
monodromy orbits and contribute trace zero; constant tuples give
relative Grassmannians whose trace is the binomial coefficient in
\eqref{eq:relative-trace}. This is the point at which monodromy
replaces geometric Euler characteristic in the ordinary argument.

Let $\mathcal M'$ be the mutation of $\mathcal M$ at $k$, followed
by reduction inducing the prescribed isomorphism of Jacobian algebras. Assume the source and target are
two-acyclic at the original vertices; in the global pattern this is
Corollary~\ref{thm:good-global-pattern}. Write
$h_k=h_k(M)$ and $h'_k=h_k(M')$. Thus $-h_k$ and $-h'_k$ are the
source and target cokernel dimensions; the latter includes the source
decoration. Trivial arrows act by zero, so these dimensions may be
computed before reduction. Set
\begin{equation}\label{eq:Yproj}
 Y_i=
 \begin{cases}
 (y_k')^{-1},& i=k,\\
 y_i'(y_k')^{[b_{ik}]_+}(1+y_k')^{-b_{ik}},& i\ne k.
 \end{cases}
\end{equation}

\begin{lemma}\label{lem:interval}
Put $\widehat e=1-e_k$ and fix quotient ranks away from $k$ for the
restriction $M\widehat e$ over the common algebra
$P_{\widehat k,\widehat k}=\widehat eP\widehat e$.
At a geometric point of its balanced representation Grassmannian,
write $N_i\subseteq M_i$ for the kernels after scalar extension.
Each sheet above $i$ has quotient dimension $e_i$. They induce
\[ N_{\mathrm{in}}=\bigoplus_{i\ne k}N_i\otimes_{K_i}A_{ik}\subseteq M_{\mathrm{in}},\qquad N_{\mathrm{out}}=\bigoplus_{j\ne k}\Hom_{K_j}(A_{kj},N_j)\subseteq M_{\mathrm{out}}. \]
Write
\[
\begin{gathered}
 r_{\mathrm{in}}=\dim_{K_k}M_{\mathrm{in}},\quad
 r_{\mathrm{out}}=\dim_{K_k}M_{\mathrm{out}},\quad
 r=\dim_{K_k}M_k,\quad r'=\dim_{K_k}M'_k,\\
 u=\sum_{i\ne k}[b_{ik}]_+e_i,\qquad
 w=\sum_{i\ne k}[-b_{ik}]_+e_i.
\end{gathered}
\]
At one geometric sheet above $k$, all dimensions below are ordinary
vector-space dimensions, and the possible kernels satisfy
\[
\begin{aligned}
 L=\alpha N_{\mathrm{in}}&\subseteq N_k\subseteq U=\beta^{-1}N_{\mathrm{out}},\\
 L'=\alpha'N_{\mathrm{out}}&\subseteq N'_k\subseteq U'=(\beta')^{-1}N_{\mathrm{in}}.
\end{aligned}
\]
If $q_-=r-\dim U$ and $q_+=r-\dim L$ are the least and greatest source
quotient ranks, then the corresponding target ranks satisfy
\begin{equation}\label{eq:field-interval-bounds}
 q'_-=u-h_k-q_+,\qquad q'_+=w-h'_k-q_-.
\end{equation}
\end{lemma}
\begin{proof}
The fixed quotient is a module over $P_{\widehat k,\widehat k}$.
Its stability gives
$\beta\alpha N_{\mathrm{in}}\subseteq N_{\mathrm{out}}$ and
$\gamma N_{\mathrm{out}}\subseteq N_{\mathrm{in}}$.
Together with $\beta'\alpha'=-\gamma$, these inclusions show that
both intervals are nonempty.
Use
$\ker\alpha'=\im\beta$, $\im\beta'=\ker\alpha$, and
$\beta'\alpha'=-\gamma$. At each sheet we have
$\dim N_{\mathrm{in}}=r_{\mathrm{in}}-u$ and
$\dim N_{\mathrm{out}}=r_{\mathrm{out}}-w$. It follows that
\[
\begin{aligned}
 \dim L'&=r_{\mathrm{out}}-w-\check h_k(M)-\dim U,\\
 \dim U'&=r_{\mathrm{in}}-u-\check h_k(M')-\dim L.
\end{aligned}
\]
The local dimension identities in the proof of
Theorem~\ref{thm:E-mutation} give
\[ h_k=r_{\mathrm{in}}-r-r'-\check h_k(M'),\qquad h'_k=r_{\mathrm{out}}-r-r'-\check h_k(M). \]
Substituting into $q'_-=r'-\dim U'$ and $q'_+=r'-\dim L'$ proves
\eqref{eq:field-interval-bounds}.
\end{proof}

\begin{theorem}\label{thm:mon-mutation}
The $F$-polynomial of the mutation satisfies
\begin{equation}\label{eq:mon-mut-p}
 F_{M'}(\mathbf y')
 =\frac{(1+(y_k')^{-1})^{h_k}}{(1+y_k')^{h'_k}}F_M(\mathbf Y).
\end{equation}

\end{theorem}
\begin{proof}
\emph{The common base and its rank strata.}
Let $\widehat e=1-e_k$. The isomorphism
$\widehat eP\widehat e\simeq\widehat eP'\widehat e$ identifies the
restricted modules and their balanced representation Grassmannians.
Denote the common
base at fixed outside ranks by $\mathcal Z$. Its stability conditions include all
outside paths, not only arrows not incident to $k$.

Stratify $\mathcal Z_{\overline K}$ by the tuple $(q_{-,a},q_{+,a})$ of extreme ranks in
Lemma~\ref{lem:interval}, for the sheets $a$ above $k$.
These are finitely many strata defined by matrix ranks, and $\sigma$
permutes them.
A nonconstant tuple is in a nontrivial cyclic orbit. Both its source
and target inverse-image strata therefore contribute zero alternating
trace by Lemma~\ref{lem:trace-strata}. Only strata on which the extreme ranks agree across all sheets
$q_{-,a}=q_-$, $q_{+,a}=q_+$ contribute. Their strata descend to $K$, and
\eqref{eq:field-interval-bounds} bijectively matches the constant source
and target tuples.

\emph{The relative Grassmannians.}
Write $\mathcal Z_{q_-,q_+}$ for the stratum with these extreme ranks.
On it, $L$ and $U$ are subbundles after passage to $K_k$, and $U/L$
is a bundle of rank $q_+-q_-$. The middle-kernel choices
form its relative Grassmannian, followed by Weil restriction to the
base. Lemma~\ref{lem:relative-grassmannian} gives the generating function
\[ y_k^{q_-}(1+y_k)^{q_+-q_-}\chis(\mathcal Z_{q_-,q_+}). \]
Finite invariant stratification and trace additivity supply the sum of
these terms.

\emph{The change of variables.}
The outside monomial under \eqref{eq:Yproj} contributes
$(y_k')^u(1+y_k')^{w-u}$. Including the prefactor in \eqref{eq:mon-mut-p}, the source fibre is
\[ (y_k')^{u-h_k-q_+}(1+y_k')^{w-u-h'_k+h_k+q_+-q_-} =(y_k')^{q'_-}(1+y_k')^{q'_+-q'_-}, \]
which is the target fibre by \eqref{eq:field-interval-bounds}.
Sum over the constant rank strata and outside quotient ranks.
The prescribed reduction identifies the representation Grassmannians
and their inertia actions. This is the quotient form of \cite[Lemma 5.2]{DWZ2}, with the relative
trace identity \eqref{eq:relative-trace} giving the coefficients.
\end{proof}

\begin{corollary}\label{cor:weight-character}
On a projective presentation with disjoint supports, the mutation identity is
\[ F_{M'}(\mathbf y') =(y_k')^{[\delta_k]_+}(1+y_k')^{-\delta_k}F_M(\mathbf Y). \]
\end{corollary}
\begin{proof}
Use $h_k=-[\delta_k]_+$ and $h'_k=-[-\delta_k]_+$ from
\eqref{eq:disjoint-support-betti}, and $1+(y_k')^{-1}=(1+y_k')/y_k'$.
\end{proof}

Apply the theorem to the opposite QP and $\mathbb DM$ for the injective
statement. This replaces $B$ by $-B$, $\delta$ by $\check\delta$, and
$\mathbf h(M)$ by $\check{\mathbf h}(M)$. Equation
\eqref{eq:F-opposite-duality} converts its coefficients to subrepresentation
counts. The substitution and prefactor are both taken in the injective
convention.

\begin{example}[A degree-$(1,4)$ family]
\label{run:ex:one-arrow}\label{run:lem:one-arrow-generality}
\label{run:prop:one-arrow-trace}\label{ex:monodromy-two-points}
We construct general presentations of one species whose integral
presentation matrices have the same reduction modulo $t$, but whose
monodromy $F$-polynomials have different coefficients.

Let $E=K(s)$, $s^4=t$, and take one arrow $1\to2$ with
\[
R=K\times E,\quad A_{12}={}_KE_E,\quad S=0,
 \qquad
 B=\begin{pmatrix}0&1\\-4&0\end{pmatrix},\quad d=(1,4).
\]
Write $P$ for its path algebra. Then $\dim_KP=9$, and the modulation is good.
Over $L=\Omega((t))$, $E_L=\Omega((s))$, where $\Omega\supseteq\C$ is an
algebraically closed residue field, choose four independent residue
parameters and put
\[
C=\begin{pmatrix}\alpha&\gamma\\\beta&\varepsilon\end{pmatrix},\quad
 J=\begin{pmatrix}0&0\\1&0\end{pmatrix},\quad
 T_m=J-t^mC,\quad Z_m=sI_2+s^2T_m\quad(m\ge1).
\]
The representation $M_m$ has ranks $(4,2)$, with arrow action
$[I_2\;Z_m]:E_L^4\longrightarrow E_L^2$.
For $P_i=e_iP$ it has the presentation
\begin{equation}\label{run:eq:one-arrow-presentation}
 P_2^{\oplus2}\xrightarrow{\binom{-Z_m}{I_2}}
 P_1^{\oplus4}\longrightarrow M_m,
 \qquad\delta=(4,-2).
\end{equation}
In \eqref{run:eq:one-arrow-presentation}, projection to the last two
coordinates is the identity on the source, so the integral cokernel
is free at both vertices. For every $m\ge1$, reduction modulo $t$ gives the arrow matrix
$[I_2\;sI_2+s^2J] \quad\text{over }(\Omega,\Omega[s]/(s^4))$.
For $m\ge N$, the presentations agree modulo $t^N$. Their weight is $2\eta$, where $\eta=(2,-1)$.

\emph{Generality.} We show that the union of the orbits of
$\binom{-sI_2-s^2T}{I_2}$, $T\in\op{Mat}_2(K)$, is dense
in $\PHom(4,-2)$. We then adjoin independent integral changes of
presentation bases to obtain algebraically generic coordinates.

After splitting $E/K$, a presentation of maximal rank is a quadruple
of two-dimensional subspaces, each with a chosen basis, in one
four-dimensional vector space. Generically the
first two planes are complementary, the third is normalized to the graph
of the identity, and the fourth is the graph of a two-by-two matrix up to
conjugacy. Its unordered pair of eigenvalues determines the orbit on a dense open.
For our family the four planes split according to the two eigenlines of
$T$. The four line coordinates for an eigenvalue $\lambda$ are
$z_a(\lambda)=i^as+(i^as)^2\lambda$, $a=0,1,2,3$. Their cross ratio is
$\frac{(z_0-z_2)(z_1-z_3)}{(z_0-z_3)(z_1-z_2)} =\frac{-2i}{2s^2\lambda^2-i}$.
This is nonconstant. The two eigenvalues of a generic $T$ give two
independent general cross ratios, so the union of their orbits is dense. The affine
change $C\mapsto J-t^mC$ is invertible over $K$. Extend the residue field and choose matrices in
$\GL_4(\Omega[[t]])$ and $\GL_2(\Omega[[s]])$ whose coefficients in
integral vertex-ring bases have independent residues. Their
determinants are units in the vertex rings, so they give integral
changes of presentation bases. The resulting map is dominant, and the same changes may be used for
every $m$. They give isomorphic integral presentations and hence
isomorphic special fibres.

\emph{The coefficient at quotient rank $(2,1)$.} Keep the valued
field and the representation $M_m$ just constructed.
The representation Grassmannian $\Gr^R_{(2,1)}(M_m)$ is finite
\'etale of degree two and
\[
[y_1^2y_2]F_{M_m}
   =\begin{cases}2,&m\text{ even},\\0,&m\text{ odd}.
     \end{cases}
\]

To compute this coefficient, let $\lambda_1,\lambda_2$ be the eigenvalues of $T_m$. Its discriminant is
\[ \Delta_m=-4t^m\gamma+ t^{2m}\bigl((\alpha-\varepsilon)^2+4\beta\gamma\bigr). \]
Both $\lambda_1-\lambda_2$ and $\lambda_1+\lambda_2$ are nonzero
generically. At the four sheets of $E$, given by $s_a=i^as$, the arrow matrices
are $[I_2\;s_aI_2+s_a^2T_m]$.

A quotient of ranks $(2,1)$ has a two-dimensional kernel $U$ at the
first vertex and a one-dimensional kernel at each sheet of the second
vertex. Diagonalize $T_m$ and use the corresponding eigenvector basis in each
of the two summands of $L^2\oplus L^2$. The four rank-at-most-one conditions
on $U$, in its Pl\"ucker coordinates, are
\[ p_{12}+z_a(\lambda_2)p_{14}-z_a(\lambda_1)p_{23} +z_a(\lambda_1)z_a(\lambda_2)p_{34}=0, \quad z_a(\lambda)=s_a+s_a^2\lambda. \]
Their four-by-four coefficient determinant is
$16i\,s^6(\lambda_1-\lambda_2)(\lambda_1+\lambda_2)\ne0$.
Thus $p_{12}=p_{14}=p_{23}=p_{34}=0$. The Pl\"ucker equation reduces
to $p_{13}p_{24}=0$ on the projective line $(p_{13}:p_{24})$,
giving exactly two reduced points in the Grassmannian of planes.
Their images at each sheet of the second vertex are one-dimensional,
so the corresponding kernels are unique. The identification
with the two eigenlines is equivariant: an eigenline $\ell$ gives
$U=\ell\oplus\ell$ and its extension to the sheets of the second vertex.

Every unit in $\Omega[[t]]$ has a square root. Since
$v_t(\Delta_m)=m$, the eigenlines are individually defined over $L$
when $m$ is even and are exchanged over $L(t^{1/2})$ when $m$ is odd.
Only $H^0$ contributes, giving traces two and zero, respectively.
Integral changes of presentation bases preserve this Grassmannian and
its Galois action.

Thus the specified integral representations have isomorphic reductions
modulo $t$, but their general valued families have different monodromy
$F$-polynomials, for the same QP over $K$ and the same $\delta$-vector. Section~\ref{app:subsec:rank-two-basis} uses
this same species with the Gauss valuation of its integral tensor
algebra, starting at $\eta$ and then treating its multiples.
\end{example}

\subsection{Newton polytopes and tropical \texorpdfstring{$F$}{F}-polynomials}
For a nonzero polynomial $F=\sum_{\mathbf e}a_{\mathbf e}\mathbf y^{\mathbf e}$,
its \emph{support} is $\{\mathbf e:a_{\mathbf e}\ne0\}$ and its \emph{Newton polytope} $\Newt(F)$ is the convex hull of that support. A polytope is
\emph{integral} if all its vertices belong to $\Z^n$.
Let $\overline M=\overline K\otimes_KM$; its vertex-field factors
are the geometric sheets. For a sheet dimension vector put
\[ \pi(\mathbf e)_i=\frac1{d_i}\sum_{a\mid i}e_a,\qquad \mathcal N_R(M)=\op{conv} \{\pi(\underline{\dim}(\overline M/L)):L\subseteq\overline M\}. \]
Thus $\mathcal N_R(M)$ is the projected \emph{quotient} polytope;
its defining quotients need not have balanced sheet dimensions.
Only finitely many dimension vectors occur.

\begin{theorem}\label{thm:Newton-monodromy}
For every finite ramified representation $M$,
\[ \Newt(F_M)=\mathcal N_R(M). \]
This polytope is integral and every vertex coefficient is one.
More precisely, every vertex is represented by a unique geometric
quotient, where quotients are parametrized by their kernels. That
kernel descends to $K$, and the quotient has balanced sheet dimensions.

\end{theorem}
\begin{proof}
The function $d_u(L)=u\cdot\pi(\underline{\dim}L)$ is \emph{modular} on
the submodule lattice:
\[ d_u(L_1)+d_u(L_2)=d_u(L_1\cap L_2)+d_u(L_1+L_2). \]
Choose $u$ separating the finitely many projected quotient dimensions.
Its maximizing quotients correspond to minimizers of $d_u$.
Modularity shows that their intersections and sums also minimize.
All minimizers have the same projected dimension; an inclusion between
two of them is therefore equality, since sheet dimensions can only
increase. Thus the minimizing kernel is unique.

The functional is Galois-invariant, so the kernel descends to $K$.
Transitivity on each set of sheets makes the quotient ranks balanced.
The corresponding representation Grassmannian has one geometric point defined over
$K$; its compact-support cohomology has trace one, irrespective of
nilpotent structure. Every vertex can be exposed by such a $u$, while
every exponent of $F_M$ already lies in $\mathcal N_R(M)$. This proves
both equality of the polytopes and the coefficient assertion. Compare
the subrepresentation argument in \cite[Theorem 4.2]{FeiCombinatorics}.
\end{proof}

The \emph{support function} of a polytope $N$ is $u\mapsto\max_{q\in N}u\cdot q$.
Define the \emph{tropical $F$-polynomial} by
\[ f_M(u)=\max_{a_{\mathbf e}\ne0}u\cdot\mathbf e \quad\text{when }F_M=\sum_{\mathbf e}a_{\mathbf e}\mathbf y^{\mathbf e}. \]
The theorem gives
\[ f_M(u)=\max_{L\subseteq\overline M} u\cdot\pi(\underline{\dim}(\overline M/L)). \]
It is enough to maximize over quotients defined over $K$. Indeed,
the intersection of minimizing kernels remains minimizing by modularity.
In a finite-dimensional module, the intersection of the whole family
already equals a finite intersection. It is therefore minimizing and
Galois-stable. Thus $f_M$ is an
integral convex piecewise-linear function. The submodule support
function is recovered, without a new definition, from
\begin{equation}\label{eq:tropical-submodule}
 \max_{L\subseteq\overline M}u\cdot\pi(\underline{\dim}L)
     =u\cdot\mathbf r_M+f_M(-u)
     =f^{\mathrm{op}}_{\mathbb DM}(u).
\end{equation}
Here and below the superscript $\mathrm{op}$ only indicates that the
same construction is taken over the opposite algebra.

For $u\in\mathbb R^n$ put
$(T_{k,B}u)_i=\begin{cases} -u_k & (i=k)\\ u_i+[b_{ik}]_+u_k-b_{ik}[u_k]_+ & (i\ne k) \end{cases}$.
This is a transformation of variable weights, not of presentation
indices.

\begin{corollary}\label{cor:tropical-mutation}
For the mutation in Theorem~\ref{thm:mon-mutation}, we have
\[ f_{M'}(u)=f_M(T_{k,B}u)-h'_k[u_k]_++h_k[-u_k]_+. \]
On a projective presentation with disjoint supports the correction is
$[\delta_k]_+u_k-\delta_k[u_k]_+$. Opposite duality gives the
subrepresentation formula, with $B$ replaced by $-B$.
\end{corollary}
\begin{proof}
On either half-space $u_k>0$ or $u_k<0$, $T_{k,B}$ is an
invertible linear map. For a general $u$ in either half-space, the
distinct exponents of the old polynomial have distinct weights after
substitution. In particular, there is a unique term of greatest weight. Its coefficient is one by
Theorem~\ref{thm:Newton-monodromy}. The factors $1+y_k'$ and
$1+(y_k')^{-1}$ likewise have a unique term of greatest weight on
each half-space. These extremal terms cannot cancel. Taking greatest
weights in \eqref{eq:mon-mut-p} proves the formula on a dense set;
continuity proves it everywhere. Equation~\eqref{eq:disjoint-support-betti}
gives the simplified correction.
\end{proof}

\subsection{Generic Newton polytopes and fixed-field transport}
For the valued family of Section~\ref{subsec:valued-families}, with
$\kk=\C$, define
\begin{equation}\label{eq:universal-F}
 F_{\delta;\PP_0}=F_{\mathcal M_{\delta,\mathrm{univ}}}.
\end{equation}
All trace arguments apply over any algebraically closed
characteristic-zero residue field, including
$\overline{\C(z_1,\ldots,z_q)}$. The initial order selects the
valuation.

Fix a reduced Jacobi-finite QP and put $P=\mathcal P(A,S)$.
Use the presentation space $X_\delta=\PHom(\delta)$ and its function
field $K_\delta^{\mathrm{pres}}$ from Section~\ref{subsec:presentation-spaces}.
Let $M_\delta^{\rm gen}$ be the cokernel of the universal presentation
over $K_\delta^{\mathrm{pres}}$, and let $\overline M_\delta^{\rm gen}$
be its extension to an algebraic closure. Decoration is omitted here
because it does not affect the representation Grassmannians. Define
\begin{equation}\label{newton:eq:generic-polytope}
 \mathcal N_\delta(P)=\op{conv}
 \left\{\pi\bigl(\underline{\dim}(\overline M_\delta^{\rm gen}/L)\bigr):
      L\subseteq\overline M_\delta^{\rm gen}\right\}.
\end{equation}
The average $\pi$ uses the same $d_i$ sheets as above. We call this the \emph{generic Newton polytope} at presentation weight
$\delta$. It uses only $P$, its distinguished vertex fields, and $\delta$; it uses
neither an integral order nor a valuation on $K_\delta^{\mathrm{pres}}$.

\begin{theorem}\label{newton:thm:order-independence}
For any two full initial orders $\PP_0,\mathscr Q_0\subseteq P$
containing $\RR$, and every weight $\delta\in\Z^n$,
\begin{equation}\label{newton:eq:order-independent}
 \Newt(F_{\delta;\PP_0})
   =\mathcal N_\delta(P)
   =\Newt(F_{\delta;\mathscr Q_0}).
\end{equation}
The polytope is integral and every vertex has coefficient one in
both polynomials. More generally, the same conclusion holds for
any algebraically generic valued presentation of weight $\delta$
over $L=\Omega((t))$, with $\Omega\supseteq\C$ algebraically closed
and vertex fields $L(t^{1/d_i})$. Consequently, its tropical
$F$-polynomial is the order-independent function
$f_\delta^P(u)=\max_{q\in\mathcal N_\delta(P)}u\cdot q$.
This assertion needs neither goodness, nondegeneracy, full matrix rank,
nor the seed reachability condition.
\end{theorem}
\begin{proof}
The full Hom lattice of any order spans the same affine space
$X_\delta$, by \eqref{eq:full-presentation-lattice}. Its universal map
therefore induces an embedding $K_\delta^{\mathrm{pres}}\hookrightarrow L_\delta$
(Proposition~\ref{prop:universal-invariance}), and its positive module
is the scalar extension of $M_\delta^{\rm gen}$. Different orders
change that embedding as valued data, not algebraic genericity.
The same observation holds for any algebraically generic presentation
by its definition.

Extend each embedding to algebraic closures. For every possible sheet
dimension vector, the representation Grassmannian of the extended module is the
corresponding representation Grassmannian of $\overline M_\delta^{\rm gen}$ after
field extension. A scheme is empty if and only if its extension to a
larger field is empty: field extension is faithfully flat. Hence precisely
the same geometric quotient dimension vectors occur. Their projected
convex hull is \eqref{newton:eq:generic-polytope}. Permuting the sheets
permutes the coordinates within each average and has no effect on it.
Theorem~\ref{thm:Newton-monodromy}, applied over the valued coefficient
field, identifies this hull with the Newton polytope of the trace
polynomial and gives the vertex coefficients. Equality of support
functions proves the final assertion.
\end{proof}

\begin{corollary}\label{newton:cor:transport}
Under Hypothesis~\ref{hyp:initial-realization}, fix a QP
mutation sequence $w$. For every initial full order, its transported
$F$-polynomial at that sequence has Newton polytope
$\mathcal N_{\mu_w\delta}(P_w)$, where $P_w$ is the target
Jacobian algebra. The tropical $F$-polynomials of these polytopes obey
Corollary~\ref{cor:tropical-mutation} with the positive and negative parts of the weight.
\end{corollary}
\begin{proof}
There is a nonempty open in $X_\eta$, invariant under changes of
presentation bases, on which all geometric quotient dimension vectors
equal their generic values. Indeed, first restrict so the universal cokernel is a vector
bundle of fixed ranks. After finite vertex-field splitting, each of
the finitely many possible sheet-representation Grassmannians is projective
over this open. Its image is closed. If it meets the generic fibre,
its image is the whole base; otherwise remove that proper closed image.
The intersection for all ranks descends to $K$ and is invariant under
changes of presentation bases. Corollary~\ref{thm:fixed-field-generality}
and Lemma~\ref{lem:orbit-open} put the transported object
in this open at every fixed sequence. Apply Theorem~\ref{thm:Newton-monodromy} and
Corollary~\ref{cor:tropical-mutation}. 
\end{proof}

Coefficients at nonvertices may differ, even on the boundary. Thus
\eqref{newton:eq:order-independent} is weaker than equality of supports
or of polynomials, but stronger than mere agreement of their leading
terms. The Jacobian algebra over $K$ and vertex data are fixed throughout.

\begin{corollary}\label{thm:fixed-field-transport}
Under Hypothesis~\ref{hyp:initial-realization}, keep the valued field
and tame generator of $\mathcal M_{\delta,\mathrm{univ}}$ fixed.
For every finite mutation sequence $w$, set
$\mathcal M_{\delta,w}=\mu_w\mathcal M_{\delta,\mathrm{univ}}$
and $\delta_w=\mu_w\delta$. Write $wk$ for the sequence obtained by
appending $k$ to $w$. Then
\begin{equation}\label{eq:transported-universal-mutation}
 F_{\mathcal M_{\delta,wk}}(\mathbf y')
   =(y_k')^{[\delta_{w,k}]_+}(1+y_k')^{-\delta_{w,k}}
             F_{\mathcal M_{\delta,w}}(\mathbf Y).
\end{equation}
The injective general-weight assertion is this corollary for the
opposite initial order and the dual copresentation.
\end{corollary}
\begin{proof}
Corollary~\ref{thm:fixed-field-generality} preserves generality.
Proposition~\ref{prop:disjoint-support-indices} supplies the local cokernel
ranks, and Corollary~\ref{cor:weight-character} gives the identity.
\end{proof}

The field comparisons in Proposition~\ref{prop:coefficient-extension}
preserve the entire polynomial. The weaker Newton-polytope comparison
of Theorem~\ref{newton:thm:order-independence} also allows different
valuations defined by orders.

\section{From monodromy characters to cluster bases}
\label{sec:cluster-application}
We now apply the mutation formula to the valued universal presentations.
Their fixed-field transport gives Laurentness and compatible
pointedness. Under the seed reachability condition, Qin's criterion
then gives a basis with full-rank coefficients. The passage to
arbitrary geometric coefficients uses linear independence of specialized
theta functions and yields a basis of the middle algebra.
We keep this step separate from the conditions identifying the middle
and upper algebras.

\subsection{An initial realization and transport of characters}

Fix $(A,S,\PP_0)$ as in Hypothesis~\ref{hyp:initial-realization}.
A seed is a pair $(\mathbf x,\widetilde B)$, where
$\mathbf x=(x_1,\ldots,x_m)$ is a transcendence basis generating the
ambient rational function field over $\Q$ and
$\widetilde B=(B\mid C)$ is an $n\times m$ integer matrix.
The variables $x_1,\ldots,x_n$ are mutable, and
$x_{n+1},\ldots,x_m$ are frozen. The entries $b_{ij}$ are indexed by
$1\leq i\leq n$ and $1\leq j\leq m$. Mutation at $1\leq k\leq n$
is defined by the exchange relation
\begin{equation}\label{eq:cluster-exchange-row}
 x_kx'_k=\prod_{j=1}^m x_j^{[b_{kj}]_+}
              +\prod_{j=1}^m x_j^{[-b_{kj}]_+},
 \qquad x'_j=x_j\quad(j\ne k),
\end{equation}
and the matrix rule \eqref{alg:eq:FZmutation}, with $1\leq j\leq m$. A seed pattern assigns these data to the $n$-regular tree with mutation directions assigned to its edges:
every vertex has one incident edge for each mutable index,
and crossing that edge applies its mutation. The coefficient ring
is $\Z[x_{n+1}^{\pm1},\ldots,x_m^{\pm1}]$. The cluster algebra is
the algebra over this ring generated by all mutable cluster variables;
the upper cluster algebra $\mathcal U$ is the intersection of the
Laurent rings of all seeds in the ambient field. A cluster monomial is
a product of nonnegative powers of mutable variables from one seed
(with a frozen Laurent monomial allowed). \emph{Principal coefficients}
mean $m=2n$ and $\widetilde B_0=(B_0\mid-I_n)$ in our convention;
\emph{coefficient-free} means $m=n$.

For each initial mutable weight $\delta$, choose the universal
representation over its single valued field $L_\delta$, as in
\eqref{eq:universal-F}. Let $\mathcal M_{\delta,w}$ be its transport along a sequence $w$, and let $\delta_w=\mu_w\delta$.
Corollary~\ref{thm:fixed-field-transport} applies to every such sequence.

An extended weight $\widetilde\delta\in\mathbb Z^m$ has mutable
part $\delta$ and arbitrary frozen coordinates, which record a
Laurent coefficient monomial. All its coordinates transform by
\eqref{eq:weight-p}.
Put
\[ \widehat y_i=x^{-\widetilde b_i}\qquad(1\leq i\leq n), \]
and define the initial character by
\begin{equation}\label{eq:generic-character}
 \mathsf G^{\PP_0}_{\widetilde\delta}
  =x^{-\widetilde\delta}
        F_{\delta;\PP_0}(\widehat{\mathbf y}).
\end{equation}
At the seed reached by $w$, its expression uses
$F_{\mathcal M_{\delta,w}}$. The injective character is the same
construction on the opposite QP and the dual representation. Indeed,
\eqref{eq:mon-reciprocity} gives the alternative expression
\[ x^{-\widetilde{\check\delta}} F^{\mathrm{op}}_{\mathbb DM}(x^{\widetilde b_i})_{i=1}^n, \qquad \widetilde{\check\delta}=\widetilde\delta+\mathbf r_M\widetilde B. \]
This expresses the same character in the dual convention.

For the full-rank argument below, put
$p^*(\mathbf e)=-\mathbf e\widetilde B$ and assume
$\rank\widetilde B=n$. The \emph{dominance order} is defined by
$g'\preceq_{\mathbf s}g$ when $g'=g+p^*(\mathbf e)$ for some
$\mathbf e\in\Z_{\ge0}^n$. Injectivity of $p^*$ makes this relation
antisymmetric. A Laurent polynomial is \emph{pointed} of degree $g$
if it has the form
\[ x^g\left(1+\sum_{0\ne\mathbf e\in\Z_{\ge0}^n}a_{\mathbf e}x^{-\mathbf e\widetilde B}\right),\qquad a_{\mathbf e}\in\Z, \]
with a finite sum. Its leading monomial is $x^g$, with coefficient one.
A collection is \emph{compatibly pointed} along a sequence
\cite[Definition 3.4.2]{Qin} if its degrees
at adjacent seeds transform as $g'=-\mu_k(-g)$, using
\eqref{eq:weight-p} on all coordinates.

\begin{proposition}\label{prop:generic-character-mutation}
The initial character \eqref{eq:generic-character} is a Laurent
polynomial in every seed. More precisely, at the seed $\mathbf s_w$
it is
\[ x(\mathbf s_w)^{-\widetilde\delta_w} F_{\mathcal M_{\delta,w}} \bigl(x(\mathbf s_w)^{-\widetilde b_i(\mathbf s_w)}\bigr)_{i=1}^n. \]
It therefore belongs to the upper cluster algebra. If
$\rank\widetilde B=n$, this expression is pointed with degree
$-\widetilde\delta_w$.
\end{proposition}
\begin{proof}
Write $\widehat y_i'=\widehat y_i(\mathbf s')$ and
$M_{\pm}=\prod_j(x'_j)^{[\pm b_{kj}]_+}$. Then $\widehat y_k'=M_+/M_-$. Matrix mutation and
\eqref{eq:cluster-exchange-row} give
\[ \widehat y_k(\mathbf s)=(\widehat y_k')^{-1},\qquad \widehat y_i(\mathbf s) =\widehat y_i(\mathbf s')(\widehat y_k')^{[b_{ik}]_+}(1+\widehat y_k')^{-b_{ik}}. \]
The extended weight satisfies
\begin{equation}\label{eq:character-monomial-identity}
 x^{-\widetilde\delta}
 =(x')^{-\mu_k\widetilde\delta}
           (\widehat y_k')^{[\delta_k]_+}(1+\widehat y_k')^{-\delta_k}.
\end{equation}
For $\delta_k\geq0$, the exchange relation contributes
$M_-^{-\delta_k}(1+\widehat y_k')^{-\delta_k}$, while the weight change
contributes $M_+^{-\delta_k}(\widehat y_k')^{\delta_k}$, which is the same
factor. For $\delta_k\leq0$ use $M_-$ instead of $M_+$.
Now apply \eqref{eq:transported-universal-mutation}. This proves
the one-step identity, including every frozen coordinate, and
iteration proves the assertion in every seed.

The polynomials have integer coefficients and constant term one.
Every exponent besides the leading exponent is
$-\widetilde\delta_w-\mathbf e\widetilde B$, with
$\mathbf e\geq0$ and $\mathbf e\ne0$. Full rank excludes zero from
these displacement vectors and makes their cone pointed. Thus the leading monomial has coefficient one and is the unique
maximal monomial. Every edge is evaluated in the initially fixed valued field.
\end{proof}

When $\rank\widetilde B=n$, the monomial substitution is injective. Equal
Laurent characters in the same seed and extended weight therefore have
equal $F$-polynomials, even without identifying their target QPs or orders.

\begin{corollary}\label{cor:tropical-denominator}
Suppose $\rank\widetilde B=n$. For the character of a
specified representation in this section, let
$\op{den}_j(\mathsf G)$ be minus the smallest $x_j$-exponent
in its Laurent expansion. This is the signed denominator convention;
in particular $\op{den}_j(x_j)=-1$. Then
\[ \op{den}_j(\mathsf G) =\widetilde\delta_j+f_M(\widetilde B_{\bullet j}). \]
Here $\widetilde B_{\bullet j}=(b_{1j},\ldots,b_{nj})^{\mathsf T}$
acts on a quotient-rank vector $\mathbf e$ by
$\sum_i e_i b_{ij}$. If $\rank\widetilde B<n$ and $\mathsf G\ne0$, the right-hand side
is an upper bound for this signed denominator;
equality holds whenever the minimizing Laurent terms do not cancel.
\end{corollary}
\begin{proof}
A quotient exponent $\mathbf e$ contributes the Laurent exponent
$-\widetilde\delta-\mathbf e\widetilde B$. If $\widetilde B$ has
full row rank, this affine map is injective, and the smallest $j$th
coordinate in its image is
\[ -\widetilde\delta_j-\max_{\mathbf e\in\Newt(F_M)}\mathbf e\cdot\widetilde B_{\bullet j}=-\widetilde\delta_j-f_M(\widetilde B_{\bullet j}). \]
Negating this minimum gives the asserted signed denominator. Without
injectivity, cancellation can only raise the minimum, giving the upper
bound and the stated equality condition.
\end{proof}

In particular, for fixed $(P,\widetilde B,\widetilde\delta)$ with
$\rank\widetilde B=n$, Theorem~\ref{newton:thm:order-independence} gives
\begin{equation}\label{newton:eq:character-polytope}
 \Newt(\mathsf G^{\PP_0}_{\widetilde\delta})
   =-\widetilde\delta-\mathcal N_\delta(P)\widetilde B.
\end{equation}
Here $N\widetilde B=\{q\widetilde B:q\in N\}$.
There is no cancellation under the injective exponent map. The signed
denominator vector is consequently independent of $\PP_0$ at the
initial seed and, by Corollary~\ref{newton:cor:transport}, after every
fixed mutation sequence.

\subsection{Cluster variables and monomials from reachable objects}
\label{subsec:cluster-model}
For $a\in\Z_{\ge0}^n$, let
$\mathcal T_{a;\mathbf s}=\bigoplus_i\mathcal T_{i;\mathbf s}^{a_i}$.
Its extended weight at the initial seed is obtained by transporting $(-a,0)$ back
from $\mathbf s$; the frozen coordinates record only Laurent
coefficient factors, not frozen modules.

\begin{theorem}\label{thm:cluster-model}
For every node $\mathbf s$, the extended character of
$\mathcal T_{i;\mathbf s}$ is the mutable cluster variable $x_i(\mathbf s)$.
The extended character of $\mathcal T_{a;\mathbf s}$ is
$\prod_i x_i(\mathbf s)^{a_i}$. Allowing frozen Laurent monomial factors gives every cluster monomial.
Multiplication within a common reachable collection agrees with direct sum.

If $\rank\widetilde B=n$, equality of two such extended
characters forces equality of their extended weights at the initial seed and
isomorphism of their decorated representations at the initial seed. Modulo Laurent
monomials in the frozen variables, this gives a bijection between
reachable decorated isomorphism classes and cluster monomials modulo
such coefficient factors. The seed reachability condition is not needed for these assertions.
\end{theorem}
\begin{proof}
At $\mathbf s$, the object $\mathcal T_{a;\mathbf s}$ is the negative
decoration $(0,\bigoplus_iK_i^{a_i})$, with $F=1$ and extended weight
$(-a,0)$, giving the stated monomial. Along the reverse sequence it stays
$E$-rigid; thus \eqref{eq:E0-betti} supplies the positive and negative parts of the weight.
Theorem~\ref{thm:mon-mutation} and
\eqref{eq:character-monomial-identity} preserve the Laurent function
on every edge. This uses only the rigid statements, independently of
general-presentation transport and the basis theorem.

Every cluster variable and monomial occurs at a node. Multiplicativity for direct sums proves the multiplication assertion. When $\rank\widetilde B=n$, the initial character is pointed with leading exponent minus the
extended weight, so equal functions have the same extended weight.
Proposition~\ref{prop:rigid-weight} identifies the corresponding
modules at the initial seed. Allowing a frozen Laurent factor changes only the frozen
part of the leading exponent, which proves the final formulation.
\end{proof}

At a reachable weight every universal presentation associated with the order lies in the
rigid open orbit (Proposition~\ref{prop:rigid-weight}), so its
character is order-independent. The basis theorem extends this
construction to all extended weights.

\subsection{The middle cluster algebra and the basis theorem}

The notation in Qin's criterion is related to ours by
\[ \widehat B=-\widetilde B^{\mathsf T},\qquad g=-\widetilde\delta. \]
We define $g$-vectors using principal coefficients. Give their initial
variables the degrees
\[ \deg x_i=\mathbf e_i,\qquad \deg u_i=(b_{i1},\ldots,b_{in}). \]
The $g$-vector of a cluster variable is its degree for this grading.
When the extended matrix has full rank, it is also the mutable part of
the pointed degree above. A seed satisfies the \emph{seed reachability condition}
if a mutation sequence gives terminal cluster variables whose principal
$g$-vectors are $-\mathbf e_{\pi(i)}$ for a permutation $\pi$.
We impose this condition on the principal-coefficient pattern; it is
independent of the coefficient extension subsequently chosen.
Equivalently, there is a predecessor seed $\mathbf s[-1]$ reaching
$\mathbf s$ in this way \cite[Definition 2.3.1]{Qin}.

We first recall the coefficient construction of the middle cluster
algebra. Write $\widetilde B=(B\mid C)$, with $C=(c_{ij})$ of size
$n\times f$, and put
\[ z_j=x_{n+j},\qquad R_{\mathrm f}=\Z[z_1^{\pm1},\ldots,z_f^{\pm1}],\qquad R_{\mathrm{prin}}=\Z[u_1^{\pm1},\ldots,u_n^{\pm1}]. \]
When $f=0$, take $R_{\mathrm f}=\Z$. Principal coefficients have
matrix $(B\mid-I_n)$. The coefficient homomorphism is
\begin{equation}\label{eq:coefficient-substitution}
 \rho:R_{\mathrm{prin}}\longrightarrow R_{\mathrm f},
 \qquad u_i\longmapsto\prod_{j=1}^f z_j^{-c_{ij}}.
\end{equation}
We also write $\rho$ for the induced homomorphism of initial Laurent
rings, which fixes $x_1,\ldots,x_n$.

The \emph{theta functions} of Gross--Hacking--Keel--Kontsevich are
canonical Laurent series attached to the cluster pattern
\cite{GHKK}. For principal coefficients, let
$\vartheta_g^{\mathrm{prin}}$ be the function with initial degree
$(g,0)$, where $g\in\Z^n$; arbitrary frozen degrees are obtained by
multiplying by Laurent monomials in the $u_i$. Under the seed reachability condition, the initial and reached clusters have mutable $g$-vectors
$\mathbf e_i$ and $-\mathbf e_i$, up to permutation. Their cluster
complex is therefore not contained in a half-space. By
\cite[Proposition 8.28]{GHKK}, the theta series are Laurent polynomials
and form a basis of the upper cluster algebra with principal
coefficients. The integer broken-line coefficients and the
leading-coefficient-one expansions give its integral form. In particular,
\begin{equation}\label{eq:principal-theta-basis}
 \Mid^{\mathrm{prin}}=\mathcal U^{\mathrm{prin}}
   =\bigoplus_{g\in\Z^n}R_{\mathrm{prin}}\,
       \vartheta_g^{\mathrm{prin}}.
\end{equation}
Here $\Mid^{\mathrm{prin}}$ is the algebra spanned by the polynomial
theta functions. These functions and their multiplication depend only
on the cluster pattern, not on the ramified realization.

The \emph{middle cluster algebra} for the chosen coefficients is the
coefficient extension of this theta algebra
\cite[Definition--Lemma 7.14]{GHKK}. Theorem 3.10 and Corollary 3.11
of \cite{CMMM} show that the specialized polynomial theta functions
are linearly independent over $R_{\mathrm f}$. Their coefficient
hypothesis holds for \eqref{eq:coefficient-substitution}: every nonzero
integer multiple of a coefficient monomial maps to a nonzero
nonzerodivisor in $R_{\mathrm f}$. We therefore identify the middle
algebra with its image in the initial Laurent ring:
\begin{equation}\label{eq:middle-algebra}
 \Mid(\widetilde B)
 =\bigoplus_{g\in\Z^n}R_{\mathrm f}\,
    \rho(\vartheta_g^{\mathrm{prin}})
 \subseteq R_{\mathrm f}[x_1^{\pm1},\ldots,x_n^{\pm1}].
\end{equation}
The multiplication is the specialization of theta multiplication.
Choosing different frozen degrees for the representatives multiplies
the functions by invertible coefficient monomials, without changing
the algebra.

We first prove the basis theorem with full-rank coefficients, then use
principal coefficients to pass to arbitrary geometric coefficients.
For the first step we use \cite[Theorem 4.3.1]{Qin}: with full rank and the seed reachability condition,
a family of elements of the upper cluster algebra, with one element
in every degree, is a basis if it is compatibly pointed along the reaching
sequence. The criterion is skew-symmetrizable and does not require
nonnegative Laurent coefficients. Its finiteness argument compares
leading-term expansions at the two end seeds. Compatibility identifies
their coefficients; one seed bounds the support above and the other
bounds it below. The support lies in finitely many dominance intervals.
Each interval is finite because the two nonnegative vectors expressing
its inequalities have a fixed sum under the injective map $p^*$.
Thus the expansions are finite. We use this argument before specializing
coefficients.

\begin{theorem}\label{thm:generic-basis}
Fix a good ramified modulation, a Jacobi-finite nondegenerate potential,
and a full initial order $\PP_0\supseteq\RR$ as in
Hypothesis~\ref{hyp:initial-realization}. Assume the seed reachability condition for the principal-coefficient pattern. For every geometric
coefficient extension $\widetilde B=(B\mid C)$, the characters satisfy
\[ \Mid(\widetilde B)=\bigoplus_{\delta\in\Z^n}R_{\mathrm f}\,\mathsf G^{\PP_0}_{(\delta,0)}. \]
The characters with all extended weights form a $\Z$-basis of this
algebra and include all cluster monomials. Moreover,
\[ \mathcal A(\widetilde B)\subseteq\Mid(\widetilde B)\subseteq\mathcal U(\widetilde B). \]
\end{theorem}
\begin{proof}
\emph{Full-rank coefficients and the principal case.}
Suppose first that $\rank\widetilde B=n$. For every mutable weight $\delta$,
take the universal presentation of
\eqref{eq:full-presentation-lattice}--\eqref{eq:Gauss} over its fixed
valued field $L_\delta$. Corollary~\ref{thm:fixed-field-generality} keeps
its transport general, so the projective cokernel identities for
disjoint supports hold along the reaching sequence.

By Proposition~\ref{prop:generic-character-mutation}, each character
belongs to the upper cluster algebra. At every seed, the constant
$F$-coefficient is one and the other Laurent exponents differ from
$-\widetilde\delta$ by elements of
$-\Z_{\ge0}^n\widetilde B\setminus\{0\}$. Full rank makes this cone
pointed and the exponent map injective. The characters are therefore
compatibly pointed, and the bijective tropical weight transformations
give one character in every degree. Qin's criterion proves that they
form a $\Z$-basis of the upper cluster algebra. Changing only the
initial frozen weight multiplies a character by a frozen Laurent
monomial. Thus the members with zero initial frozen weight form a basis over the
Laurent coefficient ring.

Apply this result to principal coefficients and write
\[ b_\delta^{\mathrm{prin}}=x^{-\delta}F_{\delta;\PP_0}(x^{-\mathbf b_i}u_i)_{i=1}^n,\qquad \mathcal U^{\mathrm{prin}}=\bigoplus_{\delta\in\Z^n}R_{\mathrm{prin}}\,b_\delta^{\mathrm{prin}}, \]
where $x^{-\mathbf b_i}=\prod_{j=1}^n x_j^{-b_{ij}}$.
Together with \eqref{eq:principal-theta-basis}, this gives two bases of
the same free $R_{\mathrm{prin}}$-module.

\emph{Specialization to arbitrary geometric coefficients.}
For arbitrary $C$, coefficient substitution induces the homomorphism
\[ \psi:R_{\mathrm f}\otimes_{R_{\mathrm{prin}}}\mathcal U^{\mathrm{prin}}\longrightarrow R_{\mathrm f}[x_1^{\pm1},\ldots,x_n^{\pm1}],\qquad r\otimes f\longmapsto r\rho(f). \]
The tensors $1\otimes\vartheta_g^{\mathrm{prin}}$ form an
$R_{\mathrm f}$-basis of its domain. Their images are linearly
independent by \cite[Theorem 3.10]{CMMM}, so $\psi$ is injective.
By \eqref{eq:middle-algebra}, its image is $\Mid(\widetilde B)$.
The tensors $1\otimes b_\delta^{\mathrm{prin}}$ form another basis of
the domain: tensoring a free-module decomposition gives a direct sum of
copies of $R_{\mathrm f}$, without a flatness assumption. Their images
are exactly $\mathsf G^{\PP_0}_{(\delta,0)}$, since the universal
$F$-polynomial depends on the mutable weight and the order, not on the
coefficient extension. They give the asserted basis of the middle cluster algebra.

\emph{Cluster monomials and extended weights.}
The transported-character formula places every member in
$\mathcal U(\widetilde B)$. For a cluster monomial, use
Theorem~\ref{thm:cluster-model} to obtain its reachable decorated
representation and extended weight. This object is $E$-rigid, so
Proposition~\ref{prop:rigid-weight} identifies it with the universal
decorated cokernel of its mutable weight. Its character is the cluster
monomial. This proves $\mathcal A(\widetilde B)\subseteq\Mid(\widetilde B)$.
Finally,
\[ \mathsf G^{\PP_0}_{(\delta,\gamma)}=z^{-\gamma}\mathsf G^{\PP_0}_{(\delta,0)} \]
for $\gamma\in\Z^f$. The Laurent monomials in the $z_j$ form a
$\Z$-basis of $R_{\mathrm f}$, proving the statement for all extended
weights.
\end{proof}

\begin{corollary}\label{cor:upper-character-basis}
Under the hypotheses of Theorem~\ref{thm:generic-basis}, the character
family is a basis of $\mathcal U(\widetilde B)$ whenever
$\Mid(\widetilde B)=\mathcal U(\widetilde B)$. In particular, this
holds if $\rank\widetilde B=n$, or if
$\mathcal A(\widetilde B)=\mathcal U(\widetilde B)$.
\end{corollary}
\begin{proof}
The full-rank case was proved in the first part of the basis argument. If $\mathcal A(\widetilde B)=\mathcal U(\widetilde B)$,
the two inclusions in Theorem~\ref{thm:generic-basis} force equality
with the middle algebra.
\end{proof}

\begin{corollary}\label{cor:character-span}
For a fixed exchange pattern satisfying the seed reachability condition,
the algebra spanned by these characters is independent of the full
initial order and of the chosen good Jacobi-finite nondegenerate
ramified realization. It is $\Mid(\widetilde B)$.
\end{corollary}
\begin{proof}
Apply Theorem~\ref{thm:generic-basis} to each realization and order.
\end{proof}

The comparison uses independence of specialized theta functions to
prove injectivity of $\psi$. It identifies the two spanned algebras,
not their individual basis elements. In particular, the order may
change the monodromy coefficients while the middle algebra stays fixed.
For full-rank coefficients, the Newton polytope and signed denominator
of each character also depend only on the Jacobian algebra over $K$ and
its weight, by \eqref{newton:eq:character-polytope}.

Suppose the hypotheses of Theorem~\ref{thm:generic-basis} hold and
$P_0$ is a Jacobian algebra over $\C$. For the constant order
$\PP_0=\C[[t]]\otimes_\C P_0$ with $d_i=1$, the universal presentation is defined over the residue
function field. Inertia is trivial, so the initial polynomials are
the ordinary generic quotient $F$-polynomials. Theorem~\ref{thm:generic-basis}
identifies their span with the middle algebra. Under the hypotheses of
Corollary~\ref{cor:upper-character-basis}, they form the usual generic
basis of the upper cluster algebra.

\subsection{Bases from extensions and a four-vertex cycle}
The realization theorem for $\mathfrak P$ gives a corresponding family
of bases without a restriction on the geometric coefficients.

\begin{corollary}\label{cor:class-P-basis}
Let $B\in\mathfrak P$, and fix positive integral degrees $d$ such that
$BD$ is skew-symmetric. There is a good Jacobi-finite nondegenerate
ramified realization. For every such realization, every full initial
order containing $\RR$, and every geometric coefficient extension,
the monodromy characters form a basis of the middle cluster algebra
containing all cluster monomials.
\end{corollary}
\begin{proof}
Corollary~\ref{gm:cor:class-P} gives a good modulation and a reddening
sequence, and Corollary~\ref{gm:cor:reddening-finite} gives the potential.
A reddening sequence implies the seed reachability condition
\cite[Corollary 2.3.4]{Qin}. Apply
Theorem~\ref{thm:generic-basis}.
\end{proof}

In particular, for an acyclic matrix every two-acyclic modulation is
good, and its zero potential is rigid, nondegenerate, and Jacobi-finite
after splitting. Thus its integral tensor algebra, or any other full initial
order containing $\RR$, gives this basis. Principal coefficients also
give a basis of the upper cluster algebra by
Corollary~\ref{cor:upper-character-basis}.

\paragraph{Cartan string diagrams.}
Cao proves that the exchange matrices attached to Cartan string
diagrams for arbitrary double words and triangulations belong to
$\mathfrak P'$ \cite[Theorem 4.8]{Cao}. His matrix $\widehat B$ has
a left symmetrizer, whereas we use $B=-\widehat B^{\mathsf T}$.
Negative transpose commutes with mutation and preserves triangular
and sink--source extensions. Consequently
Corollary~\ref{cor:class-P-basis} gives the same conclusion for these
matrices, with every compatible positive integral choice of degrees.

For the natural coefficients of a decorated double Bott--Samelson
cell, the cluster algebra equals its upper cluster algebra over $\C$
and is the coordinate ring of the cell
\cite[Theorems 1.1 and 4.13]{ShenWeng}. Since
$\mathcal A\subseteq\Mid\subseteq\mathcal U$, the monodromy characters
therefore give a $\C$-basis of that coordinate ring, with the prescribed
frozen variables inverted.

For an acyclic modulation, the chosen integral arrow lattice $\AA$
gives the finite tensor algebra $T_{\RR}(\AA)$, a full order in
$T_R(A)$. For a nonzero potential,
Theorem~\ref{alg:thm:contraction} gives the contracted order.
Section~\ref{sec:geometric-orders} constructs such orders from coherent
sheaves on curves.

The next example separates the ramified character construction from
categorification by finite-group orbits. The matrix without a global unfolding
is classical \cite[Example 14.4 and Remark 14.5]{LZ}; its explicit
reddening sequence gives the required realization by
Corollary~\ref{gm:cor:reddening-finite}.

\Needspace{11\baselineskip}
\begin{example}[A basis without a global unfolding]
\label{ex:nonunfoldable-basis}
Take the vertex fields
\[ K_1=K_3=K,\qquad K_2=K(u),\quad u^3=t,\qquad K_4=K(v),\quad v^2=t, \]
and the full tensor arrow bimodules on the cycle
\[
\xymatrix@C=34mm@R=11mm{
 1\ (K)\ar[r]^a & 2\ (K_2)\ar[d]^b\\
 4\ (K_4)\ar[u]^d & 3\ (K)\ar[l]^c
}
\]
Thus $A_{12}={}_KK_2{}_{K_2}$, $A_{23}={}_{K_2}K_2{}_K$,
$A_{34}={}_KK_4{}_{K_4}$, and $A_{41}={}_{K_4}K_4{}_K$.
The arrow matrix and vertex degrees are
\begin{equation}\label{ex:eq:nonunfoldable-matrix}
 B=\begin{pmatrix}
 0&1&0&-1\\-3&0&3&0\\0&-1&0&1\\2&0&-2&0
 \end{pmatrix},\qquad d=(1,3,1,2).
\end{equation}
This modulation is good by Proposition~\ref{gm:thm:fulltensor}.

\paragraph{\emph{A finite-dimensional realization and the basis.}}
For principal coefficients, set $\widetilde B_0=(B\mid-I_4)$.
Applying \eqref{alg:eq:FZmutation} and \eqref{eq:weight-p} along
$w=(1,3,4,2,1,3,2,1,2,4)$ gives
\begin{equation}\label{ex:eq:nonunfoldable-reaching}
 \mu_w\widetilde B_0=(B\mid I_4),\qquad
 \mu_w(\mathbf e_i,0)=(-\mathbf e_i,0)\quad(1\le i\le4).
\end{equation}
Equation~\eqref{ex:eq:nonunfoldable-reaching} shows that $w$ is a
reddening sequence and verifies the seed reachability condition. Since $A$ is good, Corollary~\ref{gm:cor:reddening-finite}
supplies a Jacobi-finite nondegenerate potential $S$ on $A$.
The finiteness assertion uses \cite[Corollary 5.2]{KellerGreen}:
a nondegenerate potential on an ordinary quiver admitting a reddening
sequence has finite-dimensional completed Jacobian algebra. We apply
this after lifting $w$ to the split quiver, as in
Theorem~\ref{gm:thm:reddening}.

Theorem~\ref{thm:generic-basis} therefore gives a basis of the middle
cluster algebra for every geometric coefficient extension and every
full initial order in $\mathcal P(A,S)$ containing $\RR$.
For principal coefficients, Corollary~\ref{cor:upper-character-basis}
gives a basis of the upper cluster algebra as well.

\paragraph{\emph{The obstruction to an unfolding.}}
The corresponding matrix in the column-exchange convention is
\begin{equation}\label{ex:eq:nonunfoldable-column}
 \widehat B=-B^{\mathsf T}=
 \begin{pmatrix}
 0&3&0&-2\\-1&0&1&0\\0&-3&0&2\\1&0&-1&0
 \end{pmatrix}.
\end{equation}
After interchanging vertices $2$ and $4$, this is the matrix of
\cite[Example 14.4]{LZ} for $(a,b)=(2,3)$. That result excludes finite
global unfoldings; \cite[Remark 14.5]{LZ} attributes this case to
Speyer and Zelevinsky.
The coefficient-free exchanges, in either convention, are
\[
\begin{aligned}
 x_1x_1'=x_2+x_4,\quad x_2x_2'=x_1^3+x_3^3,\quad x_3x_3'=x_2+x_4,\quad x_4x_4'=x_1^2+x_3^2.
\end{aligned}
\]

The same obstruction excludes locally finite unfoldings with infinite
orbit sets in the sense of \cite[Definition 2.4]{HuangLi}. Such an
unfolding is a locally finite skew-symmetric matrix $\widetilde C$ on
disjoint nonempty sets $\Sigma_i$, with no arrows inside a set, satisfying
$\widehat b_{ij}=\sum_{x\in\Sigma_i}\widetilde c_{xy}\quad(y\in\Sigma_j)$.
Between each pair of sets, all nonzero entries have one sign, and these
conditions must persist under simultaneous mutations in the sets.
For $i=1,3$, the entries of \eqref{ex:eq:nonunfoldable-column} give neighbour maps
$f_i:\Sigma_i\to\Sigma_2$ and $g_i:\Sigma_i\to\Sigma_4$, with fibres
of sizes three and two, respectively. After mutating at $2$ and $4$,
the signed number of arrows between $x\in\Sigma_1$ and $y\in\Sigma_3$
is
\[ p_{xy}-q_{xy},\qquad p_{xy}=\mathbf1_{\{f_1(x)=f_3(y)\}},\quad q_{xy}=\mathbf1_{\{g_1(x)=g_3(y)\}}. \]
For fixed $x$, the sum over $y$ is $3-2=1$. The common-sign condition
therefore forces $q_{xy}\le p_{xy}$. Choosing $y$ in a fibre of $g_3$
shows that $f_1$ is constant on the corresponding two-element fibre
of $g_1$. Each three-element fibre of $f_1$ would consequently be a
union of two-element fibres, a contradiction. Only these finite fibres,
not finiteness of the orbit sets, are used.

Hence the exchange pattern cannot arise from the orbit constructions
of \cite[Definition 3.36, Proposition 3.40 and Corollary 3.51]{Demonet}
or \cite[Proposition 6.18]{Wu}. Frozen vertices do not affect the
obstruction, since the principal matrix mutates independently of them.

The cyclic splitting used for the ramified modulation is consistent
with this obstruction. Its orbit sizes are $(1,3,1,2)$ and its orbit
matrix in Demonet's formula is $B$, not $\widehat B$.
Identifying its cluster variables along orbits would give
$x_1x_1'=x_2^3+x_4^2$, whereas the ramified character gives
$x_1x_1'=x_2+x_4$.
\end{example}

\Needspace{9\baselineskip}
\begin{remark}[The dual examples]\label{ex:four-vertex-duality}
Let $B$ and $D=\diag(1,3,1,2)$ be as in
\eqref{ex:eq:nonunfoldable-matrix}, and put
\[ B^\vee=-B^{\mathsf T},\qquad D^\vee=6D^{-1}=\diag(6,2,6,3). \]
Both $BD$ and $B^\vee D^\vee$ are skew-symmetric. The second pair is
the Langlands dual of the first \cite[Appendix A]{GHKK}.
Both matrices use our row convention. The pair
$(B,D)$ has the full tensor realization just constructed, whereas
$(B^\vee,D^\vee)$ is excluded by
Example~\ref{ex:four-vertex-type-obstruction}. Thus existence of a good
modulation with the specified vertex degrees is not preserved by
Langlands duality. For the single exchange pattern of
Example~\ref{ex:nonunfoldable-basis}, the column matrix is instead
$\widehat B=-B^{\mathsf T}$. Its numerical equality with the excluded
arrow matrix does not identify the two exchange patterns.

Neither matrix belongs to $\mathfrak P$. Indeed,
Corollary~\ref{gm:cor:class-P} and the obstruction at degrees
$(6,2,6,3)$ exclude $B^\vee$, and negative transpose preserves
$\mathfrak P$. Thus the good realization of $B$ is not a consequence
of the extension theorem.
\end{remark}

\section{Factorization and geometric orders}\label{sec:factorization}
The basis theorem completes the main construction. This section and
the next two study the arithmetic information retained by its characters:
factorization and geometric orders, nonzero-potential examples, and
order dependence with valued transport.

For representations over the same valued field, characters multiply
under direct sums. Universal presentations initialized separately at
different weights require an additional comparison: the valuations on
the summands must agree with the valuation induced from the total
presentation space. We give a finite criterion for this comparison,
apply it to the one-arrow species, and then realize its order by
coherent sheaves.
\subsection{A finite criterion for arithmetic factorization}
\label{fac:subsec:criterion}
Generic dominance of a direct-sum map is weaker than preservation of
the Gauss valuations on its source and target. The following criterion
characterizes the latter property for an integral direct-sum map.

Fix a coordinatewise sign-coherent sum $\delta=\delta_1+\cdots+\delta_r$.
Write $P_{a,+}=\PP_0([\delta_a]_+)$ and
$P_{a,-}=\PP_0([-\delta_a]_+)$. Take independent universal presentations
$f_a:P_{a,-}\to P_{a,+}$ over $\OO[\mathbf z]$. For $a\ne b$, form
\begin{equation}\label{fac:eq:linear-map}
 \begin{split}
 D_{ab}:\Hom(P_{a,+},P_{b,+})\oplus\Hom(P_{a,-},P_{b,-})
       &\longrightarrow\Hom(P_{a,-},P_{b,+}),\\
 (u,v)&\longmapsto uf_a-f_bv.
 \end{split}
\end{equation}
When the codomain is zero, this map is automatically surjective.
All Hom spaces here are integral projective Hom lattices, extended to
$\OO[\mathbf z]$. Put
$\mathcal V=\OO[\mathbf z]_{(t)},\qquad \mathcal E_{ab}=\coker(D_{ab}\otimes\mathcal V)$.
The ring $\mathcal V$ is a discrete valuation ring with uniformizer $t$
and residue field $\C(\mathbf z)$.

Let $G$ be the product of the automorphism groups of the total source
and target projectives, and let $\mathscr X_\eta$ be the affine space
of the integral presentation lattice of weight $\eta$. Consider
\[ \Phi:G\times\prod_a\mathscr X_{\delta_a}\longrightarrow\mathscr X_\delta, \qquad (g_+,g_-,(f_a))\longmapsto g_+\Bigl(\bigoplus_a f_a\Bigr)g_-^{-1}. \]
Both the source and target are smooth over $\OO$; a unit group is the
principal open given by the determinant of multiplication in its finite
free endomorphism algebra.

\begin{theorem}\label{prop:integral-direct-sum}\label{fac:thm:criterion}
The following conditions are equivalent.
\begin{enumerate}
\item Every $\mathcal E_{ab}$, $a\ne b$, is zero.
\item Each reduction $\overline D_{ab}$ is surjective over
$\C(\mathbf z)$; equivalently, its minors of order equal to the
dimension of its codomain do not all vanish.
\item The reduction $\overline\Phi$ is dominant. Equivalently,
$\Phi$ is smooth on an open subset containing the generic point of
its source special fibre.
\item Give the presentation coefficients and the coordinates of the two
changes of bases their Gauss valuation. Then $\Phi^*$ embeds the target function field in the source
function field and its valuation restricts to the target Gauss valuation.
\end{enumerate}
When these conditions hold, the total universal decorated
representation becomes the direct sum of the universal
representations of the summands over a common valued field of ramification index one.
Consequently
\begin{equation}\label{fac:eq:factorization}
 F_{\delta;\PP_0}=\prod_a F_{\delta_a;\PP_0}.
\end{equation}
For each fixed decomposition, the criterion involves finitely many
Hom maps over the chosen order. The equivalence concerns preservation
of the Gauss valuation by this direct-sum morphism.
\end{theorem}
\begin{proof}
Nakayama's lemma gives (i)$\Leftrightarrow$(ii). At the direct sum
$f=\bigoplus_a f_a$, the tangent space of the target decomposes as
\[ \Hom\!\left(\bigoplus_aP_{a,-},\bigoplus_bP_{b,+}\right) =\bigoplus_{a,b}\Hom(P_{a,-},P_{b,+}). \]
Varying $f_a$ supplies the summands with $a=b$. In the summand indexed
by $a\ne b$, the differential of the changes of bases is
$D_{ab}$ in \eqref{fac:eq:linear-map}. This description also holds
modulo $t$. In characteristic zero, a morphism between smooth
irreducible varieties is dominant if and only if its differential is
surjective on a nonempty open subset. Changes of presentation bases
preserve the rank of this differential, proving
(ii)$\Leftrightarrow$(iii). The relative Jacobian criterion gives
the asserted smooth open subset over $\OO$.

For (iii)$\Rightarrow$(iv), let $H$ be a nonzero target polynomial.
Multiply $H$ by a power of $t$ so that its coefficients are integral
and at least one has valuation zero.
Its residue is nonzero and remains nonzero under the dominant
$\overline\Phi$. Thus
$v_G(\Phi^*H)=v_{\rm Gauss}(H)$.
The inverses in $\Phi$ have denominators with nonzero generic residue;
the argument extends to rational functions. Conversely, equality of the
valuations prevents a nonzero residue polynomial from vanishing after
pullback, so the residue function-field map is injective and
$\overline\Phi$ is dominant.

Over the source field, the pullback of the universal target map is
the direct sum of the $f_a$, up to the two changes of presentation bases. The embeddings of the target and summand
fields preserve the parameter $t$, so completion with geometric residue
fields preserves the chosen tame generator. Proposition~\ref{prop:mon-directsum} therefore gives
\eqref{fac:eq:factorization}.
\end{proof}

If the corresponding spaces $\op{E}$ for independently general presentations vanish, $\mathcal E_{ab}$ is
finite torsion over $\mathcal V$. After choosing bases in its source and target, a map of free
$\mathcal V$-modules has diagonal Smith normal form, whose nonzero
entries are units times powers of $t$. Its nonunit factors are an
arithmetic obstruction to this comparison: their absence is equivalent
to the theorem's conditions. More generally, if the nonzero invariant
factors of $D_{ab}$ are $t^{a_1},\ldots,t^{a_q}$, then
\[ \dim\coker\overline D_{ab} =\dim\coker(D_{ab}\otimes\Frac\mathcal V) +\#\{j:a_j>0\}. \]
The rank defect counts the nonunit invariant factors, rather than the
length of the torsion module. The ordinary criterion for the vanishing of $\op{E}$ on general presentations
\cite[Theorem 4.4]{DF} therefore controls a different condition: generic
dominance. Equality of the monodromy $F$-polynomials does not by itself establish
preservation of the Gauss valuation by this direct-sum map.

\begin{corollary}\label{fac:cor:powers}
Fix $\eta$. If the off-diagonal cokernels in
Theorem~\ref{fac:thm:criterion} vanish for two independently general
presentations of weight
$\eta$, then, for every $r\ge1$,
$F_{r\eta;\PP_0}=F_{\eta;\PP_0}^{\,r}$.
The extended characters with zero initial frozen weights satisfy the
same power identity. 
\end{corollary}
\begin{proof}
For $r$ independently general presentations, choose a nonzero maximal
minor for each ordered pair. Their product is nonzero in the polynomial
ring of the presentation coefficients. All the required surjectivity
conditions therefore hold on one nonempty residue open. Apply the theorem and add the leading weights. The same argument
handles any finite sign-coherent collection of weight types satisfying
the cokernel conditions for every ordered pair.
\end{proof}

\paragraph{\em Factorization and canonical decomposition.}
A \emph{canonical decomposition} of a presentation weight is the
multiset of weights of the indecomposable summands of a general
presentation. The ordinary criterion is vanishing of the directed
spaces $\op{E}$ between independently general summands \cite[Theorem 4.4]{DF}.
Theorem~\ref{fac:thm:criterion} adds the valuation condition needed
for characters defined by the order on the separate presentation spaces.

\begin{corollary}\label{cor:canonical-character}
Let $\delta=\delta_1+\cdots+\delta_s$ be a sign-coherent decomposition
satisfying Theorem~\ref{prop:integral-direct-sum}. Then
\[ F_{\delta;\PP_0}= \prod_{a=1}^sF_{\delta_a;\PP_0}. \]
For any extended weights with
$\widetilde\delta=\sum_a\widetilde\delta_a$,
\[ \mathsf G^{\PP_0}_{\widetilde\delta} =\prod_{a=1}^s\mathsf G^{\PP_0}_{\widetilde\delta_a}. \]
The Newton polytopes add by Minkowski sum
$N_1+N_2=\{u+v:u\in N_1,\ v\in N_2\}$, and their support
functions add. The injective assertions follow by opposite duality.
\end{corollary}
\begin{proof}
Theorem~\ref{prop:integral-direct-sum} realizes the total universal
presentation as the sum of the summand universal presentations over a
common valued extension with uniformizer $t$. The inertia generator
is unchanged. Apply Proposition~\ref{prop:mon-directsum} and multiply the
leading monomials. The Newton polytope of a nonzero product is the
Minkowski sum of the Newton polytopes; its support function is their sum.

\end{proof}

\subsection{The one-arrow species}
\label{app:subsec:rank-two-basis}
We now compute the basis selected by the integral tensor algebra of
Example~\ref{run:ex:one-arrow}. The species is unchanged:
\[ \vcenter{\hbox{$\displaystyle \xymatrix@C=45pt{1\ar[r]^{a}&2}$}}, \qquad (K_1,K_2)=(K,E). \]
The coefficient valuation is the Gauss valuation of the order,
rather than one chosen from the preceding valued family. We begin
at $\eta=(2,-1)$ and then pass to its multiples. The data are
\begin{equation}\label{app:eq:rank-two-data}
 K=\C((t)),\quad E=K(s),\ s^4=t,\quad
 R=K\times E,\quad A_{12}={}_KE_E,\quad S=0,
 \qquad B=\begin{pmatrix}0&1\\-4&0\end{pmatrix}.
\end{equation}
Let $P=R\oplus A_{12}$ and take
$\PP_*=\OO e_1\oplus\OO_Ee_2\oplus a\OO_E$.
The algebra $P$ has dimension nine over $K$. Mutation reverses its single arrow,
so the modulation is good and the zero potential is Jacobi-finite and
nondegenerate. After the universal coefficient
extension, write $L=\Omega((t))$ and $E_L=\Omega((s))$, with $\Omega$
algebraically closed. Representations and endomorphism dimensions in the
calculation below are over this valued field $L$.

\paragraph{\em The primitive nonrigid weight.}
Put $\eta=(2,-1)$. The universal presentation $P_2\to P_1^2$ is
given by two independently general elements of $\OO_E$. The second has Gauss valuation zero, hence is a unit in the Gauss
valuation ring. Changing the source basis by its inverse
gives
\begin{equation}\label{app:eq:normalized-column}
 f_z=\binom{-z}{1},\qquad
 z=z_0+z_1s+z_2s^2+z_3s^3.
\end{equation}
The $z_i$ have independent residues. Indeed, replacing the two original
entries by their ratio and the second entry is a birational change over
$\OO$ on this unit open, and its reduction is also birational. These
coordinates therefore describe the same universal presentation on a
nonempty open subset.
Its decorated cokernel $M_z$ has no decoration, ranks $(2,1)$, and arrow
matrix $[1\;z]:E_L^2\to E_L$.

For $e_2=0$, every subspace of $L^2$ can be the scalar kernel. At
quotient rank $e_1$ its Grassmannian is $\Gr_{2-e_1}(L^2)$, and its
inertia trace is $\binom2{e_1}$.
For $e_2=1$, the kernels of the four conjugate maps at the second vertex
have zero intersection, since the conjugates of $z$ are distinct.
Thus $e_1=2$, and there is one quotient. Consequently
\begin{equation}\label{app:eq:Feta}
 F_{\eta;\PP_*}(y_1,y_2)=(1+y_1)^2+y_1^2y_2.
\end{equation}

The representation is nonrigid. An endomorphism is given by
$A\in\op{Mat}_2(L)$ and multiplication by $\lambda\in E_L$,
with $\lambda(1,z)=(1,z)A$. Hence $\lambda=a+bz$ and
$b z^2+(a-d)z-c=0$ for some $a,b,c,d\in L$. Generic $z$ has degree four over $L$,
so $b=c=0$ and $a=d$; therefore $\End(M_z)=L$.
The presentation is injective, and
$\langle\eta,(2,1)\rangle_D=0$. Equation~\eqref{eq:E-dimension} gives
$\op{e}(M_z)=1$, proving that $M_z$ is not $E$-rigid.

\paragraph{\em The integral direct-sum calculation.}
For two independent normalized presentations $f_z,f_w$, quotienting
$\Hom(P_{2,L},P_{1,L}^2)=E_L^2$ by $f_wE_L$ identifies the remaining off-diagonal
map with
\[ \op{Mat}_2(L)\longrightarrow E_L, \qquad A\longmapsto(1,w)A\binom{-z}{1}. \]
Its image is the $L$-span of $1,z,w,zw$. The integral calculation is the
same. To apply Theorem~\ref{fac:thm:criterion}, reduce $1,z,w,zw$ modulo $t$ and express them in the basis
$(1,s,s^2,s^3)$ of $\C[s]/(s^4)$. The determinant of their coefficient
matrix gives the required minor. Writing $z=\sum u_is^i$ and $w=\sum v_is^i$ gives
\begin{equation}\label{app:eq:pairdet}
 \det[1,z,w,zw]
 =u_1^2(v_2^2-v_1v_3)+v_1^2(u_1u_3-u_2^2).
\end{equation}
This determinant polynomial is nonzero: at $z=s$ and $w=s^2$ it
has value one. The cokernels for both ordered pairs therefore vanish
on a common dense residue open. For any finite number
of summands one can even use the simultaneous residue specializations
$z_j=s+j s^2$, $j=1,\ldots,r$, for which the pair determinant is
$j^2-i^2\ne0$. These specializations give a nonzero minor; the universal
families still have independent parameters.

\begin{proposition}\label{app:thm:ray}
For every integer $r\ge0$,
\begin{equation}\label{app:eq:ray-F}
 F_{r\eta;\PP_*}(y_1,y_2)
       =\bigl((1+y_1)^2+y_1^2y_2\bigr)^r.
\end{equation}
After a coefficient-field extension of ramification index one, the
universal representation is a direct sum of $r$ independently general
rank-$(2,1)$ representations. In particular
\begin{equation}\label{app:eq:ray-coeff}
 [y_1^a y_2^b]F_{r\eta;\PP_*}
 =\binom rb\binom{2(r-b)}{a-2b},
\end{equation}
with a binomial coefficient interpreted as zero outside its usual range.
For the coefficient-free cluster algebra, its character is $Z^r$, where
\begin{equation}\label{app:eq:Z}
 Z=\frac{(x_2+1)^2+x_1^4}{x_1^2x_2}.
\end{equation}
All these assertions hold for every full initial order in $P$ containing
the distinguished vertex order.
\end{proposition}
\begin{proof}
The determinant computation and Corollary~\ref{fac:cor:powers} give the
unramified decomposition and \eqref{app:eq:ray-F}. Expanding by the
number $b$ of factors contributing $y_1^2y_2$ gives
\eqref{app:eq:ray-coeff}. Substitute
$\widehat y_1=x_2^{-1}$ and $\widehat y_2=x_1^4$ and multiply by
$x^{-\eta}=x_1^{-2}x_2$ to obtain \eqref{app:eq:Z} and its powers.

For the last assertion, let $\Lambda$ be a full order. Integral
closedness gives $e_1\Lambda e_1=\OO$ and $e_2\Lambda e_2=\OO_E$.
The module $e_1\Lambda e_2$ is a full $\OO_E$-lattice in $E$, hence
is $s^q\OO_E$
for some $q\in\Z$. Scaling the arrow over $K$ by $s^q$ is an $R$-algebra
automorphism carrying $\PP_*$ to this order. It identifies the integral
presentation lattices and the representation Grassmannians with their inertia
actions. Thus $\PP_*$ determines the same character as every full
initial order in this particular algebra.
\end{proof}

For $r=2$, the formula reads
\begin{equation}\label{app:eq:Fdouble}
 F_{2\eta;\PP_*}
 =1+4y_1+6y_1^2+4y_1^3+y_1^4
   +2y_1^2y_2+4y_1^3y_2+2y_1^4y_2+y_1^4y_2^2.
\end{equation}
The coefficient at $y_1^4y_2$ illustrates ramification directly. That
representation Grassmannian is $\op{Res}_{E_L/L}\mathbb P^1$, geometrically
$(\mathbb P^1)^4$. Its ordinary geometric Euler characteristic is $16$,
but the tame generator cyclically permutes the four factors, leaving
exactly two Schubert classes fixed. Its trace is $2$, as in
\eqref{app:eq:Fdouble}.

At the different coefficient $y_1^2y_2$, the value is also $2$.
Example~\ref{run:ex:one-arrow} has the same presentation weight $2\eta$, but its
valued presentations with odd discriminant valuation give zero there.
Thus the odd members cannot be universal presentations of a full
initial order in this one-arrow algebra.

\paragraph{\em Mutation at the sink.}
At vertex $2$ the incoming map is $[1\;z]$, with kernel spanned by
$(-z,1)$. The reflected rank remains $(2,1)$, and the weight is
$-\eta=(-2,1)$ for the matrix $-B$. The reversed arrow action is adjoint
to this inclusion. Its images span $L^2$: this follows
from the nondegenerate field pairing and linear independence of $1,z$.
A nonzero quotient must therefore have $e_2=1$; then its quotient at the first vertex
is arbitrary. Directly,
\[ F_{\mu_2M_z}(y'_1,y'_2)=1+y'_2(1+y'_1)^2. \]
The two polynomials satisfy
$(1+y'_2)F_{M_z}\left( \frac{y'_1y'_2}{1+y'_2},(y'_2)^{-1}\right) =1+y'_2(1+y'_1)^2$.
Direct sums give its $r$-th power as well. Mutation at $1$
gives the same reflected polynomial by the character formula.
With coefficient-free exchange $x_2x'_2=1+x_1^4$, the equality of
Laurent characters also follows directly from \eqref{app:eq:Z}.

For principal extension $(B\mid-I_2)$ and frozen variables $p,q$, the
initial ray character has the explicit form
\[ \mathsf G_{(r\eta,0,0)}^{\PP_*} =\left(\frac{(x_2+p)^2+p^2q x_1^4}{x_1^2x_2}\right)^r. \]
Here $p,q$ are coefficient variables. Setting them equal to one gives
the coefficient-free formula; the basis assertion below is proved
directly for $B$.

\paragraph{\em Cluster monomials.}
We next describe the characters at all remaining weights.
\begin{lemma}\label{app:lem:rank-two}
For the matrix in \eqref{app:eq:rank-two-data}, every
$\delta\in\Z^2$ except $r(2,-1)$, $r>0$, has a finite mutation sequence
with terminal weight coordinatewise nonpositive. The exceptional ray
never has such a sequence.
\end{lemma}
\begin{proof}
Every mutation negates $B$. If the second coordinate is nonnegative,
mutation at the sink $2$, followed when needed by the sink $1$, gives
a nonpositive weight. We may therefore assume
$\delta=(a,-c)$ with $a,c>0$; weights with both coordinates nonpositive
need no mutation. Put $h=a-2c$.

If $h<0$ and $a\le c$, mutation at $1$ already gives
$(-a,a-c)\le0$. If $h<0$ and $a>c$, mutate at $1$ and then $2$.
The matrix returns to $B$ and the weight is
$(3a-4c,c-a)=(a+2h,-(c+h))$.
It is nonpositive when its first coordinate is nonpositive. Otherwise
it is again of the same mixed form, with the same $h$ and strictly
smaller positive $c$.

If $h>0$ and $a\ge4c$, mutation at $2$ gives a nonnegative weight;
the successive sink mutations for the resulting orientation make it
nonpositive. If $2c<a<4c$, mutate at $2$ and then $1$. The resulting
weight for $B$ is
$(4c-a,a-3c)$.
For $a\ge3c$ it is nonnegative and is handled by the sink sequence. For
$a<3c$ it has the original mixed form with
$c'=3c-a=c-h<c$ and the same $h$. Descent of the positive integer
$c$ proves termination in both cases.

If $h=0$, either mutation carries $(B,r\eta)$ to $(-B,-r\eta)$,
and either next mutation carries it back. Both weights have mixed
signs. This proves the exceptional assertion for every sequence.
\end{proof}

\begin{proposition}\label{app:thm:complete-basis}
Let $\mathrm{CM}(B)$ denote the set of distinct coefficient-free cluster
monomials, including $1$. For the one-arrow species
\eqref{app:eq:rank-two-data} and any full initial order containing the
vertex order, the basis of Theorem~\ref{thm:generic-basis} is exactly
\[ \mathrm{CM}(B)\ \sqcup\ \{Z^r:r\ge1\}. \]
None of the $Z^r$, $r>0$, is a cluster monomial.
\end{proposition}
\begin{proof}
The modulation is acyclic and good, and the sink sequence $(2,1)$
verifies the seed reachability condition.
Theorem~\ref{thm:generic-basis} and
Corollary~\ref{cor:upper-character-basis} therefore give a basis of the
upper cluster algebra, since $\det B=4$.

By Lemma~\ref{app:lem:rank-two}, a weight off the exceptional ray
transports to a nonpositive weight. Its generic presentation there has
zero positive representation and only decoration; its character is a
cluster monomial. Conversely the general basis theorem contains every
cluster monomial. On the exceptional ray Proposition~\ref{app:thm:ray}
gives exactly $Z^r$. Along every sequence its weight remains mixed, so it
cannot be a cluster monomial: at a seed containing such a monomial its
pointed mutable degree would correspond to a nonpositive presentation
weight. Full rank identifies that degree uniquely. The same argument
shows the union is disjoint.
\end{proof}

\paragraph{\em Comparison with the classical rank-two basis.}
The basis can also be obtained from the classical rank-two description;
the calculation above identifies the initial order that selects it. In the notation
of \cite[Theorem 2.8 and equation (2.8)]{SZ}, set their consecutive
variables $(y_1,y_2)=(x_2,x_1)$. Their element
$z=y_0^2y_3-(y_1+2)y_2^2$ becomes exactly \eqref{app:eq:Z}, using
$y_0=(x_2+1)/x_1$ and $y_3=(x_1^4+1)/x_2$. Their basis is
$\mathrm{CM}(B)\ \sqcup\ \{T_r(Z):r\ge1\}$,
where $T_r$ is the normalized Chebyshev polynomial of the first kind,
monic of degree $r$; for example $T_1(Z)=Z$ and $T_2(Z)=Z^2-2$.
Replacing these monic polynomials successively by $Z^r$ is an invertible
integral triangular basis change, allowing the constant $1$, already a
cluster monomial. Together with $\mathcal A=\mathcal U$ in this rank-two
case \cite[equation (2.2)]{SZ}, this gives another proof that the
power family is a basis.

\subsection{Endomorphism orders of coherent sheaves}
\label{sec:geometric-orders}
Let $X$ be proper over $\OO=\C[[t]]$, and let
$\mathcal T=\bigoplus_{i\in I}\mathcal T_i$ be a \emph{coherent sheaf},
flat over $\OO$. Here coherent means locally finitely presented, and
flatness over this discrete valuation ring means that multiplication by
$t$ is injective. We take the opposite endomorphism algebra, so that
$e_i\End_X(\mathcal T)^{\mathrm{op}}e_j=\Hom_X(\mathcal T_i,\mathcal T_j)$,
in agreement with our convention for multiplying arrows.

\begin{lemma}\label{lem:geometric-order}We put
$\mathscr P_X=\End_X(\mathcal T)^{\mathrm{op}},\ P_X=\End_{X_K}(\mathcal T_K)^{\mathrm{op}}$, and $X_K=X\times_{\OO}\Spec K$.
Then $\mathscr P_X$ is a full finite free $\OO$-order in $P_X$.
For $a=(a_i)\in\Z_{\geq0}^{I}$ write
$\mathcal T(a)=\bigoplus_i\mathcal T_i^{a_i}$ and
$\mathscr P_X(a)=\bigoplus_i(e_i\mathscr P_X)^{a_i}$. There is a natural
identification
\begin{equation}\label{eq:geometric-presentation-lattice}
 \Hom_{\mathscr P_X}\bigl(\mathscr P_X(b),\mathscr P_X(a)\bigr)
 \simeq \Hom_X\bigl(\mathcal T(a),\mathcal T(b)\bigr).
\end{equation}
In particular, a decomposition $\mathcal T=\bigoplus_i\mathcal T_i$
whose opposite endomorphism algebra after tensoring with $K$ is $P$ and whose diagonal
endomorphism rings are the prescribed $\OO_i$ determines an initial
order and all its presentation lattices.
\end{lemma}
\begin{proof}
The sheaf $\mathcal{H}\!om_X(\mathcal T,\mathcal T)$ is coherent.
Properness makes its global sections a finite $\OO$-module
\cite[Tags 01XY and 02O6]{Stacks}. If $tf=0$ for an endomorphism $f$,
flatness of $\mathcal T$ gives $f=0$. Thus $\mathscr P_X$ is torsion-free
and hence free over $\OO$. A local finite presentation shows that
$\mathcal{H}\!om$ commutes with the flat extension $\OO\to K$; flat base
change for global sections then identifies
$\mathscr P_X\otimes_{\OO}K$ with $P_X$
\cite[Tag 02KH]{Stacks}. This proves fullness.

The functor $\Hom_X(-,\mathcal T)$ is a contravariant equivalence
from the additive category of summands of finite sums of $\mathcal T$
to finite projective right $\mathscr P_X$-modules.
For $\mathcal T_i,\mathcal T_j$ this is the identity
$e_i\mathscr P_Xe_j=\Hom_X(\mathcal T_i,\mathcal T_j)$;
finite sums and idempotent summands give
\eqref{eq:geometric-presentation-lattice}.
\end{proof}

Thus the morphisms of the specified sheaves supply the integral
presentation coefficients. Write
\[ H=\Hom_X(\mathcal T([\delta]_+),\mathcal T([-\delta]_+)),\qquad H^\vee=\Hom_{\OO}(H,\OO). \]
The associated affine space is $\Spec\op{Sym}_{\OO}H^\vee$.
The local ring at the generic point of its special fibre defines the
Gauss valuation. In any basis of $H$, that ring is
$\OO[z_1,\ldots,z_q]_{(t)}$. Tensoring every $\mathcal T_i$ by the same
line bundle gives a canonical isomorphism of the endomorphism orders.

\subsection{Ramified divisors and star-shaped species}
\label{sec:divisor-orders}
The preceding construction is explicit on a curve. Let $C$ be a smooth
projective connected curve over $\C$, put $X=C\times\Spec\OO$, and let
$j_i:D_i\hookrightarrow X$, $1\leq i\leq r$, be pairwise disjoint
effective Cartier divisors, finite and flat over $\OO$, with
\[ \Gamma(D_i,\OO_{D_i})=\OO_i=\C[[s_i]],\qquad s_i^{d_i}=t. \]
Here $\Gamma$ denotes global sections. An effective Cartier divisor
is a closed subscheme whose ideal is locally generated by one
non-zero-divisor. A line bundle is a locally free sheaf of rank one.
Choose a line bundle $\mathcal L$ on $X$ and a line bundle $\mathcal N_i$
on each $D_i$. The direct image $(j_i)_*\mathcal N_i$ is the sheaf on
$X$ obtained by extending $\mathcal N_i$ by zero outside $D_i$.
Set
\[ \mathcal T_0=\mathcal L,\qquad \mathcal T_i=(j_i)_*\mathcal N_i\ (i\geq1),\qquad \mathcal T=\bigoplus_{i=0}^r\mathcal T_i. \]
These data exist for all $d_i\geq1$. For instance, take
$X=\mathbb P^1_{\OO}$, choose distinct $c_i\in\C$, and define
$D_i$ in the affine coordinate $z$ by $(z-c_i)^{d_i}=t$.
The homogeneous equation has no point at infinity. Its coordinate ring
is $\OO[s_i]/(s_i^{d_i}-t)=\C[[s_i]]$, with $s_i=z-c_i$.
The divisors are disjoint because their defining polynomials are
relatively prime modulo $t$.

\begin{theorem}\label{thm:divisor-order}
For $\mathscr P_X=\End_X(\mathcal T)^{\mathrm{op}}$, we have
$e_i\mathscr P_Xe_i=\OO_i$ for $0\le i\le r$, where $\OO_0=\OO$.
The only nonzero off-diagonal components are
\begin{equation}\label{eq:divisor-hom-modules}
 e_0\mathscr P_Xe_i=\Lambda_i,
 \qquad
 \Lambda_i=\Gamma(D_i,\mathcal N_i\otimes j_i^*\mathcal L^{-1}),
 \qquad 1\leq i\leq r,
\end{equation}
and each $\Lambda_i$ is free of rank one over $\OO_i$. The algebra $P_X=\mathscr P_X\otimes_{\OO}K$ is the path algebra of the
species $0\to i$ with arrow bimodule ${}_KK_i{}_{K_i}$, and its matrix is
\[ b_{0i}=1,\qquad b_{i0}=-d_i,\qquad b_{ij}=0\quad(i,j\ne0). \]
For this $K$-algebra with its fixed vertex order, every
full initial order gives the same polynomial $F_{\delta}$ at every
weight $\delta$. For each geometric coefficient extension, these
characters form one basis of the middle cluster algebra; with principal
coefficients, it is also a basis of the upper cluster algebra. These
bases are independent of the curve, divisors, and line bundles used in
the construction, after the vertex fields have been identified.
\end{theorem}
\begin{proof}
Since $C$ is connected and projective,
$\End_X(\mathcal L)=\Gamma(X,\OO_X)=\OO$. Full faithfulness of the
pushforward along $j_i$ gives
$\End_X((j_i)_*\mathcal N_i)=\OO_i$. A morphism from $(j_i)_*\mathcal N_i$
to $\mathcal L$ is locally annihilated by an equation of $D_i$, whereas
multiplication by that equation is injective on $\mathcal L$; hence the
morphism is zero. Disjoint supports give zero Hom spaces between
distinct divisors. The remaining Hom space is
\eqref{eq:divisor-hom-modules}. Since $\OO_i$ is local, the module $\Lambda_i$ is free of rank one.
Choosing one generator for each $\Lambda_i$ identifies $\mathscr P_X$
with the integral tensor algebra of this species.

Let $\mathscr Q\subset P_X$ be any full order containing the vertex
order. The algebra $e_i\mathscr Qe_i$ is finite over $\OO$, contained
in $K_i$, and contains $\OO_i$. Integral closedness therefore gives
$e_i\mathscr Qe_i=\OO_i$. The remaining nonzero components are full
$\OO_i$-lattices in the one-dimensional $K_i$-spaces
$e_0P_Xe_i$. Each is $s_i^{q_i}\OO_i$ for some integer $q_i$.
Independent multiplication of the arrows by $s_i^{q_i}$ extends to an
$R$-algebra automorphism of $P_X$, since every product of two arrows is
zero. It carries the integral tensor algebra to $\mathscr Q$ and identifies
their integral presentation lattices. Proposition~\ref{prop:universal-invariance}
and the induced isomorphisms of representation Grassmannians give equality of all
$F$-polynomials.

Each arrow is the full tensor $K\otimes_KK_i$, so the modulation is good
by Proposition~\ref{gm:thm:fulltensor}. It is acyclic and has zero
potential. Corollary~\ref{cor:class-P-basis} gives a basis of the middle
cluster algebra for every geometric coefficient extension;
Corollary~\ref{cor:upper-character-basis} gives the upper-algebra conclusion
with principal coefficients. The preceding order comparison shows that,
for fixed coefficients, this basis depends only on the species and
the identified vertex fields.
\end{proof}

\begin{example}\label{sec:elementary-transformations}\label{cor:divisor-basis}
	Let $j:D\hookrightarrow X=\mathbb P^1_{\OO}$ be a ramified divisor of
	degree four and put $\mathcal T=\OO_X\oplus j_*\OO_D$.
	Writing $E=K(s)$, $s^4=t$, we have
	\[
	\mathscr P_X:=\End_X(\mathcal T)^{\mathrm{op}}
	\simeq
	\begin{pmatrix}\OO&\OO_E\\0&\OO_E\end{pmatrix},
	\qquad
	B=\begin{pmatrix}0&1\\-4&0\end{pmatrix}.
	\]
	Thus the geometry supplies precisely the standard order for the
	degree-$(1,4)$ species.
	
	At the primitive weight $\eta=(2,-1)$, presentations correspond to
	morphisms $\OO_X^2\to j_*\OO_D$. After a scalar extension
	$L=\Omega((t))$, a general one has the form $(-\lambda,1)$, where
	$\lambda=\lambda_0+\lambda_1s+\lambda_2s^2+\lambda_3s^3$
	and the $\lambda_i$ have independent residues. It gives the elementary
	transformation
	\begin{equation}\label{eq:elementary-transformation}
		0\longrightarrow\mathcal V_\lambda\longrightarrow\OO_{X_L}^{\,2}
		\xrightarrow{(-\lambda,1)}j_*\OO_{D_L}\longrightarrow0.
	\end{equation}
	The kernel in \eqref{eq:elementary-transformation} is locally free;
if $h$ cuts out $D_L$ and
	$\widetilde\lambda$ lifts $\lambda$, it is locally generated by
	$(1,\widetilde\lambda)$ and $(0,h)$.
	
	Under \eqref{eq:geometric-presentation-lattice}, this morphism is the
	presentation $P_2\to P_1^2$ with arrow matrix $[1\ \lambda]$, and its
	cluster character is $Z=\frac{(x_2+1)^2+x_1^4}{x_1^2x_2}$
	by \eqref{app:eq:Feta}. For independent elementary transformations,
	the determinant \eqref{app:eq:pairdet} has nonzero reduction, so
	Theorem~\ref{fac:thm:criterion} and multiplicativity give the character
	$Z^r$ at weight $r\eta$. Together with
	Proposition~\ref{app:thm:complete-basis}, this yields
	\begin{equation}\label{eq:divisor-basis}
		\mathrm{CM}(B)\ \sqcup\ \{Z^r:r\geq1\}.
	\end{equation}
	Equivalently, among families agreeing with cluster monomials on
	reachable weights, this basis is determined by $b_\eta=Z$ and
	$b_{r\eta}b_{s\eta}=b_{(r+s)\eta}$.
	
	The divisor therefore determines the order and the basis
\eqref{eq:divisor-basis} before any presentation coordinates are chosen. The same
	construction works on an arbitrary smooth projective curve, with the
	line bundles allowed in Theorem~\ref{thm:divisor-order}.
\end{example}

\section{Examples with nonzero potentials}\label{sec:nonzero}
These examples distinguish rigidity of the potential from rigidity
of its representations. Both use a nonzero second-derivative map.
The cubic example extends the degree-four order calculation: its
potential is rigid, but its general representations are not. The
surface example has a nonrigid potential and illustrates character
mutation without the reachability hypothesis needed for the basis
theorem.

\subsection{A cubic mutation with nonrigid general presentations}
We extend the order construction of
Section~\ref{app:subsec:rank-two-basis} by inserting a second scalar
vertex and the identity map on its representation space. The same
presentation parameters now give an example with nonzero potential.
Keep $E=K(s)$, $s^4=t$, and begin with
\begin{equation}\label{eq:cubic-root}
 \begin{gathered}
 \vcenter{\hbox{$\displaystyle \xymatrix@C=30pt{1\ar[r]^a&2\ar[r]^b&3}$}}\\
 (K_1,K_2,K_3)=(K,K,E),\qquad S_0=0
 \end{gathered}
 \qquad
 B_0=\begin{pmatrix}0&1&0\\-1&0&1\\0&-4&0\end{pmatrix}.
\end{equation}
The two arrows have bimodules $K$ and ${}_KE_E$. Every multiplicity vector has
length one, so the modulation is good. The integral tensor algebra and principal
coefficients satisfy Corollaries~\ref{cor:class-P-basis} and~\ref{cor:upper-character-basis}.
Mutation at $2$ gives the triangle
\[
\vcenter{\hbox{$\displaystyle \xymatrix@C=30pt@R=20pt{&2\ar[dl]_{a^*}&\\1\ar[rr]_c&&3\ar[ul]_{b^*}}$}}
 \qquad
 B_1=\begin{pmatrix}0&-1&1\\1&0&-1\\-4&4&0\end{pmatrix}.
\]
Its potential is the canonical cubic coevaluation term. Explicitly, on
splitting $E/K$, it is
\[ S_1=\sum_{j=0}^3 c_j b_j^*a^*. \]
The sheets and arrows in this expression are cyclically permuted by
Galois, so the expression descends to the specified ramified modulation.
There is no quadratic pair in this first mutation. The potential is
Jacobi-finite and nondegenerate, being the mutation of
\eqref{eq:cubic-root}.

Let $L=\Omega((t))$ be the valued field of the universal presentation and write
$E_L=\Omega((s))$. Use the same expression $z=z_0+z_1s+z_2s^2+z_3s^3$ as in
\eqref{app:eq:normalized-column}, with four independent residue
coordinates, and take
\[ M_z:\quad\vcenter{\hbox{$\displaystyle \xymatrix@C=45pt{L^2\ar[r]^I&L^2\ar[r]^{[1\ z]}&E_L}$}}. \]
The second map denotes its $E_L$-linear extension. This is the cokernel
of the map $P_3\to P_1^2$ given by $(-z,1)^{\mathsf T}$, with weight
$\xi=(2,0,-1)$. Its integral presentation matrices are exactly those in
\eqref{app:eq:normalized-column}.
Its endomorphism algebra is $L$, and
$\langle\xi,(2,2,1)\rangle_D=0$; hence $\op{e}(M_z)=1$.

Set $U_z=\op{span}_L\{1,z\}\subset E_L$ and
$W_z=E_L/U_z$, and let $q:E_L\twoheadrightarrow W_z$ be the quotient map.
The original local triangle at $2$ has
$\alpha=I$, $\beta(x,y)=x+zy$, and $\gamma=0$.
The mutated triangle on $N_z=\mu_2M_z$ is
\begin{equation}\label{eq:cubic-triangle}
 \vcenter{\hbox{$\displaystyle \xymatrix@C=34pt@R=20pt{&W_z\ar[dr]^0&\\E_L\ar[ur]^{-q}&&L^2\ar[ll]_\beta}$}}.
\end{equation}
In particular, the second-derivative map $\gamma$ at the cubic seed
has rank two over $L$. Vertex
$2$ is neither a source nor a sink. Its action $a^*$ is zero, $c$ acts
by $[1\ z]$, and $b^*$ is adjoint to $-q$.

\begin{proposition}\label{prop:cubic-character}
For the integral tensor algebra at the initial seed and every $r\ge0$, the
universal family of weight $r\xi$ and its mutation have quotient
polynomials
\begin{align}
 F_{M_r}&=\bigl((1+y_1+y_1y_2)^2+y_1^2y_2^2y_3\bigr)^r,
       \label{eq:cubic-source-F}\\
 F_{N_r}&=\bigl((1+y_1)^2+y_1^2y_3(1+y_2)^2\bigr)^r.
       \label{eq:cubic-target-F}
\end{align}
For $r>0$ these are nonrigid general presentations, and the target
second-derivative map has rank $2r$. The corresponding characters are
members of the basis with principal coefficients.
\end{proposition}
\begin{proof}
For $r=1$ and $e_3=0$, source quotient kernels are flags in $L^2$;
their generating polynomial is $(1+y_1+y_1y_2)^2$. For $e_3=1$,
the four conjugates of $[1\ z]$ have zero common kernel, forcing
$e_1=e_2=2$. This gives \eqref{eq:cubic-source-F}.

At the target, if $e_3=0$, the images of the four maps from the
sheets of $E_L$ span $W_z$. The kernel at vertex $2$ is therefore
$W_z$, so $e_2=0$, while the quotient at vertex $1$ is arbitrary. If $e_3=1$, the four maps $c_j$
force the kernel at $1$ to be zero; the quotient of $W_z$ is arbitrary.
These two cases give \eqref{eq:cubic-target-F} for $r=1$ directly.

For general $r$, the integral presentation space
$\Hom(P_3^r,P_1^{2r})$ and its off-diagonal direct-sum matrices are the
same as those of the one-arrow calculation. The determinant
\eqref{app:eq:pairdet} therefore gives a direct sum of $r$ independently general primitive
presentations after a coefficient-field extension of ramification index one. Mutation commutes with direct sums,
so both polynomials are these powers. The value of $\op{e}$ on this direct sum is $r$. Its local maps are
the direct sums of $r$ copies of the maps in
\eqref{eq:cubic-triangle}. Generality at the target follows from
Corollary~\ref{thm:fixed-field-generality}.

Since $\xi_2=0$, the mutation prefactor is one, and the identity is
\[ F_{N_r}(\mathbf y') =F_{M_r}\left(\frac{y_1'y_2'}{1+y_2'},(y_2')^{-1},y_3'(1+y_2')^4\right). \]
Together with the good acyclic initial seed and principal coefficients, this
identifies the basis members by the transported-character formula.
\end{proof}

For $r=2$, the target coefficient at $(e_1,e_2,e_3)=(4,0,1)$ is two.
Its representation Grassmannian is $\Res_{E_L/L}\mathbb P^1$: the
kernel at vertex $1$ is zero and the kernel at vertex $2$ is the whole
representation space.
The kernel at vertex $3$ is any $E_L$-line in $E_L^2$. This representation Grassmannian has geometric Euler characteristic $16$ and tame trace $2$.

To see nontrivial permutation of cohomology at the same cubic seed,
use the valued family of Example~\ref{run:ex:one-arrow} on vertices
$2,3$, give vertex $1$ the same space as vertex $2$, and let
$1\to2$ act by the identity. Mutate this representation at $2$. The quotient of target ranks $(2,2,1)$ is again
the two-eigenline scheme calculated there.
An outside kernel consists of a plane $U$ at vertex $1$ and one line
at each sheet of vertex $3$, containing the corresponding image of
$[I\ Z_m]U$. These are precisely the two-point Grassmannian conditions in that example.
The kernel at vertex $3$, of dimension four over $L$,
meets the image of $L^4\to E_L^2$ in dimension two. Its image in the
mutated middle space therefore has dimension two and is the unique
possible middle kernel. All the maps in this construction are defined over $L$, so the
resulting scheme isomorphism respects inertia. Consequently
\[
[y_1^2y_2^2y_3]F_{\mu_2M_m}
 =\begin{cases}2,&m\text{ even},\\0,&m\text{ odd}.\end{cases}
\]
The parity calculation uses the specified valued family, whereas the
power formulas use the universal presentations of that integral tensor algebra.
Both concern the same cubic potential, which is rigid because it is
mutation-equivalent to an acyclic zero potential.

\paragraph{\em An endomorphism order on a curve.}
\label{para:elliptic-order}
We now realize the initial order geometrically, including the
composition of its two arrow spaces. Let $C$ be a smooth projective
curve of genus one over $\C$, choose distinct points $p,q\in C$, and
put $X=C\times\Spec\OO$.
Choose a ramified divisor $D\subset X$ supported at $q$ on the special
fibre, disjoint from the constant section $p$, with
$\Gamma(D,\OO_D)=\OO_E=\C[[s]]$, $s^4=t$.
Such a divisor is obtained from an \emph{\'etale local coordinate} at
$q$, namely a regular parameter inducing an \'etale map to the affine
line, by taking the branch of $u^4=t$ through $q$. The completed local
ring is $\C[[u,t]]$, and this branch has ring $\C[[s]]$.
The resulting map $\Spec\OO_E\to X$ is a closed immersion near $q$:
on an affine neighbourhood, its ring map is surjective because the
image contains $\OO$ and $s$. Its image is closed in $X$, since
$\Spec\OO_E$ is proper over $\OO$ and $X$ is separated. The image
is Cartier, since $X$ is regular and the image has codimension one.

Let $\OO_X(-p)$ be the ideal line bundle of the section and set
$\mathcal T=\OO_X(-p)\oplus\OO_X\oplus j_*\OO_D$.
Riemann--Roch on $C$ gives $H^0(C,\OO_C(p))=\C$ and
$H^0(C,\OO_C(-p))=0$ \cite[Tag 0B5B]{Stacks}.
The canonical section vanishing at $p$ spans the first space and
trivializes $\OO_X(p)$ along $D$. Consequently
\[
\End_X(\mathcal T)^{\mathrm{op}}\simeq
 \begin{pmatrix}
 \OO&\OO&\OO_E\\
 0&\OO&\OO_E\\
 0&0&\OO_E
 \end{pmatrix}.
\]
Composition induces an isomorphism
\[ \Hom_X(\mathcal O_X(-p),\mathcal O_X)\otimes_{\OO} \Hom_X(\mathcal O_X,j_*\mathcal O_D) \simeq\Hom_X(\mathcal O_X(-p),j_*\mathcal O_D). \] All reverse Hom spaces vanish, by negative degree or by
the Cartier-divisor argument in Theorem~\ref{thm:divisor-order}.
Thus this is precisely the initial integral tensor algebra used in the cubic
calculation. General morphisms
$\OO_X(-p)^{2r}\to j_*\OO_D^{r}$ determine its presentations at
$r(2,0,-1)$. Their characters, and their mutations to the cubic
potential, are those computed above.

\subsection{A nonrigid surface potential}
Let $E=K(t^{1/2})$, take $R=E^3$, and put two diagonal $E$-arrows
on each side of the oriented triangle
\[
\vcenter{\hbox{$\displaystyle \xymatrix@C=48pt@R=34pt{
 &2\ar@<.7ex>[dr]^{b_1}\ar@<-.7ex>[dr]_{b_2}&\\
 1\ar@<.7ex>[ur]^{a_1}\ar@<-.7ex>[ur]_{a_2}&&
 3\ar@<.7ex>[ll]^{c_1}\ar@<-.7ex>[ll]_{c_2}
}$}}.
\]
Consider
\begin{equation}\label{eq:markov-potential}
 S=a_1b_1c_1+a_2b_2c_2-a_1b_2c_1a_2b_1c_2.
\end{equation}
This is the standard once-punctured-torus potential. Its completed
Jacobian algebra has dimension $36$ over $E$, and $S$ is not rigid
\cite[Propositions 2.4, 4.2, 4.4 and 4.8]{LadkaniClosed}.
The multiplicity vector at each edge is supported at the diagonal type. Tensor
composition preserves that support, so the modulation is good.

Nondegeneracy can be seen directly in these coordinates. On mutating
at $2$, put $x_{ij}=[a_i b_j]$ and $r_{ij}=b_j^*a_i^*$.
The premutated potential is
\[ x_{11}c_1+x_{22}c_2-x_{12}c_1x_{21}c_2 +\sum_{i,j}x_{ij}r_{ij}. \]
Put $C_i=c_i+r_{ii}$ and make the further unitriangular changes
\[ \widehat x_{11}=x_{11}-x_{21}C_2x_{12}+x_{21}r_{22}x_{12}, \qquad \widehat x_{22}=x_{22}+x_{12}r_{11}x_{21}. \]
Cyclically the premutation becomes
$\widehat x_{11}C_1+\widehat x_{22}C_2$ plus the reduced potential
\[ x_{12}r_{12}+x_{21}r_{21}-x_{12}r_{11}x_{21}r_{22}. \]
It has the form \eqref{eq:markov-potential} after renaming the arrows.
Cyclic symmetry covers the other two vertices, so iteration proves
nondegeneracy. We next compute a character for this nonrigid potential. The
finite-dimensionality statement above permits use of the completed
Betti formula.

Let $L=\Omega((t))$, where $\Omega$ is an algebraically closed
extension of $\C$, and put $E_L=\Omega((t^{1/2}))$. Choose
$\lambda=\lambda_0+\lambda_1t^{1/2}\in E_L^\times$, with
$\lambda_0,\lambda_1$ algebraically independent over $K$, and define
\[ M_1=E_L,\quad M_2=0,\quad M_3=E_L,\qquad c_1=1,\quad c_2=\lambda, \]
with all $a_i,b_j$ zero and no decoration. Every cyclic derivative of
$S$ acts by zero. At vertex $2$,
\[
M_{\rm in}=E_L^2,
 \quad M_{\rm out}=E_L^2,
 \quad \alpha=\beta=0,
 \quad \gamma=\begin{pmatrix}1&0\\0&\lambda\end{pmatrix}.
\]
The sextic term makes no contribution on this representation, whereas
the two cubic terms make $\gamma$ invertible. The completed Betti
formula gives $\delta=(-1,0,1)$. Since $\End(M)=E_L$ and
$\langle\delta,(1,0,1)\rangle_D=0$, one has $\op{e}(M)=2$.
This family represents general presentations of that weight, up to
changes of presentation bases:
the two-dimensional $K$-parameter $\lambda$ distinguishes the
isomorphism classes, while their presentation orbits have codimension
two by the formula for $\op{e}$. The rational minimal-cover
charts of Lemma~\ref{lem:cover-charts} lift this family to
presentations. Its orbit saturation has dimension
$2+(\dim\PHom(-1,0,1)-2)$: distinct parameters give distinct orbits,
and changes of presentation bases supply each orbit. It is therefore dense in the full
presentation space.

The mutated middle space is $E_L^2$, with
$\alpha'=-\diag(1,\lambda)$ and $\beta'=I_2$.
The surviving $x_{12},x_{21}$ act by zero. The maps $b_j^*$ are nonzero
scalar multiples of the two coordinate inclusions $E_L\to E_L^2$,
and the maps $a_i^*$ are the coordinate projections. Direct quotient enumeration gives
\begin{equation}\label{eq:markov-F}
 F_M=1+y_3+y_1y_3,
 \qquad
 F_{M'}=1+y_3'(1+y_2')^2+y_1'y_3'(y_2')^2.
\end{equation}
For the second formula, quotient ranks $(e_1,e_3)=(0,0)$ force
$e_2=0$; ranks $(0,1)$ allow every quotient of $E_L^2$; ranks $(1,1)$
force $e_2=2$; and ranks $(1,0)$ are impossible. In the middle case
Lemma~\ref{lem:relative-grassmannian} gives coefficients $1,2,1$.
In particular the representation Grassmannian corresponding to quotient ranks $(0,1,1)$ is
$\Res_{E_L/L}\mathbb P^1$, with geometric Euler characteristic four
and tame trace two.

The exchange matrix is
$B=\begin{pmatrix}0&2&-2\\-2&0&2\\2&-2&0\end{pmatrix}$.
Both local $h$-coordinates at $2$ are zero. Substituting
\[ Y_1=\frac{y_1'(y_2')^2}{(1+y_2')^2},\qquad Y_3=y_3'(1+y_2')^2 \]
into $F_M$ gives the second polynomial in \eqref{eq:markov-F}.
This verifies the trace mutation law with both a nonrigid potential
and nonrigid general representations. It is a character application:
the reachability hypothesis of the basis theorem fails, every mutation negates this matrix, and the
sum of the mutable
presentation-weight coordinates is preserved: $\sum_j[b_{ij}]_+=\sum_j[-b_{ij}]_+=2$ for each $i$. A unit weight therefore cannot mutate to a
negative unit weight. Adjoining principal coefficients does not change this obstruction.

\Needspace{8\baselineskip}
\section{Order dependence and valued transport}
\label{sec:order-sensitivity}\label{sec:arithmetic-transport}
The auxiliary changes summarized at the end of
Section~\ref{sec:weights} leave the initial order unchanged.
We now vary that order, study the resulting characters, and trace the
induced valuation under mutation without choosing a new target order. The
examples lead to an intrinsic mutation operation on stable equivalence
classes of valued presentations; this refines, but is not required by,
the fixed-field construction used in the basis theorem.

\subsection{An order-dependent family and its mutation}
\label{arith:subsec:counterexample}
The valued family in Example~\ref{run:ex:one-arrow} showed that
the monodromy trace can vary while the reductions modulo $t$ of the
specified integral representations remain isomorphic.
We now obtain such variation from actual full orders in one path
algebra. After mutation, one of the resulting characters is obtained
from no full order in the target algebra.

Let
\[ Q:\quad\vcenter{\hbox{$\displaystyle \xymatrix@C=42pt{1\ar@<.6ex>[r]^{a_0}\ar@<-.6ex>[r]_{a_1}&2\ar@<.6ex>[r]^{b_0}\ar@<-.6ex>[r]_{b_1}&3}$}}, \]
and fix $P=KQ=\mathcal P(Q,0)$, $R=K^3$. Thus the modulation is good, the
potential is fixed and rigid, and $\dim_KP=11$. Put
\[ u_i=a_ib_0,\quad v_i=a_ib_1,\quad h_0=v_0-u_1,\quad h_1=v_1. \]
For $m\ge0$ define
\[ \mathscr P_m=\OO^3\oplus\bigoplus_{i=0}^1t^m\OO a_i \oplus\bigoplus_{i=0}^1\OO b_i \oplus\OO h_0\oplus\OO h_1 \oplus t^m\OO u_0\oplus t^m\OO u_1. \]
Writing $A_i=t^ma_i$, $U_i=t^mu_i$, the nonzero products of positive
basis elements are
\[ A_0b_0=U_0,\quad A_1b_0=U_1,\quad A_0b_1=U_1+t^mh_0,\quad A_1b_1=t^mh_1. \]
All products of three arrows vanish. Thus these are nested augmented
finite full orders, with $\mathscr P_0=\OO Q$.

For $\delta=(1,0,-2)$ the universal presentation is
$P_3^2\to P_1$. Its two components belong to
$L_m=\OO h_0+\OO h_1+t^m\OO u_0+t^m\OO u_1$.
With two independent generic $2$-by-$2$ coefficient matrices $Z,W$,
the map is $(h_0,h_1)Z+t^m(u_0,u_1)W$.
On the open $\det Z\ne0$, a source basis change gives
\[
(h_0,h_1)+t^m(u_0,u_1)C,\qquad
 C=WZ^{-1}=\begin{pmatrix}\alpha&\gamma\\\beta&\varepsilon\end{pmatrix}.
\]
The coordinates of $C$ are algebraically independent. Its decorated
cokernel has no negative decoration, dimension $(1,2,2)$, the two
$a_i$ images a basis, and
\[
b_0=I_2,\qquad
 b_1=T_m=\begin{pmatrix}
 -t^m\alpha&-t^m\gamma\\1-t^m\beta&-t^m\varepsilon
 \end{pmatrix}.
\]
Here $J_2=\left(\begin{smallmatrix}0&0\\1&0\end{smallmatrix}\right)$.
The map $C\mapsto J_2-t^mC$ is an invertible affine map over $K$,
so these modules are represented by general presentations of weight
$\delta$.

\begin{proposition}\label{arith:prop:order-dependence}
For the fixed QP and weight above, we have
\[
[y_1y_2y_3]F_{\delta;\mathscr P_m}
 =\begin{cases}2,&m\text{ even},\\0,&m\text{ odd}.\end{cases}
\]
In particular arbitrary full orders in a fixed Jacobian algebra do not give the
same monodromy $F$-polynomial.
\end{proposition}
\begin{proof}
A rank-$(1,1,1)$ quotient has a zero kernel at vertex $1$ and line
kernels at vertices $2,3$. The map $b_0=I_2$ makes these lines equal;
$b_1$ requires an eigenline of $T_m$. The discriminant is
\[ \Delta_m=t^m\left[-4\gamma+ t^m\bigl((\alpha-\varepsilon)^2+4\beta\gamma\bigr)\right]. \]
For $m\ge1$ its valuation is $m$. Over $L=\Omega((t))$, with
$\Omega$ algebraically closed and containing the universal parameters,
the parity argument of Example~\ref{run:prop:one-arrow-trace} gives
trace two for even $m$ and zero for odd $m$. For $m=0$, the generic
matrix is constant over $\Omega$ and has two distinct eigenlines,
both fixed by inertia. In every case the Grassmannian is finite
\'etale of degree two, so only $H^0$ contributes.
\end{proof}

The whole polynomial in this example is
\[
F_{\delta;\mathscr P_m}
 =1+y_1(1+y_2+y_2y_3)^2
   -\begin{cases}0,&m\text{ even},\\2y_1y_2y_3,&m\text{ odd}.
     \end{cases}
\]
This follows by enumerating $0\le e_3\le e_2\le2$ for the nonzero
quotients, which all have $e_1=1$: apart from the eigenline scheme,
the relevant schemes are points or $\mathbb P^1$.
Both polynomials have the tetrahedron
$\op{conv}\{(0,0,0),(1,0,0),(1,2,0),(1,2,2)\}$ as Newton
polytope, as predicted by Theorem~\ref{newton:thm:order-independence}.
The changing exponent $(1,1,1)$ is the midpoint of the edge joining
$(1,0,0)$ and $(1,2,2)$. Order dependence can therefore occur on the
boundary, not only in the interior.
For any geometric coefficient extension with $\rank\widetilde B=n$, these
different polynomials give different cluster characters.

\paragraph{\em After mutation.}
We now keep the valued family and mutate it at the sink $3$.

\begin{proposition}\label{arith:thm:no-target-order}
Mutate the preceding QP at the sink $3$. For odd $m$ there is no full
order $\mathscr Q$ in the target Jacobian algebra, containing its vertex order,
whose monodromy $F$-polynomial at weight $\delta'=(1,-4,2)$ equals
that of the transported family. Their cluster characters differ after
any geometric coefficient substitution.
\end{proposition}
\begin{proof}
The target quiver is
\[
Q':\quad\vcenter{\hbox{$\displaystyle \xymatrix@C=46pt{
 1\ar@<.7ex>[r]^{a_0}\ar@<-.7ex>[r]_{a_1}&2&
 3\ar@<.7ex>[l]^{b_0^*}\ar@<-.7ex>[l]_{b_1^*}
}$}},
\]
with zero potential and radical-square-zero Jacobian algebra. Sink reflection replaces
$M_3$ by $\ker(I_2\ T_m)$, identified with $L^2$ by
$w\mapsto(-T_mw,w)$. The reversed maps are $-T_m,I_2$.
The two line kernels for a rank-$(1,1,1)$ quotient must again be the
same eigenline. Its trace is zero for odd $m$. The weight
rule gives $\delta'=(1,-4,2)$.

Every full order in the radical-square-zero target is the sum of the
vertex order and arbitrary full lattices in its two arrow spaces. All products of two arrows vanish. Independent $K$-linear changes
of bases in the two arrow spaces carry the standard integral algebra $\OO Q'$ to any
such order and extend to an $R$-algebra automorphism. This identifies their universal
presentation lattices and representation Grassmannians over the same valued field.
Their universal presentations therefore have the monodromy
$F$-polynomial of the constant order. The generic
representation at $\delta'$ has dimension $(1,2,2)$; on the chart where
one reversed map is invertible, their ratio is a generic constant
$2$-by-$2$ matrix. Its two eigenlines are fixed by inertia, giving
coefficient two. More precisely, the polynomial for every target
order is
\[ (1+y_3')^2+y_1'(1+y_3'+y_2'y_3')^2, \]
whereas the transported polynomial subtracts $2y_1'y_2'y_3'$.
After any geometric coefficient substitution, their characters
therefore differ by a nonzero Laurent monomial with coefficient two.
They cannot be equal.
\end{proof}

For odd $m$, the source order $\mathscr P_m$ cannot be extended to a
choice of full orders at every seed giving the same characters at every weight.
The existence of another such global choice remains a separate question.

\begin{remark}\label{arith:rem:finite-jets}
Using the bases above to identify their underlying free modules,
$\PP_1/t\PP_1$ and $\PP_2/t\PP_2$ have identical multiplication
tables. Their universal presentations agree under these identifications. The preceding coefficient nevertheless differs.
Thus the reduction of the order and presentation modulo $t$ does not
determine the monodromy $F$-polynomial.
\end{remark}

The transported valuation has a geometric model on the target
presentation space. In the coordinates given by the ratio of the two
reversed arrow maps, it is defined by
$T=J_2-t^mC$, where the entries of $C$ have algebraically independent
residues. Adjoining the coordinates $(J_2-T)/t^m$ gives an affine modification
of the target parameter space over $\OO$. For $m=1$, it is the chart
of the blowup of $(t,(T-J_2)_{ab})$ on which these ratios are regular.
For larger $m$, repeat this construction $m$ times. The resulting model
carries the induced valuation on the function field of the target
presentation space. We next construct this valuation for an arbitrary
algebraically generic valued presentation.

\subsection{Valued presentations and the target valuation}
\label{arith:subsec:valued-presentations}
The arithmetic data include a valuation on the function field
$K(X_\delta)$ of the source presentation space. By
Theorem~\ref{thm:generic-species-generality}, a dominant map obtained
by adjoining changes of target presentation bases embeds the target
function field in a valued extension of the source field. Restricting
that valuation does not require a new target order.

Let $F/K$ be a finitely generated field extension. We consider discrete
valuations $v$ satisfying
\begin{equation}\label{arith:eq:geometric}
 v|_K=v_t,\qquad v(F^\times)=\Z,\qquad v(t)=1,\qquad
 \op{trdeg}_{\C}\kappa(v)=\op{trdeg}_KF.
\end{equation}
The notation $\op{trdeg}_F E$ means transcendence degree.
An $\OO$-model of $F/K$ is an integral finite-type $\OO$-scheme
with function field $F$. A \emph{divisorial valuation} is the valuation
of the discrete valuation ring at a prime divisor, that is, a
codimension-one integral subvariety, on a normal model.
A divisor is \emph{vertical} if it lies over $t=0$; its multiplicity
is the valuation of $t$. Lemma~\ref{arith:lem:restriction} identifies
the valuations in \eqref{arith:eq:geometric} with such divisors of
multiplicity one. In particular, $v(t)=1$ gives ramification index one
over $K$. In fact $X^{d_i}-t$ is Eisenstein over
$\mathcal O_v$, so $F\otimes_KK_i$ is still a field of degree $d_i$.
The initial Gauss field satisfies \eqref{arith:eq:geometric}.

Put $X_\delta=\PHom(\delta)$. A \emph{valued presentation of weight $\delta$}
is a triple $(F,v,f)$ with $(F,v)$ satisfying
\eqref{arith:eq:geometric} and a $K$-embedding $K(X_\delta)\hookrightarrow F$;
$f$ is the universal presentation pulled back by this embedding. Its coefficients are algebraically generic but
need not have independent residues in the chosen coordinates. We associate to it the decorated cokernel $\mathcal M_f$.

Two valued presentations are \emph{stably equivalent} if they are connected by a
finite chain of the following comparisons: valuation-preserving $K$-embeddings
of their fields into one field satisfying \eqref{arith:eq:geometric},
with isomorphic decorated cokernels after scalar extension. In particular, one may adjoin independent Gauss parameters for
changes of presentation bases. Write $\mathfrak A_\delta(P)$ for the set of these equivalence classes.
When the presentation space is a point, its function field is $K$.

\begin{lemma}\label{arith:lem:trace}
The polynomial of a valued presentation is invariant under stable equivalence, and
hence defines $F_{\mathfrak a}$ for $\mathfrak a\in\mathfrak A_\delta(P)$.
\end{lemma}
\begin{proof}
Choose compatible strict henselizations for the valuation-preserving
embeddings in a generating comparison, and complete them. The induced
embeddings are isometric and keep $t$ as uniformizer. Their residue fields
are separably closed and, in characteristic zero, algebraically closed.
Compatible roots of $t$ identify the tame generators. The general
completed-strictly-henselian case of
Proposition~\ref{prop:coefficient-extension} therefore identifies the
geometric cohomology equivariantly. No identification of chosen
coefficient subfields is needed. The isomorphism of
the decorated cokernels identifies the representation Grassmannians, so every
coefficient trace agrees. Equality then holds along the transitive
closure defining stable equivalence. 
\end{proof}

For fixed $P$ and $\delta$, these algebraically generic valued
presentations have the same Newton polytope by
Theorem~\ref{newton:thm:order-independence}. Their $F$-polynomial coefficients
can nevertheless differ, as in Proposition~\ref{arith:prop:order-dependence}.

\paragraph{\em Restriction of the valuation.}
\label{arith:subsec:divisorial}
We first verify that passing to a finitely generated subfield preserves
\eqref{arith:eq:geometric}.

\begin{lemma}\label{arith:lem:restriction}
If $(F,v)$ satisfies \eqref{arith:eq:geometric} and $K\subset E\subset F$
is finitely generated over $K$, then $(E,w=v|_E)$ satisfies the same
conditions. Its valuation ring is the local ring at a multiplicity-one
vertical divisor on a normal finite-type $\OO$-model of $E$.
\end{lemma}
\begin{proof}
Set $n=\op{trdeg}_KF$, $m=\op{trdeg}_KE$, and
$r=\op{trdeg}_{\C}\kappa(w)$. The value group is a subgroup
of $\Z$ containing $w(t)=1$, hence equals $\Z$. Lifting independent
residues and reducing a polynomial relation after normalization by a
power of $t$ gives $r\le m$. Applying the same independence argument to residues that are
algebraically independent over $\kappa(w)$ gives
$\op{trdeg}_{\kappa(w)}\kappa(v)\le
\op{trdeg}_EF$. Thus $n-r\le n-m$, and $r=m$.

Choose valuation-zero lifts $u_1,\ldots,u_m$ of a residue transcendence
basis. Their field is purely transcendental over $K$, its valuation
ring is $\OO[\mathbf u]_{(t)}$, and $E/K(\mathbf u)$ is finite.
Normalize $\OO[\mathbf u]$ in $E$. This normalization is finite:
$\OO[\mathbf u]$ is a finite-type algebra over a complete Noetherian
local ring, so its integral closure in a finite extension of its fraction
field is finite \cite[Tag 0335]{Stacks}. The selected extension valuation is
the DVR at a height-one prime above $(t)$. Its vertical multiplicity
is one because $w(t)=1$. The residue extension is finite by the
valuation degree inequality \cite[Tag 0ASH]{Stacks}.
\end{proof}

\paragraph{\em The target presentation.}\label{arith:subsec:framing}
Let $P'=\mathcal P(\mu_k(A,S))$, $\delta'=\mu_k\delta$, and
$Y=\PHom_{P'}(\delta')$. On a dense source chart let $q_k(x)$ be a
rational presentation with disjoint supports of the mutated object, supplied
by Lemma~\ref{lem:cover-charts} and
Theorem~\ref{thm:generic-species-generality}. Put
$G'=\Aut_{P'}(P'([\delta']_+))\times \Aut_{P'}(P'([-\delta']_+))$.
Each automorphism group is open in its finite-dimensional endomorphism
space, so the function field of $G'$ is purely transcendental over $K$.

\begin{lemma}\label{arith:lem:framing}
The map
\begin{equation}\label{arith:eq:framedmap}
 \Phi_k:X_\delta\times G'\dashrightarrow Y,\qquad
 (x,g_+,g_-)\longmapsto g_+q_k(x)g_-^{-1}
\end{equation}
is dominant. For a valued presentation $(F,v,f)$, adjoin independent Gauss
coordinates $\mathbf u$ for $G'$, giving $(F(\mathbf u),v_G)$. Then
$\Phi_k^*:K(Y)\hookrightarrow F(\mathbf u)$ is injective and the
universal target cokernel pulled back there is the mutation of
$\mathcal M_f$.
\end{lemma}
\begin{proof}
Theorem~\ref{thm:generic-species-generality} makes the image of
\eqref{arith:eq:framedmap} dense. Algebraic independence of the source
coefficients and of the parameters in $G'$ therefore makes pullback
injective. The two automorphisms change the presentation bases while
preserving its decorated cokernel. Adjoining Gauss parameters preserves
\eqref{arith:eq:geometric}, and the finite-dimensionality of the
Jacobian algebra reduces evaluation to finite matrices.
\end{proof}

The comparison of valued fields is
\[ (K(Y),w)\xrightarrow{\ \Phi_k^*\ }(F(\mathbf u),v_G) \longleftarrow(F,v),\qquad w=v_G|_{K(Y)}. \]
Both embeddings preserve the uniformizer $t$ and therefore the chosen
tame generator under the comparisons of Proposition~\ref{prop:coefficient-extension}.

\subsection{Mutation on equivalence classes}
\label{arith:subsec:theorem}
\begin{theorem}\label{arith:thm:transport}
For the chosen good Jacobi-finite nondegenerate ramified realization, mutation
induces a well-defined bijection
$\mu_k^{\mathrm{ar}}:\mathfrak A_\delta(P) \xrightarrow{\sim}\mathfrak A_{\mu_k\delta}(P')$. 
A target representative is defined over the function field $K(Y)$
of $Y=\PHom_{P'}(\mu_k\delta)$, with valuation
$w(H)=v_G(\Phi_k^*H)$ for $H\in K(Y)^\times$. 
It is divisorial and satisfies $w(t)=1$. After pullback to $F(\mathbf u)$, its decorated cokernel is the
mutated object and has its monodromy $F$-polynomial. The class
is independent of kernel bases, sections, projective-cover charts,
changes of presentation bases, and
QP reductions inducing the prescribed isomorphisms of Jacobian algebras. Mutation is involutive under the
prescribed double-mutation comparison and composes along every finite
mutation sequence.
\end{theorem}
\begin{proof}
Lemmas~\ref{arith:lem:framing} and~\ref{arith:lem:restriction} construct
the target valued presentation and its divisorial model. Its pulled-back presentation
is $g_+q_k(f)g_-^{-1}$, so Lemma~\ref{arith:lem:trace} identifies its
monodromy $F$-polynomial with that of the mutated representation.

For two choices, adjoin independent parameters for both pairs of
changes of presentation bases to the source valued field. The
independence statement in Theorem~\ref{thm:representation-mutation}
identifies the target decorated cokernels there, and both target fields
embed valuation-preservingly. Thus the two valued presentations are
stably equivalent.
For a generating comparison of valued presentations, first use its common field
and then adjoin the parameters for the changes of bases. Tensor products, kernels, images and their
finite-dimensional quotients commute with field extension, so the same
argument proves well-definedness on the generated equivalence relation.

Applying mutation twice identifies the resulting decorated representation
with the source over a common valued field, through the
prescribed right-equivalence of the twice-mutated QP. The $\delta$-vector returns
to $\delta$ and every comparison keeps $t$ as uniformizer. This proves
the inverse assertion; iteration proves composition along finite mutation sequences and
covariance under the prescribed reductions.
\end{proof}

We may now repeat the construction using the target presentation and
its induced valuation.

\section{Further questions}\label{sec:outlook}
\paragraph{\em Goodness without a reddening sequence.}
A supplied reddening sequence gives a finite test by
Theorem~\ref{gm:thm:reddening}, while
Proposition~\ref{gm:thm:finitealgorithm} treats finite matrix mutation
classes. When the mutation class is infinite and no reddening sequence
is available, can goodness be decided by a finite computation? Does
triangular extension preserve goodness without the reddening hypotheses
of Theorem~\ref{gm:thm:triangular-extension}?

\paragraph{\em Geometric orders for other Jacobian algebras.}
Theorem~\ref{thm:divisor-order} determines a basis from ramified divisors
on a curve, and the genus-one construction gives the initial order for
the cubic example. Can comparable sheaf constructions realize
modulations with nonconstant type multiplicities or Jacobian algebras of
nonrigid potentials? Which geometric property characterizes their bases?

\paragraph{\em Multiplication.}
For the degree-four divisor, independent elementary transformations
explain multiplication along the exceptional ray. A formula for arbitrary
products should involve extensions and the monodromy of their parameter
spaces. Can a correspondence of extensions compute the structure
constants of a basis associated with an endomorphism order?

\part*{Appendices}
\appendix
\section[Decorated Auslander--Reiten translation (Part I)]{Decorated Auslander--Reiten translation}
\label{app:AR}
This appendix supplements the representation theory of
Part~\ref{part:representations} by treating decorated
Auslander--Reiten translation. We first describe its effect on
Hom spaces and their dimensions, then deduce the corresponding
statements for the presentation invariant.

Let $\nu=\mathbb D\Hom_{\mathcal P}(-,\mathcal P)$ be the
\emph{Nakayama equivalence} from finite projectives to finite injectives.
For a module $N$ with no projective summand, its
\emph{Auslander--Reiten translate} is
$\tau N=\ker(\nu Q_-\to\nu Q_+)$, where
$Q_-\to Q_+\to N$ is a minimal projective presentation.
The inverse translate is defined dually using an injective copresentation.
Write $\mathcal M=(M,V)$, $v_i=\dim_{K_i}V_i$, and
$M=M_{\mathrm{np}}\oplus P(a)$, where $M_{\mathrm{np}}$ has no
projective summand. Define
\[ \tau_*\mathcal M =\bigl(\tau M_{\mathrm{np}}\oplus I(v),\ \bigoplus_iK_i^{a_i}\bigr). \]
Its inverse uses $\tau^{-1}$ on the part without injective summands,
sends $I_i$ to $E_i^-$, and sends $E_i^-$ to $P_i$.
Thus $P_i\mapsto E_i^-\mapsto I_i$. These formulas define mutually inverse transformations of decorated
isomorphism classes. Applying $\nu$ to the decorated projective presentation
and reading the resulting map as an injective copresentation gives
\begin{equation}\label{eq:tau-weights}
 \check\delta_{\tau_*\mathcal M}=-\delta_{\mathcal M},\qquad
 \delta_{\tau_*^{-1}\mathcal M}=-\check\delta_{\mathcal M}.
\end{equation}
Also $\mathbb D\tau_*=\tau_*^{-1}\mathbb D$, with the translate
on the opposite algebra on the right.

\begin{proposition}\label{thm:AR-invariants}
Write $M^+=(\tau_*\mathcal M)_+$ and
$N^+=(\tau_*\mathcal N)_+$. The Hom dimensions satisfy
\begin{equation}\label{eq:AR-Hom-plus}
 \hom(M^+,N^+)=\hom(M,N)
  -\langle\delta_{\mathcal M},\mathbf r_N\rangle_D
  -\langle\mathbf r_{M^+},\delta_{\mathcal N}\rangle_D.
\end{equation}
There is also a natural isomorphism
\begin{equation}\label{eq:AR-Hom-E}
 \Hom_{\mathcal P}(N,M^+)
     \simeq\mathbb D\op{E}_{\mathcal P}(\mathcal M,\mathcal N).
\end{equation}
For the presentation invariant, there is an isomorphism
\begin{equation}\label{eq:AR-E-switch}
 \op{E}_{\mathcal P}(\mathcal M,\mathcal N)
     \simeq\op{E}_{\mathcal P^{\mathrm{op}}}
       (\mathbb D\tau_*\mathcal N,\mathbb D\tau_*\mathcal M).
\end{equation}
Consequently, $\op{e}(\mathcal M)$ and
$\op{e}(\mathcal M,\mathcal N)+\op{e}(\mathcal N,\mathcal M)$
are preserved by simultaneous translation under $\tau_*$ or
$\tau_*^{-1}$. The formulas for inverse translation follow by
opposite duality.
\end{proposition}
\begin{proof}
We first compute $\Hom(N,M^+)$. For a finite projective $Q$, the
natural duality
$\mathbb D\Hom(Q,N)\simeq\Hom(N,\nu Q)$ identifies the dual of
\eqref{eq:E-pres} with $\Hom(N,\ker\nu f_{\mathcal M})$.
That kernel is $M^+$, including the injective contribution of the
decoration. This proves \eqref{eq:AR-Hom-E}; it is the transpose
argument of \cite[Corollary 10.9]{DWZ2}.

To obtain \eqref{eq:AR-Hom-plus}, apply the fully faithful equivalence
$\nu$ to the quotient \eqref{eq:E-two-presentations}. By injectivity,
the resulting quotient is the cokernel computed from the injective
copresentations of $\tau_*\mathcal M$ and $\tau_*\mathcal N$.
Proposition~\ref{prop:opposite-duality} identifies it with the right
side of \eqref{eq:AR-E-switch}. Taking dimensions in
\eqref{eq:E-dual-dimension} and using \eqref{eq:tau-weights} gives
\[ \op{e}(\mathcal M,\mathcal N)=\hom(M^+,N^+)+\langle\mathbf r_{M^+},\delta_{\mathcal N}\rangle_D. \]
On the other hand, \eqref{eq:E-dimension} expresses the same number
as $\hom(M,N)-\langle\delta_{\mathcal M},\mathbf r_N\rangle_D$.
Equating these expressions proves \eqref{eq:AR-Hom-plus}.
The self and symmetric presentation invariants are unchanged by
opposite duality, so \eqref{eq:AR-E-switch} proves their translation
invariance. Apply the same argument over the opposite algebra to
obtain the inverse statements.
\end{proof}

The decoration keeps track of the projective summands that ordinary
Auslander--Reiten translation sends to zero.

We expect the Auslander--Reiten transformation formula for
$F$-polynomials in \cite[Theorem 2.1]{FeiSyzygy} to extend to these
monodromy characters. The tropical transformation formulas in
\cite[Section 4.2]{FeiSyzygy} should likewise extend under the
corresponding homeomorphism hypothesis.

\section[Hom dimensions and tropical F-polynomials (Part II)]{Hom dimensions and tropical \texorpdfstring{$F$}{F}-polynomials}
\label{app:tropical-consequences}
This appendix supplements Part~\ref{part:characters}. We express
Hom dimensions involving a representation that mutates to pure
decoration in terms of tropical $F$-polynomials. The corresponding
evaluation at an arbitrary presentation weight may require a positive
multiple of that weight.

Recall that $\langle\eta,\mathbf r\rangle_D=\eta D\cdot\mathbf r$.
The duality formula \eqref{eq:tropical-submodule} reads
\[ f^{\mathrm{op}}_{\mathbb DN}(\eta D)=\langle\eta,\mathbf r_N\rangle_D+f_N(-\eta D). \]
Here $f^{\mathrm{op}}_{\mathbb DN}$ is the tropical $F$-polynomial of
$\mathbb DN$ over the opposite Jacobian algebra.

\begin{proposition}\label{prop:reachable-tropical-Hom}\label{cor:Hom-tropical}
Let $\mathcal M=(M,V)$ and $\mathcal N=(N,W)$ be finite decorated
representations of a good Jacobi-finite nondegenerate ramified QP.
Suppose a finite mutation sequence sends $\mathcal M$ to a decorated
representation with zero positive part. Then
\begin{equation}\label{eq:reachable-tropical-Hom}
\begin{aligned}
 \hom(M,N)&=f^{\mathrm{op}}_{\mathbb DN}(\delta_{\mathcal M}D),\qquad
 \op{e}(\mathcal M,\mathcal N)&=f_N(-\delta_{\mathcal M}D).
\end{aligned}
\end{equation}
There is no generality or rigidity assumption on $\mathcal N$.
Opposite duality gives the injective formulas.
\end{proposition}
\begin{proof}
For a projective presentation $f$ of $\mathcal M$ and any submodule
$L\subseteq N$, the presentation exact sequence gives
\[ \langle\delta_{\mathcal M},\mathbf r_L\rangle_D=\hom(M,L)-\dim_K\op{E}(f,L)\leq\hom(M,N). \]
Maximizing over $L$ shows that the difference
\begin{equation}\label{eq:Hom-tropical-defects}
\begin{aligned}
 \mathfrak d(\mathcal M,\mathcal N)
 &=\hom(M,N)-f^{\mathrm{op}}_{\mathbb DN}(\delta_{\mathcal M}D)\\
 &=\op{e}(\mathcal M,\mathcal N)-f_N(-\delta_{\mathcal M}D)
\end{aligned}
\end{equation}
is nonnegative. The second equality follows from
\eqref{eq:E-dimension} and \eqref{eq:tropical-submodule}.

This difference is unchanged by simultaneous mutation whenever the
projective Betti identities for disjoint supports hold for $\mathcal M$.
Indeed, write $\delta=\delta_{\mathcal M}$ and $\delta'=\mu_k\delta$.
Skew-symmetry of $BD$ gives
$T_{k,B}(-\delta'D)=-\delta D$. By
Corollary~\ref{cor:tropical-mutation}, the change in the second tropical
term in \eqref{eq:Hom-tropical-defects} is
\[ d_k\bigl(-h_k(N')[\delta_k]_+ +h_k(N)[-\delta_k]_+\bigr). \]
For disjoint projective supports this is precisely the change in
$\op{e}(\mathcal M,\mathcal N)$ given by \eqref{eq:Ep-mutation}.
Subtracting proves invariance.

The representation $\mathcal M$ is $E$-rigid because its final mutation
has zero positive part and $\op{e}(\mathcal M)$ is invariant under mutation.
Corollary~\ref{thm:sign-coherence} therefore supplies the Betti identities
throughout the sequence. At the final seed, write
$\mathcal M_*=(0,V_*)$ and let $v$ be the rank vector of $V_*$.
Then $\delta_{\mathcal M_*}=-v$, $\hom(0,N_*)=0$, and
$f^{\mathrm{op}}_{\mathbb DN_*}(-vD)=0$: the zero submodule maximizes
this nonpositive functional. Thus \eqref{eq:Hom-tropical-defects}
vanishes at the final seed and hence at the initial seed. This proves
\eqref{eq:reachable-tropical-Hom}. Apply the same argument to the
opposite QP for the injective assertions.
\end{proof}

\begin{remark}\label{prop:saturated-tropical}
For a finite module $N$ and an integral presentation weight $\eta$,
write $\hom(\eta,N)$ for the generic value of
$\dim_K\Hom(\coker f,N)$ as $f$ varies in $\PHom(\eta)$, and put
\[ \op{e}(\eta,N)=\hom(\eta,N)-\langle\eta,\mathbf r_N\rangle_D. \]
Taking a general presentation in the first inequality of the preceding
proof gives
\begin{equation}\label{eq:tropical-Hom-bound}
 f^{\mathrm{op}}_{\mathbb DN}(\eta D)\leq\hom(\eta,N).
\end{equation}
For some integer $q>0$, depending on $\eta$ and $N$,
\begin{equation}\label{eq:saturated-tropical}
\begin{aligned}
 \hom(q\eta,N)=qf^{\mathrm{op}}_{\mathbb DN}(\eta D),\qquad
 \op{e}(q\eta,N)=qf_N(-\eta D).
\end{aligned}
\end{equation}
The same equalities hold for every positive multiple of $q$.
For the evaluation, split the vertex fields and assign the weight
$\eta_i$ to each of the $d_i$ sheets above $i$.
The whole presentation space becomes the ordinary one, generic Hom
dimensions are unchanged, and its submodule support function is
$f^{\mathrm{op}}_{\mathbb DN}(\eta D)$. Theorem~3.6 of
\cite{FeiTropical} gives \eqref{eq:saturated-tropical}; subtracting
$q\langle\eta,\mathbf r_N\rangle_D$ gives the second equality.
If $\op{e}(\eta,\eta)=0$, one can take $q=1$: canonical decomposition
\cite[Theorem 4.4]{DF} makes generic Hom additive on repeated independently
general presentations. Here the two arguments in $\op{e}(\eta,\eta)$
are independently general; this is not
$\op{e}(\mathcal M,\mathcal M)$ for one general representation.

The inequality \eqref{eq:tropical-Hom-bound} and the exact evaluations
at positive multiples also give
\begin{equation}\label{eq:Newton-Hom-presentation}
 \mathcal N_R(N)=
 \{r\in\mathbb R^n:\langle\eta,\mathbf r_N-r\rangle_D\leq\hom(\eta,N)
                    \text{ for every }\eta\in\Z^n\}.
\end{equation}
One inclusion follows from \eqref{eq:tropical-Hom-bound}. For the other,
if $r$ is outside the quotient polytope, an integral functional strictly
separates $\mathbf r_N-r$ from the submodule polytope. Apply
\eqref{eq:saturated-tropical} to a positive multiple of that functional;
its half-space in \eqref{eq:Newton-Hom-presentation} excludes $r$.
\end{remark}

\section*{Acknowledgment}
\addcontentsline{toc}{section}{Acknowledgment}
The paper was conceived in 2022; its main ideas and overall structure took shape around 2024.
ChatGPT has helped me overcome many technical difficulties since July 2026, notably (but not limited to) Sections 4.5, 5.2, 7.3, 12.2, and 12.3.
Another helpful contribution that may not be reflected in the current paper is that it generated a large number of examples and counterexamples for me.


\begin{thebibliography}{Stacks}

\bibitem[CC]{CC}
P.~Caldero and F.~Chapoton,
\emph{Cluster algebras as Hall algebras of quiver representations},
Comment. Math. Helv. \textbf{81} (2006), 595--616.
\href{https://arxiv.org/abs/math/0410187}{arXiv:math/0410187}.

\bibitem[Cao]{Cao}
P.~Cao, \emph{On exchange matrices from string diagrams},
\href{https://arxiv.org/abs/2203.07822v2}{arXiv:2203.07822v2} (2022).

\bibitem[CL]{CaoLi}
P.~Cao and F.~Li, \emph{Uniform column sign-coherence and the existence
of maximal green sequences}, J. Algebraic Combin. \textbf{50} (2019),
403--417. \href{https://arxiv.org/abs/1712.00973}{arXiv:1712.00973}.

\bibitem[CMMM]{CMMM}
M.-W.~Cheung, T.~Magee, T.~Mandel and G.~Muller,
\emph{Valuative independence and cluster theta reciprocity},
\href{https://arxiv.org/abs/2505.09585v2}{arXiv:2505.09585v2} (2026).

\bibitem[Dem]{Demonet}
L.~Demonet, \emph{Categorification of skew-symmetrizable cluster algebras},
Algebr. Represent. Theory \textbf{14} (2011), 1087--1162.
\href{https://arxiv.org/abs/0909.1633v1}{arXiv:0909.1633v1}.

\bibitem[DF]{DF}
H.~Derksen and J.~Fei, \emph{General presentations of algebras},
Adv. Math. \textbf{278} (2015), 210--237.

\bibitem[DWZa]{DWZ1}
H.~Derksen, J.~Weyman and A.~Zelevinsky,
\emph{Quivers with potentials and their representations I: Mutations},
Selecta Math. (N.S.) \textbf{14} (2008), 59--119.
\href{https://arxiv.org/abs/0704.0649v4}{arXiv:0704.0649v4}.

\bibitem[DWZb]{DWZ2}
H.~Derksen, J.~Weyman and A.~Zelevinsky,
\emph{Quivers with potentials and their representations II: Applications
to cluster algebras}, J. Amer. Math. Soc. \textbf{23} (2010), 749--790.
\href{https://arxiv.org/abs/0904.0676v2}{arXiv:0904.0676v2}.

\bibitem[Feia]{FeiCombinatorics}
J.~Fei, \emph{Combinatorics of $F$-polynomials},
Int. Math. Res. Not. IMRN (2023), no.~9, 7578--7615.
\href{https://arxiv.org/abs/1909.10151}{arXiv:1909.10151}.

\bibitem[Feib]{FeiTropical}
J.~Fei, \emph{Tropical $F$-polynomials and general presentations},
J. Lond. Math. Soc. (2) \textbf{107} (2023), 2079--2120.
\href{https://arxiv.org/abs/1911.10513v6}{arXiv:1911.10513v6}.

\bibitem[FeiS]{FeiSyzygy}
J.~Fei, \emph{Syzygy transformations of cluster characters and geometric twists},
\href{https://arxiv.org/abs/2609.10090v1}{arXiv:2609.10090v1} (2026).


\bibitem[GLFS]{GLFS}
C.~Geiss, D.~Labardini-Fragoso and J.~Schr\"oer,
\emph{Generic Caldero--Chapoton functions with coefficients and
applications to surface cluster algebras},
\href{https://arxiv.org/abs/2007.05483v1}{arXiv:2007.05483v1} (2020).

\bibitem[GLSa]{GLS1}
C.~Geiss, B.~Leclerc and J.~Schr\"oer,
\emph{Quivers with relations for symmetrizable Cartan matrices I: Foundations},
Invent. Math. \textbf{209} (2017), 61--158.
\href{https://arxiv.org/abs/1410.1403}{arXiv:1410.1403}.

\bibitem[GLSb]{GLS5}
C.~Geiss, B.~Leclerc and J.~Schr\"oer,
\emph{Quivers with relations for symmetrizable Cartan matrices V:
Caldero--Chapoton formulas},
Proc. Lond. Math. Soc. (3) \textbf{117} (2018), 125--148.
\href{https://arxiv.org/abs/1704.06438v3}{arXiv:1704.06438v3}.

\bibitem[GLSc]{GLS2}
C.~Geiss, B.~Leclerc and J.~Schr\"oer,
\emph{Rigid modules and Schur roots},
Math. Z. \textbf{295} (2020), 1245--1277.
\href{https://arxiv.org/abs/1812.09663}{arXiv:1812.09663}.

\bibitem[GL]{GL}
J.~Geuenich and D.~Labardini-Fragoso, \emph{Species with potential arising from surfaces with orbifold points of order 2, Part I: one choice of weights}, Math. Z. \textbf{286} (2017), 1065--1143. \href{https://arxiv.org/abs/1507.04304}{arXiv:1507.04304}.

\bibitem[GHKK]{GHKK}
M.~Gross, P.~Hacking, S.~Keel and M.~Kontsevich,
\emph{Canonical bases for cluster algebras},
J. Amer. Math. Soc. \textbf{31} (2018), 497--608.
\href{https://arxiv.org/abs/1411.1394v2}{arXiv:1411.1394v2}.

\bibitem[HL]{HuangLi}
M.~Huang and F.~Li,
\emph{Unfolding of sign-skew-symmetric cluster algebras and its applications
to positivity and $F$-polynomials}, Adv. Math. \textbf{340} (2018), 221--283.
\href{https://doi.org/10.1016/j.aim.2018.10.008}{doi:10.1016/j.aim.2018.10.008}.
\href{https://arxiv.org/abs/1609.05981v2}{arXiv:1609.05981v2}.

\bibitem[Kel]{KellerGreen}
B.~Keller, \emph{A survey on maximal green sequences},
with an appendix by L.~Demonet.
\href{https://arxiv.org/abs/1904.09247v7}{arXiv:\allowbreak1904.09247v7} (2020).

\bibitem[LFM]{LabardiniMou}
D.~Labardini-Fragoso and L.~Mou,
\emph{Gentle algebras arising from surfaces with orbifold points,
Part II: Locally free Caldero--Chapoton functions},
Adv. Math. \textbf{492} (2026), article
\href{https://doi.org/10.1016/j.aim.2026.110905}{110905}.
\href{https://arxiv.org/abs/2309.16061v3}{arXiv:2309.16061v3}.

\bibitem[LZ]{LZ}
D.~Labardini-Fragoso and A.~Zelevinsky, \emph{Strongly primitive species with potentials I: Mutations}, Bol. Soc. Mat. Mex. \textbf{22} (2016), 47--115. \href{https://arxiv.org/abs/1306.3495}{arXiv:1306.3495}.

\bibitem[Lad]{LadkaniClosed}
S.~Ladkani, \emph{On Jacobian algebras from closed surfaces},
preprint (2012).
\href{https://arxiv.org/abs/1207.3778}{arXiv:1207.3778}.

\bibitem[LA]{LA}
D.~L\'opez-Aguayo, \emph{A note on species realizations and nondegeneracy of potentials}, J. Algebra Appl. \textbf{18} (2019), 1950024. \href{https://arxiv.org/abs/1703.09413}{arXiv:1703.09413}.


\bibitem[MouS]{MouScattering}
L.~Mou, \emph{Scattering diagrams of quivers with potentials and mutations},
\href{https://arxiv.org/abs/1910.13714v1}{arXiv:1910.13714v1} (2019).

\bibitem[Mou]{MouReflections}
L.~Mou, \emph{Locally free Caldero--Chapoton functions via reflections},
Math. Z. \textbf{307} (2024), article
\href{https://doi.org/10.1007/s00209-024-03483-y}{11}.
\href{https://arxiv.org/abs/2206.02289v2}{arXiv:2206.02289v2}.

\bibitem[MS]{MouSu}
L.~Mou and X.~Su,
\emph{Generic bases of skew-symmetrizable affine type cluster algebras},
Selecta Math. (N.S.) \textbf{32} (2026), article
\href{https://doi.org/10.1007/s00029-026-01159-5}{55}.
\href{https://arxiv.org/abs/2409.03954v2}{arXiv:2409.03954v2}.

\bibitem[Ngu]{Nguefack}
B.~Nguefack,
\emph{Potentials and Jacobian algebras for tensor algebras of bimodules},
\href{https://arxiv.org/abs/1004.2213v3}{arXiv:1004.2213v3} (2012).

\bibitem[Nic]{Nicaise}
J.~Nicaise, \emph{A trace formula for varieties over a discretely valued
field}, J. Reine Angew. Math. \textbf{650} (2011), 193--238.
\href{https://arxiv.org/abs/0805.1323}{arXiv:0805.1323}.

\bibitem[Palu]{Palu}
Y.~Palu, \emph{Cluster characters for triangulated $2$-Calabi--Yau categories},
Ann. Inst. Fourier (Grenoble) \textbf{58} (2008), 2221--2248.
\href{https://arxiv.org/abs/math/0703540v3}{arXiv:math/0703540v3}.

\bibitem[Pla]{Plamondon}
P.-G.~Plamondon, \emph{Generic bases for cluster algebras from the cluster
category}, Int. Math. Res. Not. IMRN (2013), no.~10, 2368--2420.
\href{https://arxiv.org/abs/1111.4431v2}{arXiv:1111.4431v2}.

\bibitem[Qin]{Qin}
F.~Qin, \emph{Bases for upper cluster algebras and tropical points},
J. Eur. Math. Soc. \textbf{26} (2024), 1255--1312.
\href{https://arxiv.org/abs/1902.09507v4}{arXiv:1902.09507v4}.

\bibitem[Rup]{Rupel}
D.~Rupel, \emph{Quantum cluster characters for valued quivers},
Trans. Amer. Math. Soc. \textbf{367} (2015), 7061--7102. 
\href{https://arxiv.org/abs/1109.6694v2}{arXiv:1109.6694v2}.

\bibitem[SW]{ShenWeng}
L.~Shen and D.~Weng, \emph{Cluster structures on double Bott--Samelson cells},
Forum Math. Sigma \textbf{9} (2021),
\href{https://doi.org/10.1017/fms.2021.59}{e66}.
\href{https://arxiv.org/abs/1904.07992}{arXiv:1904.07992}.

\bibitem[SZ]{SZ}
P.~Sherman and A.~Zelevinsky, \emph{Positivity and canonical bases in
rank 2 cluster algebras of finite and affine types}, Moscow Math. J.
\textbf{4} (2004), no.~4, 947--974.
\href{https://arxiv.org/abs/math/0307082}{arXiv:math/0307082}.

\bibitem[Stacks]{Stacks}
The Stacks Project Authors, \emph{The Stacks Project},
\href{https://stacks.math.columbia.edu}{stacks.math.columbia.edu};
the individual tags are specified where used.

\bibitem[Wu]{Wu}
Y.~Wu, \emph{Group actions on relative cluster categories and Higgs categories},
Nagoya Math. J. \textbf{261} (2026), article
\href{https://doi.org/10.1017/nmj.2026.10112}{e32}.
\href{https://arxiv.org/abs/2502.10670v3}{arXiv:2502.10670v3}.

\end{thebibliography}
\end{document}